\documentclass[11pt,letterpaper]{article}

\usepackage[affil-it]{authblk}
\usepackage[utf8]{inputenc}
\usepackage{amsmath}
\usepackage{amsfonts}
\usepackage{amssymb}
\usepackage{amsthm}
\usepackage{graphicx}
\usepackage{subcaption} 
\usepackage[left=2cm,right=2cm,top=2cm,bottom=2cm]{geometry}
\usepackage{enumerate}
\usepackage{appendix}
\usepackage{mathtools}
\usepackage{xcolor}
\usepackage{hyperref}
\usepackage{todonotes}

\newtheorem{theorem}{Theorem}[section]
\newtheorem{lemma}[theorem]{Lemma}
\newtheorem{prop}[theorem]{Proposition}

\theoremstyle{definition}
\newtheorem{definition}[theorem]{Definition}

\theoremstyle{remark}
\newtheorem{remark}[theorem]{Remark}

\DeclareMathOperator{\dv}{div}

\DeclareMathOperator{\tr}{tr}

\newcommand{\Dd}{\Delta}

\usepackage[maxbibnames=99,backend=bibtex,style=numeric]{biblatex}

\numberwithin{equation}{section}

\date{}
\begin{document}

\title{\bf Temperature dependent extensions of the Cahn-Hilliard equation} 

%

\author{Francesco De Anna$\,^1$, Chun Liu$\,^2$, Anja Schlömerkemper$\,^3$, Jan-Eric Sulzbach$\,^4$}
\affil{\small
    $\,^1$ Institute of Mathematics, University of Würzburg, Germany\\
    email: francesco.deanna@mathematik.uni-wuerzburg.de\\
    
    \vspace{0.3cm}
    $\,^2$ Illinois Institute of Technology, Chicago, US\\
    email: cliu124@iit.edu\\
    
     \vspace{0.3cm}
     $\,^3$ Institute of Mathematics, University of Würzburg, Germany\\
    email: anja.schloemerkemper@mathematik.uni-wuerzburg.de\\
    
     \vspace{0.3cm}
    $\,^4$ Illinois Institute of Technology, Chicago, US\\
    email: jsulzbach@hawk.iit.edu
}

\maketitle

\begin{abstract}

\noindent
The Cahn–Hilliard equation is a fundamental model that describes phase separation processes of binary mixtures.
In this paper we focus on the dynamics of these binary media, when the underlying temperature is not constant.

\noindent
The aim of this paper is twofold. We first derive two distinct models that extend the classical Cahn-Hilliard equation with an evolutionary equation for the absolute temperature.
Secondly, we analyse the local well-posedness of classical solution for one of these systems.

\noindent
Our modelling introduces the systems of PDEs by means of a general and unified formalism.
This formalism couples standard principles of mechanics together with the main laws of thermodynamics.
Our work highlights how certain assumptions on the transport of the temperature effect the overall physics of the systems.

\noindent
The variety of these thermodynamically consistent models opens the question of which one should be more appropriate.
Our analysis shows that one of the derived models might be more desirable to the well-posedness theory of classical solutions, a property that might be natural as a selection criteria.

\noindent
We conclude our paper with an overview and comparison of our modelling formalism with some equations, which were previously derived in literature.

    %
    %
    %
\end{abstract}

\section{Introduction}
During the last decades, the mathematical theory of phase separation has gained a fundamental role in understanding the behaviour of binary alloys, mixtures of materials and biological phenomena.
Phase separations are processes that arise in many scientific, engineering and industrial applications, such as phase-field crystals \cite{MR3778460}, coarsening in materials \cite{SunAndrewsThronton} and segregation-like phenomena \cite{BoettingerWarrenBeckermannKarma}. 

\smallskip
\noindent
This paper addresses a diffuse interface theory for the dynamics of binary materials whose temperature is not constant.
 Multi-phase flows are indeed commonly addressed by two major conventional theories, the sharp-interface model and the diffuse-interface model. The sharp-interface model treats separately the hydrodynamics of a binary system whose phase configurations are segregated by a free hyper-surface. The diffuse-interface model, on the other hand, recasts the surface description of the interface by introducing a thin layer, a region in which mixture of the pure phases can occur. Although the sharp interface provides a reliable model for pure segregation, it breaks down when the hyper-surface of separation is effected by topological changes, such as the merging and splitting \cite{MR1650795}. These phenomena are however well captured by the diffuse interface theory \cite{MR3745174, MR2358438,MR2260713,MR4285007}, which we address in this paper.

\smallskip
\noindent
 We consider materials whose temperature varies with time. We are motivated by certain type of media and biological phenomena, in which the binary structure of the phase field for the compound shall be further supplemented by an understanding of the thermal configuration for the overall material. Thermal effects indeed play often a major role in the dynamics. A typical example is provided by dendrites, i.e. protoplasmic extensions of nerve cells. The membranes of these cells are surfaces that grow in a liquid and evolve in relation to the thermodiffusivity and the mean curvature of the membranes themselves \cite{Dendrite}.

\smallskip
\noindent
This paper extends some classical formalism in fluid-mechanics. To this end, we supplement the equation of the order parameter with an equation for the evolution of the temperature (cf. Section \ref{intro:sec-goal-of-this-paper} and Section \ref{sec:modelling}). We show that certain hypotheses on the behavior of the temperature can lead to different models, which are thermodynamically consistent.
In particular, we show that different transport properties of the temperature and relations of the free energy lead to different models. Furthermore, our current analytical techniques can deal with a specific case of these models (cf.~Theorem \ref{thm:intro-thm2} and Section \ref{sec:modelling}), therefore our analysis make this model more desirable at the current stage.


\smallskip
\noindent
To introduce the main ideas and concepts of our results, we shall first recall some basics about the considered model when the temperature is assumed constant (isothermal process).
In this case, the classical Cahn–Hilliard equation can be written in the following (dimensionless) form \cite{Cahn59,Cahn58,CahnH59}: 
\begin{equation}\label{introduction:isotehrmal-cahn-hilliard}
    \partial_t \phi  = \Delta  \mu, \qquad \mu =  -\Delta \phi + W'(\phi),
\end{equation}
on a domain $(t,x) \in (0,T)\times \Omega$ with $\Omega $ being a subset of $\mathbb{R}^2$ or $\mathbb{R}^3$ (in our analysis in Section \ref{sec:well-posedness}, we eventually consider $\Omega = \mathbb{R}^d$, $d\geq 1$). Here $ \phi = \phi(t,x)$ stands for the phase field variable (also known as order parameter), modelling the local concentration of two components for the binary mixture. The value of $\phi$ ranges between $-1$ and $1$, two values that shall describe the pure states of the material. The generalised chemical potential is denoted by $\mu=\mu(\phi)$, while $W=W(\phi)$ stands for the bulk energy density of the system. Commonly the energy density $W$ owns a double well structure and attains minima in the pure phases $\phi = \pm 1$.  

\noindent
A typical thermodynamically relevant form for the density $W$ is a double-well-structured logarithmic potential \cite{MR2862028, MR4278119,  MR3986249, MR4164495}. However, it is rather common \cite{MR4000136,MR2585560,Marveggio21,Miranville,MR2151731,MR2260713} to relax the singular form of the logarithmic energy, through a standard double well-potential (cf.~Section \ref{sec:thermodynamic-relations} for further details)
\begin{equation*}
    W(\phi) = \frac{1}{4\varepsilon}\big(\phi^2-1\big)^2,
\end{equation*}
where $\varepsilon>0$ denotes the thickness of the interface.

\noindent 
When $\Omega$ is bounded and has a regular boundary, then one shall supplement equation \eqref{introduction:isotehrmal-cahn-hilliard} with the following Neumann boundary conditions:
\begin{equation*}
    \partial_n \phi = 0 ,\qquad \partial_n \mu = 0 \qquad\text{in} \quad  (0,T)\times \partial \Omega,
\end{equation*}
where $n$ denotes the outer normal unit vector to the boundary $\partial \Omega$. The first condition ensures that the diffuse interface is orthogonal to the boundary of the domain, while the second is a non-flux condition, which allows the conservation of the phase field on the domain.

\smallskip
\noindent 
The modelling and mathematical treatment of the Cahn-Hilliard system has been extensively addressed in the literature during the past decades.
 For further information on the isothermal setting, we mention for instance  \cite{Abels12,Miranville,NOVICKCOHEN2008201,taylor1994linking}, and the references therein. 

\smallskip
\noindent
Several attempts and developments have been introduced in the last decades, in order to couple the Cahn-Hilliard equations \eqref{introduction:isotehrmal-cahn-hilliard} with an equation for the absolute temperature $\theta = \theta(t,x)>0$, which we present in the next section.

\subsection{A brief overview of related literature}

%


\smallskip
\noindent
The first work incorporating the effect of temperature in the theory of phase separation is due to Caginalp (cf.~\cite{MR998924, Caginalp90, MR1643668}). In its dimensionless form, the model proposed by Caginalp couples the Cahn-Hilliard \eqref{introduction:isotehrmal-cahn-hilliard} with a heat equation for the temperature. Caginalp was in particular motivated by several applications, such as the Stefan problem for the evolution of a solid-liquid interface (for instance ice and water) and Hele-Shaw type flow between two fluids with different dynamic viscosities (for instance air and oil). 
Although the system was not originally related to the law of thermodynamics, the equations of Caginalp have been extensively investigated in literature, both from an analytical point of view \cite{MR2911120,MR2148283,MR4044385,MR3003971}, as well as from numerical simulations \cite{MR3036129, MR3985379}. A variation of this model has been used to study dynamical undercooling at the liquid-solid interface, as well as asymptotically relating to the free boundary model (sharp interface) \cite{MR2424950}.

\smallskip
\noindent 
A further development for the thermodynamics of phase-field models was proposed by Alt and Pawlow in \cite{alt1992mathematical}. Their modelling approach is based on the Ginzburg-Landau theory of phase separation combined with the non-equilibrium thermodynamics developed in \cite{Groot}. They proposed a family of system of PDEs for isotropic materials for a model of phase separation during continuous cooling.
One of their major interest was to ensure the consistency of the equations with the main laws of thermodynamics. In Section \ref{sec:conclusion} we recall their equations (cf.~\eqref{eq:alt-pawlow}) and provide a comparison with the modelling of our paper.


\smallskip
\noindent 
Along a similar approach, Miranville and Schimperna \cite{MR2151731} made use of a theory on micro-force balance to extend the models derived by Alt and Pawlow. The most compelling reason for their modelling was related to non-isotropic materials, whose state is close to thermodynamic equilibrium. 
A number of recent works have been dedicated to the analysis of these equations (cf. \cite{MR4000136,MR2585560,Marveggio21}).
In particular, Marveggio and Schimperna \cite{Marveggio21} analytically treated a sub-case of the system in \cite{MR2151731}, where under suitable conditions, the authors show the existence of weak solutions and entropy solutions of the problem.
In our conclusion of Section \ref{sec:conclusion}, we show how their system can be indeed determined by our modelling technique (cf.~\eqref{Miranville-Schimperna-eq}).

\noindent
For further information on the isothermal setting, we mention for instance  \cite{Abels12,Miranville,NOVICKCOHEN2008201,taylor1994linking}, and the references therein. 

\subsection{Goal of this paper and main results}\label{intro:sec-goal-of-this-paper}

The temperature $\theta = \theta(t,x)$ is one of the main characters of this paper. This  thermodynamic state variable changes the overall structure of our modelling technique. In particular, certain assumptions on the temperature lead to different type of equations. The aim of this paper is therefore two-fold: 
\begin{itemize}
    \item Elucidate these assumptions on the temperature
    in a formal way, making use of standard techniques of mechanics and thermodynamics, to derive several models that are consistent with the main laws of thermodynamics (cf.~Section \ref{sec:modelling}).  
    \item Show that one of these models is more amenable at the level of analysis (cf.~Section \ref{sec:well-posedness}), providing a well-posedness theory of local-in-time classical solutions.
\end{itemize}

\smallskip\noindent 
The first goal is achieved making use of a well-known formalism, based on techniques of classical mechanics and non-equilibrium thermodynamics, such as the free energy of the media (cf.~Section \ref{sec:thermodynamic-relations}), the kinematic and transport of the state variables (cf.~Section \ref{sec:kinematics}), the least action principle (cf.~Section \ref{sec:least-action-principle}), the rate of dissipation and the maximum dissipation principle (cf.~Section \ref{sec:onsager}), and finally the first and second law of thermodynamics (cf.~Section \ref{sec:derivation-temperature-eq}). 

\noindent 
The major novelties of our modelling concern the least action principle coupled with the transport of the temperature, which provide the conservative forces of the motion  (cf.~Theorem \ref{thm:conservative-forces}), by the variation related to the following free energy density
\begin{equation*}
    \psi(\phi, \nabla \phi,\theta) = 
    \epsilon(\theta) 
    \frac{|\nabla \phi |^2}{2} + 
    \frac{1}{\epsilon(\theta)}
    \bigg(
        \frac{(\phi^2-1)^2}{4} + c(\theta)\phi
    \bigg) - k_B \theta \ln \theta.
\end{equation*}
Details about this energy are provided in Section \ref{sec:thermodynamic-relations}.
Here we highlight how certain assumptions on the temperature come into play and change the overall structure of the final system. The variation of the action is indeed performed along the flow generated by the effective microscopic velocity of the phase field $\phi$. More precisely, since $\phi$ is a conserved quantity, the Cahn-Hilliard equation can always be recasted as a continuity equation with respect to a velocity field $u(\phi, \theta)$ (cf.~Section \ref{sec:kinematics}). 
A natural question relates therefore on how the velocity $u$ and its flow effect the temperature. In this article we investigate two main possibilities:
\begin{itemize}
    \item[(A1)] The temperature is transported by the flow, when performing the variation of the action. Such assumption is adopted and motivated by earlier applications, as in ideal gases (cf.~\cite{LS21}).
    
    \item[(A2)] The temperature is not effected by the velocity $u$ and is determined by the background, during the variation of the action. Here $u$ represents uniquely the dynamic of the mixture of the two pure states $\phi = \pm 1$, without giving information on the temperature evolution.
\end{itemize}
This article formally shows that each of these assumptions leads to a different set of equations. Eventually, in our second result, we show that our second assumption (A2) leads to a system of PDEs, which is more treatable in terms of existence and uniqueness of local-in-time classical solutions (cf.~Theorem \ref{thm:intro-thm2}).
From a modelling point of view both assumptions (A1) and (A2) result in models consistent with the laws of thermodynamics. However, our analysis suggests that assumption (A2) is more favorable.

\smallskip
\noindent 
Our first result about our modelling can be summarised as follows.
\begin{theorem}\label{thm:intro-theorem1}
Assume that assumption (A1) holds true. Then the first and second laws of thermodynamics lead to the following extension of the Cahn-Hilliard equation:
\begin{equation*}
    \begin{cases}
        \partial_t \phi - \alpha \Delta \partial_t \phi = \Delta \mu(\phi,\theta) + \dv\bigg({\displaystyle \frac{s(\phi,\theta)\nabla \theta}{\phi}}\bigg)\\
        \theta  \partial_t s(\phi, \theta) + 
        \theta \dv(u(\phi, \theta) s(\phi, \theta)) = 
        -\theta \dv \Big({\displaystyle\frac{ {\bf q}(\theta)}{\theta}}\Big) + \theta \Delta^*(\phi, \theta)
    \end{cases}
\end{equation*}
where the chemical potential $\mu$, the local entropy $s$, the velocity $u$, the heat flux $\mathbf{q}$ and the entropy production $\theta \Delta^*$ are given by
\begin{alignat*}{8}
    \mu = \frac{\partial \psi}{\partial \phi} 
                            - \dv\Big(\frac{\partial \psi }{\partial \nabla \phi}\Big), \quad 
    s  = -\frac{\partial \psi}{\partial \theta}(\phi, \theta), \quad 
    u = -\frac{\nabla \mu}{\phi} - 
                \frac{s\nabla\theta}{\phi^2} - \alpha \frac{\nabla \partial_t \phi}{\phi},\quad 
    \mathbf{q}=\mathbf{q}(\theta) = -\kappa(\theta) 
    \nabla \theta, 
\end{alignat*}
while the entropy production $\theta \Delta^*$ satisfies
\begin{equation*}
    \theta \Delta^*(\phi, \theta) = 
                \phi^2|u|^2+\alpha |\partial_t\phi|^2 +\kappa\frac{|\nabla \theta|^2}{\theta}.
\end{equation*}
Assume that assumption (A2) holds true. Then the main laws of thermodynamics lead to the following extension of the Cahn-Hilliard equation:
\begin{equation*}
    \begin{cases}
        \partial_t \phi - \alpha \Delta \partial_t \phi = \Delta \mu(\phi,\theta)\\
        \theta  \partial_t s(\phi, \theta) = -
        \theta \dv 
        \Big(
             {\displaystyle
                \frac{ \mathbf{q}(\theta)}{\theta}
             }
        \Big) + \theta \Delta^*(\phi, \theta)
    \end{cases}
\end{equation*}
where $\mu$, $s$ and $\mathbf{q}$ maintain the same structure as above, while the entropy production $\theta \Delta^*$ is given by
\begin{equation*}
    \theta \Delta^*(\phi,\theta):= \big| \nabla \mu  + \alpha \nabla \partial_t \phi \big|^2+\alpha |\partial_t\phi|^2+\kappa\frac{|\nabla \theta|^2}{\theta}.
\end{equation*}
\end{theorem}
\begin{remark}
    Some remarks of this result are in order. The positive parameter $\alpha>0$ provides an artificial dissipative term, that we introduce to support our analysis in Theorem \ref{thm:intro-thm2}. When $\alpha = 0$ we recover the original structure of the Cahn-Hilliard model in both (A1) and (A2). 
    
    \noindent
    Although the equation for the temperature is written in terms of the local entropy $s(\phi, \theta)$, one can further develop it to recover a form which depends only on $\theta$ and $\phi$. We provide this expansion in system \eqref{eq:phi-to-analyse}--\eqref{eq:phi-to-analyse} for the second model. 
    
    \smallskip
    \noindent 
    Both in (A1) and in (A2), the temperature equation is written in terms of the second law of thermodynamics on the local entropy $s = s(\phi, \theta)$ (i.e the Clausius-Duhem inequality). An equivalent form of this equation can be written in terms of the internal energy $e = e(\phi, \theta) = \psi(\phi, \theta) + \theta s(\phi, \theta)$ of the system:
    \begin{equation*}
        \begin{alignedat}{4}
            &(A1) \quad \partial_t e  + \dv \big(u\phi \mu \big) = \dv \big(\partial_{\nabla \phi}\psi \partial_t \phi + \theta u \partial_\theta \psi \big) - \dv{\bf q},\qquad &&\text{where}\quad \phi u = -\nabla \mu - 
                \frac{s\nabla\theta}{\phi} - \alpha \nabla \partial_t \phi,\\
            &(A2) \quad \partial_t e  + \dv \big(u\phi \mu \big) = \dv \big(\partial_{\nabla \phi}\psi \partial_t \phi \big) - \dv{\bf q},\qquad &&\text{where}\quad \phi u = -\nabla \mu - \alpha\nabla \partial_t \phi.
        \end{alignedat}
    \end{equation*}
\end{remark}

\medskip
\noindent 
Once concluded the modelling of Theorem \ref{thm:intro-theorem1}, we pass to investigate the well-posedness problem related to our derived systems. The equations in (A1) create several challenges in the analysis, which are mainly due to some nonlinearities on the $\phi$-equation, that depend on the gradient of the temperature, such as $\dv(s \nabla \theta/\phi)$. These terms, indeed, can not be generally defined in a classical sense, when localised in the region of mixing $\phi = 0$ (cf.~remark \ref{rmk:problem-with-A1}).

\noindent
On the other hand, these terms vanish under the assumptions of (A2), a fact which make the overall system more mathematically treatable. This principle is stated on our second main result. It asserts the existence and uniqueness of local-in-time classical solutions for the equations of (A2), under suitable condition for the initial data.

\smallskip
\noindent
The functional framework is that of Besov regularities $B^s = B^s_{2,1}$, which are defined on the whole space $\mathbb{R}^d$ (cf.~Section~\ref{section:Besov-spaces} for a detailed definition of these function spaces). We choose this setting instead of a Sobolev one, since it encompasses suitable a-priori estimates (cf.~Lemma \ref{lemma:a-priori-estimates-for-phi}, Lemma \ref{lemma:a-priori-estimates-for-theta} and Lemma \ref{lemma:action_of_smooth_function_on_B}), allowing to tackle the majority of the nonlinearities of our system.

\begin{theorem}\label{thm:intro-thm2}
Let $\phi_0\in B^{d/2+2}$, $\theta_0 - \bar{\theta} \in B^{d/2}$ for a positive constant temperature $\bar{\theta}$ and let $\alpha>0$, $\kappa>0$ and $\kappa_B>0$ be parameters. Then there exists a small positive constant $\varepsilon_0\in [0,1)$,  which depends only on the dimension $d$, such that for any $\epsilon,\,\bar{\theta},\,\phi_0,\,\theta_0$ satisfying the smallness conditions
\begin{equation}\label{eq:intro-smallness-condition}
\begin{aligned}
    \epsilon \| \Delta \phi_0 \|_{B^{\frac{d}{2}}}  
    &+ \frac{1}{\bar{\theta} }
    \max\Big\{1, \frac{1}{\epsilon}\Big\} (1+\| \phi_0 \|_{B^{\frac{d}{2}}})^4
    < \varepsilon_0
    \min\big\{1, \alpha,\kappa, \kappa_B \big\},\\
       \| \theta_0 - \bar{\theta}\|_{B^\frac{d}{2}} 
    &
    <
    \left(
    \min 
    \Big\{
        \epsilon^2,\, 
        \frac{1 }{1+\epsilon}
    \Big\}
    \frac{
         \varepsilon_0\min\{1,\alpha,\kappa, \kappa_B\}
    }{\bar{\theta}(1+\alpha)}
    \right)^2,
\end{aligned}
\end{equation}
the system \eqref{eq:phi-to-analyse}--\eqref{eq:theta-to-analyse} admits a unique classical solution $(\phi,\theta)$ on $(0,T)\times \mathbb{R}^d$, for a sufficiently small time $T>0$. Furthermore, the solution $(\phi,\theta)$ belongs to the function space
\begin{alignat*}{8}
    \phi &\in \mathcal{C}([0,T],B^{\frac{d}{2}+2})\cap 
    L^1(0,T;B^{\frac{d}{2}+4}),\qquad 
    &&\partial_t \phi &&&&\in \tilde L^2(0,T; B^{d/2+1}),
    \\
    \theta &\in 
    \mathcal{C}([0,T],B^{\frac{d}{2}})\cap 
    L^1(0,T;B^{\frac{d}{2}+2}),
    \qquad 
    &&\partial_t \theta &&&&\in \tilde L^2(0,T; B^{d/2+1}).
\end{alignat*}
\end{theorem}
\noindent
The function spaces $L^2(0,T; B^{s})$ with $s\in \mathbb{R}$ is of Chemin-Lerner space-time type (cf.~Section \ref{section:Besov-spaces}), which is slightly larger than the common setting of functions in $L^2(0,T; B^{s})$.
\noindent 
The statement introduces a critical constant temperature $\bar{\theta}>0$, which plays the role of a threshold, where a transition between a two-phase configuration and an homogeneous state occurs. More precisely, our energy density (cf. \eqref{eq: free energy Cahn Hilliard} with \eqref{eq: epsilon function}) switches from a double well potential structure to a single one. 

\begin{remark}\label{rmk:smallness-condition}
   Let us comment on relation  \eqref{eq:intro-smallness-condition}. From a first reading, one might understand that the smallness condition also effects the initial phase field $\phi_0$. Such a relation would be unnatural, since this state variable can vary between the extreme values -1 and 1, therefore $\phi_0$ could not be small in $B^{d/2+2}$ (which is embedded in $L^\infty(\mathbb{R}^d))$. One should instead identify \eqref{eq:intro-smallness-condition} only in terms of $\epsilon, \bar{\theta}$ and $\theta_0$. More precisely, for any phase field $\phi_0 \in B^{\frac{d}{2}+2}$ we can select a sufficiently small $\epsilon>0$ and an initial temperature $\theta_0$, which is sufficiently close to a high constant $\bar{\theta}$, such that \eqref{eq:intro-smallness-condition} is satisfied and our model admits a local-in-time classical solution. 
\end{remark}

\subsection{Plan of the paper}
The remaining part of this paper is organized as follows. Section \ref{sec:modelling} is devoted to the proof of Theorem \ref{thm:intro-theorem1} and we derive here two distinct non-isothermal Cahn-Hilliard models, depending on the assumptions (A1) and (A2) on the temperature respectively. In Section \ref{sec:well-posedness} we prove Theorem \ref{thm:intro-thm2} for the existence and uniqueness of classical solutions for the model of (A2). 
In Section \ref{section:Besov-spaces} we present some basic inequalities and useful tools related to our analysis. Finally, Section \ref{sec:conclusion} is devoted to our conclusions, where we further provide a comparison of our models with respect to others available in literature.
 
\section{A general frame to derive a non-isothermal Cahn-Hilliard}\label{sec:modelling}
This section aims to prove Theorem \ref{thm:intro-theorem1}, in which we extend the Cahn-Hilliard model to a system of PDEs, which take into account the evolution of the temperature, respectively. Any suitable extended model shall be thermodynamically consistent, in the sense that it shall satisfy the conservation of the total energy (first law of thermodynamics) and the increase of the entropy (second law of thermodynamics). Details about these relations will be specified below (cf. Section \ref{sec:derivation-temperature-eq}). 

\subsection{Some thermodynamic relations}\label{sec:thermodynamic-relations}

Before entering the mathematical details, we shall first recall some physical bases around the modelling of the media. For  constant 
temperature, the free energy of Cahn-Hilliard coincides with its internal energy up to a constant. Its formulation is  motivated  by  the  Ginzburg-Landau theory  for phase transitions, being a competition between the elastic and bulk contributions of the media (hydrophobic and hydrophilic effects between different species):
\begin{align*}
\psi(\phi,\nabla \phi):= \frac{\tilde \epsilon}{2}|\nabla \phi|^2+\frac{1}{\tilde \epsilon}W(\phi),
\end{align*}
where it is common to consider $W(\phi)=(\phi^2-1)^2/4$ as a double-well potential, whose minima are attained in proximity of pure phase configurations $\phi =\pm 1$. While the bulk energy reflects the interaction of different volume fractions of individual species, the gradient part determines a regularization (relaxation). The regions of species are separated by a small interface, whose thickness is proportional to the parameter $\tilde \epsilon$.

\noindent Under the assumption of a non-constant temperature, the definition of the free energy density shall change in accordance. In particular, it is rather natural to expect that for high temperatures the mixing term $|\nabla \phi|^2$ should dominate the free energy, whereas for low temperatures the separation term $W(\phi)$ dominates the overall dynamics. Motivated by these remarks, we extend our definition of the free energy $\psi$, which now depends also on the absolute temperature $\theta>0$:
\begin{align}\label{eq: free energy Cahn Hilliard}
\psi(\phi,\nabla \phi,\theta)=\frac{\tilde \epsilon(\theta)}{2}|\nabla \phi|^2+
\frac{1}{\tilde \epsilon(\theta)}\Big( \underbrace{W(\phi)+c(\theta)\phi^2}_{=:W(\phi,\theta)}\Big)-k_B \theta \ln \theta.
\end{align}
\begin{remark}
Some remarks are in order to highlight the effects of the temperature $\theta$ on this free energy. 
\begin{itemize}
    \item The thickness of the diffuse interface varies depending on the temperature, a fact which is grounded in $\tilde \epsilon(\theta)$.
    Here the function $\tilde \epsilon(\theta)$ is a general continuous and increasing function on $\theta$, whose explicit formulation depends on the material.
    
    \item Assuming that the materials turn eventually into a one-phase media at high temperatures, it is plausible to expect a transition of the double-well potential $W(\phi)$ to a rather single one. We therefore recast the double well potential $W(\phi)$ through a function $c(\theta)\phi^2$, which allows the transition from two-phase configurations to a homogeneous state at a critical temperature $\bar{\theta}>0$. 
    
    \item The third contribution $k_B\theta\ln \theta$ in \eqref{eq: free energy Cahn Hilliard} depends only on the (absolute) temperature and on the Boltzmann constant $k_B$. This term is common in thermodynamics and is typical of the dynamics of ideal gas. One shall remark that for isothermal process, it is constant and therefore does not effect the overall dynamics.
\end{itemize}
\end{remark}
\noindent 
We next state a plausible form of the parameters $\tilde \epsilon(\theta)$ and $c(\theta)$ in \eqref{eq: free energy Cahn Hilliard}, which we take into account when mathematically treating our final model in Section \ref{sec:well-posedness}. In particular we consider
\begin{equation}\label{eq: epsilon function}
    \tilde \epsilon(\theta)=\epsilon \theta,\qquad 
    c(\theta)=\frac{1}{3}(\theta-\bar{\theta})^3.
\end{equation}
where $\bar{\theta}>0$ is the temperature of phase transition and the constant  $\epsilon>0$ is specific to the materials. The choices in \eqref{eq: epsilon function} shall be understood as Taylor approximations of more general functions around the constant temperature $\bar{\theta}$.

\smallskip
\noindent
The free energy density being introduced, we next recall certain common thermodynamic quantities, i.e. the chemical potential $\mu$ and the local entropy $s$. These quantities are defined through the relations
\begin{equation}\label{def:chemicalpotential-entropy}
\begin{aligned}
    \mu &= \mu(\phi, \nabla \phi, \theta) := \frac{\delta \psi}{\delta \phi} = 
    \frac{\partial \psi}{\partial \phi} - \dv \frac{\partial  \psi}{\nabla \phi} 
    = \frac{1}{\tilde \epsilon(\theta)}
    \Big( W'(\phi) - 2c(\theta)\phi \Big) - \dv \big(\tilde  \epsilon(\theta) \nabla \phi \big),\\
    s&= 
     s(\phi,\nabla\phi,\theta):=-\frac{\partial \psi}{\partial \theta}  
    = -\frac{\tilde \epsilon'(\theta)}{2}|\nabla \phi|^2+\frac{\tilde \epsilon'(\theta)}{\tilde \epsilon(\theta)^2} W(\phi) -\frac{1}{\tilde \epsilon(\theta)}\frac{\partial W}{\partial\theta}(\phi,\theta)+k_B(1+\ln\theta).
\end{aligned}
\end{equation}
The chemical potential $\mu$ reflects the energy that can be absorbed or released during a phase transition, while the entropy is a quantity representing the irreversibility of the process.

\begin{remark}\label{rmk: convexity of free }
We shall point out that the system in thermodynamical equilibrium could be completely described by the basic thermodynamics variables $(\phi,\,s)$, through the thermal equation of state $\theta = \theta(\phi,s)$ instead of $(\phi, \theta)$. This identity is indeed motivated by the convex structure of the free energy $\psi$ with respect to $\theta$, which would allow to apply the implicit function theorem, in order to determine the value of $\theta(\phi,s)$. However, in this paper, we consider non-equilibrium processes, which are mainly described by $(\phi, \theta)$.
\end{remark}

\subsection{Kinematics of the phase field and transport of the temperature}\label{sec:kinematics}

In continuum mechanics the equations of motion can be interpreted as a balance of conservative (or generalised) forces and dissipative (or frictional) forces. When all active forces are conservative, the Hamilton's principle states that the evolution of a system is characterised by the variation of an action. 
In this section we introduce some assumptions related to the material, which allows to define and take the variation of an action related to the energy density \eqref{eq: free energy Cahn Hilliard}. More precisely, we introduce the kinematics of the phase field $\phi$ and state several hypotheses on how this kinematics could effect the temperature $\theta$ (when dealing with purely conservative forces). These hypotheses are treated separately in the next sections, since they lead to different non-isothermal models.

\smallskip
\noindent
Assume that the media occupies a domain $\Omega\subseteq \mathbb{R}^d$, for $d=2,3$, with or without a boundary $\partial \Omega$. We first recall that that the dynamics given by the Cahn-Hilliard equation preserves the overall volume fraction of the two species. Therefore the phase field function $\phi$ is a conserved quantity in the domain $\Omega\subset \mathbb{R}^n$, a fact that can be formalised through the continuity equation
\begin{align}\label{kinematics}
\partial_t\phi +\dv(\phi\, u)=0~~ (x,t)\in \Omega\times(0,T),\qquad 
\phi_{|t=0}=\phi_0
\qquad u\cdot n=0 ~~(x,t)\in \partial \Omega\times(0,T),
\end{align}
supplemented by an initial data and a no-flux boundary condition. Equation \eqref{kinematics} identifies what is known as the kinematics of the state variable $\phi$. Here $u:\Omega\times (0,t)\to \mathbb{R}^d$ stands for the effective  velocity. For the isothermal Chan-Hillard case, $u$ is commonly expressed in terms of the chemical potential and the phase field $\phi$. We derive a similar result for the non-isothermal case in the forthcoming sections (cf. Theorem \ref{thm:eq-for-phi}), however one should consider $u$ as undetermined at this stage of our modelling. Nevertheless, we assume that $u$ has sufficient regularity to generate a smooth flow map
\begin{equation}\label{flow-map}
        \partial_t x(X,t) = u(x(X,t),t),\quad (X,t) \in \Omega \times (0,T),\qquad 
        x(X,0) = X\quad X\in \Omega,
\end{equation}
where throughout the paper we denote by $x\in \Omega$ the Eulerian coordinates and by $X\in \Omega$ the Lagrangian coordinates. Introducing the deformation tensor $F(X,t) = \nabla_X x(t,X) = (\partial x_i/\partial X_j)_{i,j=1,\dots,d}$, the exact solution of \eqref{kinematics} for $\phi$ is given by $\phi(t,x(X,t)) = \phi_0(X)/\det F(X,t)$.

\smallskip
\noindent 
Motivated by this ansatz, we shall now ask how the flow map $x(X,t)$ interacts with the absolute temperature $\theta$.  We state two main possible situations, from which we eventually derive two distinct systems.
\begin{itemize}
    \item[(A1)] We assume that the temperature moves on the particle scale, i.e. $\theta$ is transported along the trajectory of the flow map with velocity $u$; The temperature is here determined by its initial state through $\theta(x(X,t),t) = \theta_0(X,t)$, which is equivalent to the transport equation $\partial_t \theta + u\cdot \nabla \theta = 0$. By duality the entropy $s$ is here conserved (when all forces are conservative and no frictional one is taken into account) through $\partial_t s + {\rm div}(us) = 0$.
    
    \item[(A2)] We assume that the temperature is on a fixed background, i.e. $\theta$ is independent of the flow map with velocity $u$. This relation is equivalent to the equation $\partial_t \theta = 0$, so that $s$ is conserved  by duality (when all forces are conservative) with $\partial_t s = 0$.
\end{itemize}
\begin{remark}
One could introduce further assumptions on how the flow map interacts with the temperature $\theta$ at the level of the action, for instance by interpolation the conditions (A1) and (A2). We claim therefore that there are infinite non-isothermal extensions of the classical Cahn-Hilliard equation, which are consistent with the main laws of thermodynamics. In this paper we derive just two of them, i.e. the ones related to (A1) and (A2). 
\end{remark}

\subsection{Least action principle}\label{sec:least-action-principle}
The least action principle or principle of virtual work unlocks the first contributions to the force balance equation. The principle reveals the conservative forces $ \textbf{f}_{cons}$ of the motion by a variation of an action through the flow map $x(X,t)$ introduced in Section \ref{sec:kinematics}. These forces will then compete with the dissipative ones to recover the equation of motion. 

\smallskip
\noindent 
For a prescribed initial volume $\Omega_0 \subseteq \Omega$, we denote by $\Omega_t^x =\{y\in \Omega\,:\, y = x(X,t)\; \text{for a}\; X\in \Omega_0\}$ its evolution at a specific time $t\in (0,T)$. We hence define our action $A$ by means of the free energy $\psi$ in \eqref{eq: free energy Cahn Hilliard}:
\begin{align*}
A(x(\cdot))=-\int_0^T\int_{\Omega_t^x}\psi(\phi,\nabla\phi,\theta) dxdt
\end{align*}
(formally, one shall define this action in Lagrangian coordinates, as done in \eqref{action-in-lagrangian}).
The volume force $\textbf{f}_{cons}$ is obtained as the negative variational derivative of $A$ with respect to the flow $x$:
\begin{equation}\label{eq:action-fcons}
\delta A(x(\cdot))=\frac{d}{d \bar\varepsilon}\bigg|_{\varepsilon=0} A(x(\cdot) + \bar\varepsilon \delta x(\cdot)) = 
-\int_0^T\int_{\Omega_t^x} \textbf{f}_{cons}\cdot \delta x dxdt,
\end{equation}
where $\delta x$ is any suitable perturbation of the flow map $x$.

\noindent 
The next theorem determines an explicit form of $\textbf{f}_{cons}$, depending on the given assumptions (A1) or (A2) for the temperature $\theta$. In order to state our result, we shall first recall that $\mu$ and $s$ in \eqref{def:chemicalpotential-entropy} stand for the chemical potential and the local entropy respectively.
\begin{theorem}\label{thm:conservative-forces}
    Assume that the velocity $u$ in  \eqref{kinematics} generates an unique smooth flow map $x(X,t)$.
    \begin{enumerate}[(i)]
        \item If the temperature $\theta$ satisfies the assumption (A1) (i.e. is transported by the flow), then
        \begin{equation*}
            \textbf{f}_{cons} = \phi \nabla \mu  + s  \nabla \theta.
        \end{equation*}
        \item If the temperature $\theta$ satisfies the assumption (A2), then:
        \begin{equation*}
            \textbf{f}_{cons} = \phi \nabla \mu.
        \end{equation*}
    \end{enumerate}
\end{theorem}
\begin{remark}
    Before continuing our study with the proof of Theorem \ref{thm:conservative-forces}, some remarks are in order.  
    When assumptions (A2) in $(ii)$ holds, namely when the temperature is not transported by the flow map, the result  is not really surprising, since $\textbf{f}_{cons}$ coincides with the common form $\phi \nabla \mu$ of the isothermal case. In particular, this force tends to vanish within the interfacial layer. 
    On the contrary, when considering the first regime of (A1), we remark an additional contribution 
    $s  \nabla \theta$ of the entropy. This force still effects the mixing region of $\phi = 0$. 
    Because of this property, there is a fundamental advantage when mathematically treating the assumption (A2).
\end{remark}
\begin{proof}
In this proof, we limit ourselves to the case $(i)$. Indeed $(ii)$ can be derived with an analogous argument.
 
\noindent 
In order to compute the variation of the action we first recast the action functional in Lagrangian coordinates. The relation of assumption (A1) tells us that $\theta(x(X,t),t) = \theta_0(X)$, where $\theta_0$ is the initial temperature. Thus the action reads as
\begin{equation}\label{action-in-lagrangian}
A(x(\cdot))=-\int_{t\in [0,T]}\int_{X\in \Omega_0} 
    \psi
    \bigg(
    \underbrace{
        \frac{\phi_0(X)}{\det F(X,t)},F(X,t)^{-T}\nabla_X \bigg(\frac{\phi_0(X)}{\det F(X,t)}\bigg),\theta_0(X)
    }_{=:\mathcal{G}(X,t)}
    \bigg)\det F(X,t)dXdt.
\end{equation}
We consider a variation $x^{\bar\varepsilon}(X,t):=x(X,t)+{\bar\varepsilon}\delta x(X,t)$ of the trajectory $x(X,t)$, where $\delta x(X,t)$ is an arbitrary smooth vector with compact support on $\mathring{\Omega}_0 \times ]0,T[$.
Then 
\begin{align*}
\delta  A(x(\cdot ))
&:=
    \frac{d}{d{\bar\varepsilon}}\bigg|_{\epsilon=0}A(x^{\bar\varepsilon}(\cdot ))
    = -
    \int_{t\in [0,T]}
    \int_{X\in \Omega_0}
    \partial_\phi \psi \big( \mathcal{G}(X,t)\big) 
    \bigg[\frac{-\phi_0}{(\det F)^2}\det F \tr\big(F^{-1}\nabla_X \delta x \big)\bigg] \det F dXdt + \\
    & + 
    \int_{t\in [0,T]}
    \int_{X\in \Omega_0}
    \partial_{\nabla \phi}\psi \big(\mathcal{G}(X,t)\big)
    \cdot 
    \Bigg\{
        F^{-T}\nabla_X \delta x F^{-T} \nabla_X\bigg(\frac{\phi_0}{\det F}\bigg)
        + \\ 
        &\hspace{2cm} - 
        F^{-T}\nabla_X\bigg(\frac{\phi_0}{(\det F)^2}\det F \tr\big(F^{-1}\nabla_X \delta x\big)\bigg) 
    \bigg\} \det F dXdt + \\
    &- \int_{t\in [0,T]}
    \int_{X\in \Omega_0}
    \psi\big( \mathcal{G}(X,t)\big) 
         \tr\big(F^{-T}\nabla_X \delta x\big)
         \det F dXdt.
\end{align*}
At this stage, we recast the overall integral in Eulerian coordinates:
\begin{align*}
\delta  A(x(\cdot ))
&= 
\int_{t\in [0,T]}
\int_{X\in \Omega_0}
\bigg\{
    \partial_\phi \psi\big(\phi,\nabla_x \phi ,\theta\big)\phi\, {\rm div_x}\,\delta x +
    \partial_{\nabla \phi}\psi 
    \big(\phi,\nabla_x \phi ,\theta\big)\otimes \nabla_x\phi :  \nabla_x \delta x +
    \\
    &~~~-\partial_{\nabla \phi}\psi \cdot \big(\phi,\nabla_x \phi ,\theta\big)\nabla_x(\phi \, {\rm div}_x\,\delta x) -  \psi\big(\phi,\nabla_x \phi ,\theta\big){\rm div}_x\,\delta x 
\bigg\} dxdt.
\end{align*}
We aim now to isolate the variation $\delta x$, making use of an integration by parts. Therefore
\begin{align*}
\delta  A(x(\cdot ))
&=-
\int_{t\in [0,T]}
\int_{X\in \Omega_0}
\bigg\{ 
    \nabla_x\big(\phi \partial_\phi \psi \big)+{\rm div}\,\big(\partial_\nabla \phi \otimes \nabla \phi\big) -
    \nabla_x\big(\phi\, {\rm div}\,\partial_{\nabla \phi} \psi\big) - \nabla_x \psi \bigg\} \cdot\delta x dxdt,
\end{align*}
where the boundary terms on $\partial \Omega^x_t$ vanishes because $\delta x$ is supported in $\mathring{\Omega}_0\times ]0,T[$. Let us comment also that a further term $\partial_\theta \psi \nabla \theta$ would appear in the last integral, when assumptions (A2) is taken into account.
It remains to develop the function in front of $\delta x$:
\begin{align*}
&\nabla\big(\phi\partial_\phi \psi\big)+\nabla\cdot\big(\partial_{\nabla \phi} \psi \otimes \nabla \phi\big)-\nabla\big(\phi\nabla \cdot\partial_{\nabla \phi} \psi \big)-\nabla \psi = \nabla \phi\partial_\phi \psi+\phi\nabla \partial_\phi \psi+\nabla\cdot\partial_{\nabla \phi} \psi \nabla \phi+\partial_{\nabla \phi} \psi \nabla^2\phi\\
&-\nabla\cdot\partial_{\nabla \phi} \psi \nabla \phi -\phi\nabla \nabla \cdot\partial_{\nabla \phi} \psi -\partial_\phi \psi\nabla \phi-\partial_{\nabla \phi} \psi \nabla^2\phi -\partial_\theta \psi \nabla \theta = \phi\nabla \big(\partial_\phi \psi-\nabla\cdot\partial_{\nabla \phi} \psi \big)-\partial_\theta \psi \nabla \theta,
\end{align*}
where the last term coincides with $\phi \nabla \mu + s\nabla \theta$. We conclude the proof hence by recalling identity \eqref{eq:action-fcons}.
\end{proof}

\subsection{Onsager reciprocal relation}\label{sec:onsager}

In the present section we establish the equation of motion for the phase field $\phi$. When both dissipative forces $\textbf{f}_{diss}$ and conservative forces $\textbf{f}_{cons}$ are derived, the equation of motion is revealed by the force balance (Newton’s Second Law, without inertial forces): $\textbf{f}_{diss} + \textbf{f}_{cons} = 0$. To derive the dissipative force $\textbf{f}_{diss}$, we invoke the Maximum Dissipation Principle, which implies that this force may be obtained by the variation of the dissipation functional $\mathcal{D}$ with respect to the velocity field $u$. 
In order to derive a Cahn-Hilliard-type equation, we shall consider a dissipation rate of Darcy type, namely of the form
\begin{equation}\label{eq:Darcy-dissipation}
    \mathcal{D} = \int_{\Omega} \eta(\phi,\theta) |u|^2 + \alpha |\partial_t\phi |^2.
\end{equation}
Here $\eta$ stands for a friction parameter, which is positive and commonly depends on the state variables $(\phi, \theta)$ (and their derivatives). In Cahn-Hilliard the friction usually satisfies $\eta = \gamma \phi^2$ for some positive constant $\gamma>0$. We introduce a further artificial dissipation $\alpha|\partial_t\phi|^2$ which will come out useful when mathematically treating our final model.

\smallskip\noindent 
We mention that, 
in different contexts than Cahn-Hilliard,
other types of dissipation could be taken into account here. For instance a quadratic dissipation in $\nabla u$ of Stokes type. In this paper we address uniquely the case given by the Darcy law, that leads indeed to Cahn-Hilliard.

\smallskip\noindent 
At this stage, we can invoke the Maximum Dissipation Principle to derive the dissipative force (linear with respect to the rate $u$). Recalling that $\partial_t \phi + \dv(u\phi ) = 0$, we gather that
\begin{align*}
    \int_\Omega \textbf{f}_{diss} \cdot\delta u \, dx :=
     \frac{1}{2}\frac{\delta \mathcal{D}}{\delta u} 
     &=\frac{d}{d{\bar\varepsilon}}\bigg|_{{\bar\varepsilon}=0}\int_\Omega \eta(\phi,\theta)|u+{\bar\varepsilon}\delta u |^2+\alpha|\dv(\phi u+{\bar\varepsilon}\phi \delta u )|^2 dx\\
    &=\int_\Omega \eta(\phi,\theta)u\cdot\delta u  +\alpha\dv(\phi u)\dv(\phi \delta u ) dx\\
    &=\int_\Omega \big(\eta(\phi,\theta)u -\alpha\phi\nabla \dv(\phi u)\big)\cdot\delta u dx\\
    &=\int_\Omega \big(\eta(\phi,\theta)u +\alpha\phi\nabla \partial_t \phi \big)\cdot\delta u dx.
\end{align*}
The dissipative force being revealed, we are now in the condition of stating the equation for $\phi$. This depends naturally on the hypotheses (A1) or (A2) of the temperature $\theta$, that we have introduced in Section \ref{sec:kinematics}.
\begin{theorem}\label{thm:eq-for-phi}
    Assume that the hypotheses of Theorem \ref{thm:conservative-forces} are satisfied and the friction coefficient fulfills $\eta(\phi,\theta) = \phi^2$. 
    \begin{enumerate}[(i)]
        \item If the temperature $\theta$ satisfies the assumption (A1), then the equation of motion for $\phi$ is given by
        \begin{equation*}
            \begin{cases}
                \partial_t \phi + \dv(u\phi) = 0,\\
                \phi \nabla \mu + s\nabla\theta  + \phi^2 u + \alpha \phi \nabla \partial_t \phi = 0,
            \end{cases}
            \quad \text{that is}\quad 
            \partial_t \phi - \alpha \Delta \partial_t \phi = \Delta \mu + \dv\Big( \frac{s\nabla \theta}{\phi} \Big).
        \end{equation*}
    \item If the temperature $\theta$ satisfies the assumption (A2), then the equation of motion for $\phi$ is given by
        \begin{equation*}
            \begin{cases}
                \partial_t \phi + \dv(u\phi) = 0,\\
                 \nabla \mu  + \phi u + \alpha \nabla \partial_t \phi = 0,
            \end{cases}
            \quad \text{that is}\quad 
            \partial_t \phi - \alpha \Delta \partial_t \phi = \Delta \mu.
        \end{equation*}
    \end{enumerate}
\end{theorem}
\subsection{The equation of the temperature}\label{sec:derivation-temperature-eq}
The conservation of the total energy (first law of thermodynamics) and the Clausius-Duhem inequality (second law of thermodynamics) are the key ingredients to determine the equation of the absolute temperature $\theta$. The Clausius-Duhem inequality reflects the irreversibly of the process, which is grounded in the increasing of the total entropy through the relation
\begin{equation}\label{eq:clausius-duhem}
\begin{aligned}
  &\text{(A1)} \qquad  \theta ( \partial_t s + \dv(us) )  = \theta \dv \textbf{j} + \theta \Delta^*, \\
  &\text{(A2)} \qquad  \theta \partial_t s = \theta \dv \textbf{j} + \theta \Delta^*,
\end{aligned}
\end{equation}
for a suitable entropy flux $\textbf{j}$ and entropy production $\theta \Delta^*\geq 0$. 
The action of $u$ on the rate of the entropy in (A1) is motivated by the duality of $s$ with respect to $\theta$ (cf. Section \ref{sec:kinematics}). Because of our choice on the state variable, we recall that $s = -\partial \psi/\partial \theta$ must be intended as function of $(\phi,\nabla \phi,\theta)$

\smallskip\noindent 
Since $u$ has been revealed in the previous sections, it remains to determine the entropy flux $\textbf{j}$ and the entropy production $\theta \Delta^*$. While the former is encompassed in this article by the Fourier's law, the physics  of the latter is derived through the first law of thermodynamics, namely the conservation of the total energy. 

\smallskip\noindent 
The Fourier's law is a thermodynamic relation, which shows how the local heat flux density 
$\mathbf{q}$ is equal to the product of thermal conductivity $\kappa>0$ and the local temperature gradient:
\begin{equation}
  \textbf{j} =
    -
    \frac{\textbf{q}}{\theta},\qquad \text{with}\quad  \textbf{q} = -\kappa \nabla \theta,
\end{equation}
We mention that different type of heat can be taken into account, as done for instance in \cite{Marveggio21}, where the authors consider an heat flux of the form $\mathbf{q} = k \nabla (1/\theta)$.

\smallskip
\noindent
The entropy production in \eqref{eq:clausius-duhem} is revealed by the first law of thermodynamics, which states that the total energy of the system must be conserved in time. For the Cahn-Hilliard equation, the total energy coincides with the internal energy, whose density $e = e(\phi, \nabla \phi , \theta)$ differs by the free energy $\psi(\phi, \nabla \phi, \theta)$ in \eqref{eq: free energy Cahn Hilliard} by means of an additional entropic contribution:
\begin{equation}\label{eq: general internal energy}
    e(\phi,\nabla\phi,\theta):=\psi(\phi,\nabla\phi,\theta)+\theta s(\phi,\nabla\phi,\theta)= \psi(\phi,\nabla\phi,\theta)-\theta\, \partial_\theta \psi(\phi,\nabla\phi,\theta).
\end{equation}
The total energy $E^{\rm tot}(t) := \int_{\Omega} e(\phi(t,x), \nabla \phi(t,x), \theta(t,x))dx$ is then conserved whenever 
\begin{equation*}
    \frac{d}{dt }E^{\rm tot}(t) = 0.
\end{equation*}
The next result analyse this last relation and reveals the form of the entropy production $\theta \Delta^*$, in accordance to the assumptions (A1) or (A2) for the temperature $\theta$.
\begin{theorem}\label{thm:entropy-production}
    The following relations on the entropy production holds true:
    \begin{enumerate}[(i)]
        \item If the temperature satisfies (A1) and the entropy production fulfills
        \begin{equation}\label{eq:entropyproductionA1}
            \theta \Delta^* = 
                \Big| \nabla \mu + \alpha \nabla \partial_t \phi + \frac{s\nabla \theta}{\phi} \Big|^2+ \alpha |\partial_t \phi |^2 + \kappa\frac{ |\nabla \theta|^2}{\theta},
        \end{equation}
        then the total energy of the system is conserved;
        \item if the temperature satisfies (A2)  and the entropy production fulfills
        \begin{equation}\label{eq:entropyproductionA2}
            \theta \Delta^* = 
                \Big| \nabla \mu + \alpha \nabla \partial_t \phi  \Big|^2+ \alpha |\partial_t \phi |^2 + \kappa\frac{ |\nabla \theta|^2}{\theta},
        \end{equation}
        then the total energy of the system is conserved.
    \end{enumerate}
\end{theorem}
\begin{proof}
We first establish the rate of the total energy density. For the sake of a compact presentation, we abbreviate the dependence of any function $f$ on the state variables by $f(\phi, \theta)$. For example, by $e(\phi, \theta)$ we intend $e(\phi,\nabla \phi, \theta)$.

\noindent
Recalling that the state variables $(\phi, \theta)$ depends on $(t,x)\in (0,T)\times \Omega$ and that $e$ satisfies \eqref{eq: general internal energy}, we gather that
\begin{align*}
    \partial_t [e(\phi, \theta)]
    &= \partial_t [\psi(\phi, \theta)] + \partial_t[\theta\, s(\phi, \theta)] \\
    & = \partial_\phi \psi(\phi, \theta)\, \partial_t \phi + \partial_{\nabla \phi}\psi(\phi,\theta) \cdot \partial_t \nabla \phi + \partial_\theta \psi(\phi, \theta) \,\partial_t \theta + (\partial_t \theta)s(\phi, \theta) + \theta\,\partial_t [s(\phi, \theta)].
\end{align*}
Because of the thermodynamic definition of the entropy $s(\phi, \theta) = - \partial_\theta \psi(\phi, \theta)$, the third and fourth components cancel out. 
In particular, recalling that $\partial_t \phi + \dv(u\,\phi) = 0$, we infer that
\begin{equation*}
\begin{aligned}
    \partial_t [e(\phi, \theta)]
    &= \Big[ \partial_\phi\psi(\phi,\theta) - \dv \big( \partial_{\nabla \phi} \psi(\phi,\theta) \big)\Big]\partial_t \phi  + 
    \dv \big(\partial_{\nabla \phi} \psi(\phi,\theta) \partial_t \phi\big)  + \theta \partial_t [s(\phi,\theta)]\\
     &= -\Big[ \partial_\phi\psi(\phi,\theta) - \dv \big( \partial_{\nabla \phi} \psi(\phi,\theta) \big)\Big]\dv(u\phi)  + 
    \dv \big(\partial_{\nabla \phi} \psi(\phi,\theta) \partial_t \phi\big)  + \theta \partial_t [s(\phi,\theta)].
\end{aligned}
\end{equation*}
We then invoke the definition of the chemical potential $\mu = \mu(\phi,\theta) = \partial_\phi \psi(\phi,\theta) - \dv \big(\partial_{\nabla \phi}\psi(\phi,\theta)\big)$, therefore
\begin{equation}\label{eq:biforcation}
    \partial_t [e(\phi, \theta)]=- \mu \dv(u\phi)  + 
    \dv \big(\partial_{\nabla \phi} \psi(\phi,\theta) \partial_t \phi\big)  + \theta \partial_t [s(\phi,\theta)].
\end{equation}
We further develop the last relation, by separately addressing the assumption (A1) and (A2) in the Clausius-Duhem inequality \eqref{eq:clausius-duhem}.
If (A1) holds, then \eqref{eq:biforcation} reduces to
\begin{align*}
    \partial_t [e(\phi, \theta)]&=  
    - \dv\big\{u\phi \mu - \partial_{\nabla \phi}\psi(\phi,\theta)\partial_t \phi \big\}  + \nabla\mu \cdot u\phi  - \theta \dv \big(u\,s(\phi,\theta)\big) - \dv \mathbf{q} + \frac{\mathbf{q}\cdot \nabla \theta}{\theta} + \theta \Delta^*\\
    &=  
    - \dv\big\{u\phi \mu  + \theta s u - \partial_{\nabla \phi}\psi(\phi,\theta)\partial_t \phi \big\}  + (\nabla\mu + s(\phi,\theta)\nabla \theta) \cdot u  - \dv \mathbf{q} + \frac{\mathbf{q}\cdot \nabla \theta}{\theta} + \theta \Delta^*.
\end{align*}
Next, we recall that thanks to Theorem \ref{thm:eq-for-phi} the relation $\nabla \mu + s\nabla \theta = -\phi^2 u - \alpha\phi \nabla \partial_t \phi$. Furthermore, with the Fourier's law $\mathbf{q} =- k\nabla \theta $, we get that
\begin{align*}
    \partial_t &[e(\phi, \theta)]
    =  
    - \dv\big\{u\phi \mu  + \theta s u - \partial_{\nabla \phi}\psi(\phi,\theta)\partial_t \phi \big\}  -\phi^2|u|^2 - \alpha \phi u\cdot  \nabla \partial_t \phi  - \dv \mathbf{q} - \kappa\frac{|\nabla \theta|^2}{\theta} + \theta \Delta^*\\
     &=  
    - \dv\big\{u\phi \mu  + \theta s u - \partial_{\nabla \phi}\psi(\phi,\theta)\partial_t \phi 
    - \alpha \phi u \partial_t \phi \big\}  -\phi^2|u|^2 + \alpha \dv(\phi u) \partial_t \phi  - \dv \mathbf{q} - \kappa\frac{|\nabla \theta|^2}{\theta}  + \theta \Delta^*\\
     &=  
    - \dv\big\{u\phi \mu  + \theta s u - \partial_{\nabla \phi}\psi(\phi,\theta)\partial_t \phi 
    - \alpha \phi u \partial_t \phi \big\}  -\phi^2|u|^2 -\alpha |\partial_t \phi|^2  - \dv \mathbf{q} - \kappa\frac{|\nabla \theta|^2}{\theta}  + \theta \Delta^*
\end{align*}
Integrating in space along $\Omega$ and remarking that the boundary contributions vanish (since $u = 0$ in $\partial \Omega$), we eventually gather that
\begin{equation*}
    \frac{d}{dt}E^{\rm tot}(t) = \int_{\Omega}\Big(\theta \Delta^* - \phi^2|u|^2 -\alpha|\partial_t\phi|^2 -\kappa\frac{|\nabla \theta|^2}{\theta} \Big) dx = 0,
\end{equation*}
as long as the identity \eqref{eq:entropyproductionA1} is satisfied. This concludes the proof Theorem \ref{thm:entropy-production}, part $(i)$.

\smallskip\noindent 
The argument to prove $(ii)$ is analogous. Recalling \eqref{eq:biforcation} and the assumption (A2) in \eqref{eq:clausius-duhem}, we infer that
\begin{equation}\label{eq:derivation-of-energy-eq}
\begin{aligned}
    \partial_t [e(\phi, \theta)]&=  
    - \dv\big\{u\phi \mu - \partial_{\nabla \phi}\psi(\phi,\theta)\partial_t \phi \big\}  + \nabla\mu \cdot u\phi - \dv \mathbf{q} - \frac{\mathbf{q}\cdot \nabla \theta}{\theta} + \theta \Delta^*\\
    &=  
    - \dv\big\{u\phi \mu - \partial_{\nabla \phi}\psi(\phi,\theta)\partial_t \phi \big\}   -\phi^2|u|^2 -\alpha\phi\nabla\partial_t \phi - \dv \mathbf{q} -\kappa \frac{|\nabla \theta|^2}{\theta} + \theta \Delta^*\\
    &=  
    - \dv\big\{u\phi \mu - \partial_{\nabla \phi}\psi(\phi,\theta)\partial_t \phi -\alpha \phi u\partial_t \phi \big\}   -\phi^2|u|^2 -\alpha|\partial_t \phi|^2 - \dv \mathbf{q}-\kappa \frac{|\nabla \theta|^2}{\theta}  + \theta \Delta^*.
\end{aligned}
\end{equation}
The relation \eqref{eq:entropyproductionA2} on the entropy production ensures then that the total energy is conserved. This concludes the proof of the Theorem.
 
\subsection{Summary of the overall systems} 
In this section we summarize the equations for the non-isothermal Cahn-Hilliard, obtained in Theorem \ref{thm:conservative-forces}, Theorem \ref{thm:eq-for-phi} and Theorem \ref{thm:entropy-production}. As depicted in Section \ref{sec:kinematics}, the structure of the model is entangled with the assumptions (A1) or (A2) on $\theta$, i.e. whether $\theta$ is related to the effective velocity $u$ or not. Therefore we must distinguish these two cases.

\smallskip\noindent 
For the sake of a compact and clear presentation, we maintain in the equations the notation of the velocity $u$, the chemical potential $\mu$ and the local entropy $s$. The reader shall remark that these are functions on $(\phi,\theta)$, therefore the final systems can be recast just in terms of these state variables (cf. for instance \eqref{eq:phi-to-analyse}--\eqref{eq:theta-to-analyse}). Furthermore, for simplicity, we recall here our assumption on the energy density and therefore on the corresponding local entropy and chemical potential:
\begin{equation*}
\begin{alignedat}{16}
    &W(\phi,\theta) &&= \frac{(\phi^2-1)^2}{4} + c(\theta)\phi^2,\qquad 
    &&&&\psi(\phi,\theta)
    &&&&&&&&= 
    \frac{\tilde \epsilon(\theta)}{2}|\nabla \phi|^2 + 
    \frac{1}{\tilde \epsilon(\theta)}W(\phi,\theta) - k_B \theta \ln \theta,\\
    &\mu(\phi,\theta)&&= -\tilde \epsilon(\theta)\Delta \phi + \frac{1}{\tilde \epsilon(\theta)}\partial_\phi W(\phi, \theta),\; 
    &&&&s(\phi,\theta) 
    &&&&&&&&=  
    -\frac{\tilde \epsilon'(\theta)}{2}|\nabla \phi|^2+\frac{\tilde \epsilon'(\theta)}{\tilde \epsilon(\theta)^2} W(\phi,\theta)-\!\frac{\partial_\theta W(\phi,\theta)}{\tilde \epsilon(\theta)} + k_B(1+\ln \theta).
\end{alignedat}    
\end{equation*}
The equations are defined on the domain $(x,t)\in \Omega\times (0,T)$ and they shall be supplemented by appropriate boundary conditions on  $\partial \Omega\times (0,T)$. In the forthcoming sections we address the analysis of an unbounded domain $\Omega = \mathbb{R}^d$.

\smallskip\noindent 
Let us assume that (A1) of Section \ref{sec:kinematics} holds true. Then the following system on the state variables $(\phi(x,t),\theta(x,t))$ satisfies both the first and second laws of thermodynamics:
\begin{itemize}
    \item 
the \textbf{continuity equation} (equation for $\phi$)
 \begin{equation}\label{eq:final-sec-first-system-continuity-eq}
    \partial_t \phi+\nabla\cdot(\phi \, u(\phi,\theta))=0,
 \end{equation}
 \item the \textbf{force balance equation}
    \begin{equation}\label{eq:final-sec-first-system-balance-eq}
        \phi^2 u(\phi,\theta) := - \phi \nabla (\mu(\phi,\theta)) - s(\phi,\theta) \nabla \theta  -\alpha\phi\nabla \partial_t \phi,
    \end{equation}
    \item the \textbf{Clausius-Duhem inequality} (equation for $\theta$) 
    \begin{equation}
        \theta 
        \partial_t [s(\phi,\theta)]+\dv \big[s(\phi,\theta)u(\phi,\theta)\big]=\theta \nabla\cdot\bigg(\frac{\kappa\nabla \theta}{\theta}\bigg)+\theta \Delta^*(\phi,\theta),
    \end{equation}
    \item the \textbf{entropy production rate}
    \begin{equation}
        \theta \Delta^*(\phi,\theta):= \phi^2|u(\phi,\theta)|^2+\alpha |\partial_t\phi|^2+\kappa\frac{|\nabla \theta|^2}{\theta},
    \end{equation}
    \item the \textbf{boundary condition}
    \begin{equation}
        \nabla \phi\cdot n=0,~~\nabla \theta\cdot n=0,~~ u(\phi,\theta)\cdot n=0,
    \end{equation}
    where $n(x,t)$ stands for the normal to the boundary on $(x,t)\in \partial \Omega \times (0,T)$,
    \item the \textbf{initial data}
        \begin{align}
        \phi(0,x)=\phi_0(x),~~\theta(0,x)=\theta_0(x)~~x\in \Omega.
    \end{align}
\end{itemize}
\begin{remark}\label{rmk:problem-with-A1}
Before stating the second thermodynamically consistent model, some remarks are here in order. From the force balance equation \eqref{eq:final-sec-first-system-balance-eq}, we can evince that the term $\phi\,u(\phi,\theta)$ (which is also in \eqref{eq:final-sec-first-system-continuity-eq}) could not be well defined in the interface region, because of the entropy term $s\nabla \theta/\phi$. An open question is whether or not there are initial data (such as ones with constant temperature along the mixing region), for which this system admits solutions and therefore it has a meaning, at least locally in time. 

\noindent
This problem does not occur when dealing with the second system of (A2). Although the assumption (A1) would seem inherently reasonable, i.e. to assume that the temperature is transported by the flow $u$, the mathematical treatment of our models would favors the second case (A2).
\end{remark}

\smallskip\noindent 
Let us assume that assumption (A2) holds true.  Then the following system on the state variables $(\phi(x,t),\theta(x,t))$ satisfies both the first and second laws of thermodynamics:
\begin{itemize}
    \item 
the \textbf{continuity equation} (equation for $\phi$)
 \begin{equation}\label{eq:second-system-cont-eq-final-sec}
    \partial_t \phi+\nabla\cdot(\phi \, u(\phi,\theta))=0,
 \end{equation}
 \item the \textbf{force balance equation}
    \begin{equation}
        \phi\, u(\phi,\theta) := -  \nabla (\mu(\phi,\theta))  -\alpha \nabla \partial_t \phi,
    \end{equation}
    \item the \textbf{Clausius-Duhem inequality} (equation for $\theta$) 
    \begin{equation}
        \theta 
        \partial_t [s(\phi,\theta)]=\theta \nabla\cdot\bigg(\frac{\kappa\nabla \theta}{\theta}\bigg)+\theta \Delta^*(\phi,\theta),
    \end{equation}
    \item the \textbf{entropy production rate}
    \begin{equation}\label{eq:second-system-entropy-final-sec}
        \theta \Delta^*(\phi,\theta):= \Big| \nabla (\mu(\phi,\theta)) + \alpha \nabla \partial_t \phi \Big|^2+\alpha |\partial_t\phi|^2+\kappa\frac{|\nabla \theta|^2}{\theta},
    \end{equation}
    \item the \textbf{boundary condition}
    \begin{equation}
        \nabla \phi\cdot n=0,~~\nabla \theta\cdot n=0,~~ u(\phi,\theta)\cdot n=0,
    \end{equation}
    where $n(x,t)$ stands for the normal to the boundary on $(x,t)\in \partial \Omega \times (0,T)$,
    \item the \textbf{initial data}
        \begin{equation}\label{eq:second-system-initial-data-final-sec}
            \phi(0,x)=\phi_0(x),~~\theta(0,x)=\theta_0(x)~~x\in \Omega.
        \end{equation}
\end{itemize}
The physics of the model being revealed, in the next sections we address the well-posedness of system \eqref{eq:second-system-cont-eq-final-sec}--\eqref{eq:second-system-initial-data-final-sec}. To this end, we shall first recast the entire equations just in terms of $\phi$ and $\theta$. Furthermore, the interface parameter $\tilde \epsilon(\theta)$ and the potential parameter $c(\theta)$ are here considered as
\begin{equation*}
   \tilde  \epsilon(\theta)=\epsilon \theta,\qquad c(\theta) = \frac{(\theta-\bar{\theta})^3}{3},
\end{equation*}
where $\epsilon>0$ is fixed and $\bar{\theta}>0$ represents an high temperature in which our model switch from a double well-potential to a single one. The final system reads as follows
\begin{align}
    \label{eq:phi-to-analyse}
    \partial_t \phi-\alpha\Delta\partial_t\phi + 
     \epsilon\bar{\theta} \Delta^2 \phi &=\tilde f_1(\phi,\theta),
    \\
    \label{eq:theta-to-analyse}
    k_B \partial_t \theta-\kappa\Delta \theta&= \tilde f_2(\phi,\theta),
\end{align}
where 
\begin{align*}
    \tilde f_1(\phi, \theta) 
    &:= -\epsilon 
    \Delta ((\theta-\bar{\theta})\Delta \phi ) + 
    \Delta\Big( \frac{1}{\epsilon\theta} \partial_\phi W(\phi,\theta)\Big),\\
   \tilde  f_2(\phi,\theta)&:=
    \alpha(\partial_t \phi)^2+
    \epsilon\theta \,\partial_t\nabla\phi\cdot\nabla\phi - \theta \partial_t \bigg[\frac{1}{\epsilon\theta^2} W(\phi,\theta)-\frac{1}{\epsilon\theta } (\theta- \bar{\theta})^2 \phi^2\bigg] +|\alpha\nabla \partial_t \phi+\nabla [\mu(\phi,\theta)]|^2.
\end{align*}
\end{proof}

\section{Local-in-time classical solutions}\label{sec:well-posedness}
The main goal of the present paragraph is to derive the existence of classical solutions to problem \eqref{eq:phi-to-analyse}--\eqref{eq:theta-to-analyse}, when considered in the whole space $\mathbb{R}^d$, with dimension $d\geq 2$. At the currant stage, a global in time approach does not seem to be accessible for general initial data. We will therefore address the existence and uniqueness of local classical solutions for our model. We treat in particular classical solutions that belong to Besov spaces with high indexes of regularity. Detailed information about the considered Besov ansatz is given in Section \ref{section:Besov-spaces}.  

\begin{theorem}\label{thm:well-posedness}
Let $\phi_0\in B^{d/2+2}$, $\theta_0 - \bar{\theta} \in B^{d/2}$ for a positive constant temperature $\bar{\theta}$ and let $\alpha>0$, $\kappa>0$ and $\kappa_B>0$ be parameters. Then there exists a small positive constant $\varepsilon_0\in [0,1)$,  which depends only on the dimension $d$, such that for any $\epsilon,\,\bar{\theta},\,\phi_0,\,\theta_0$ satisfying the smallness conditions
\begin{equation}\label{smallness-condition}
\begin{aligned}
    \epsilon \| \Delta \phi_0 \|_{B^{\frac{d}{2}}}  
    &+ \frac{1}{\bar{\theta} }
    \max\Big\{1, \frac{1}{\epsilon}\Big\} (1+\| \phi_0 \|_{B^{\frac{d}{2}}})^4
    < \varepsilon_0
    \min\big\{1, \alpha,\kappa, \kappa_B \big\},\\
       \| \theta_0 - \bar{\theta}\|_{B^\frac{d}{2}} 
    &
       <
    \left(
    \min 
    \Big\{
        \epsilon^2,\, 
        \frac{1 }{1+\epsilon}
    \Big\}
    \frac{
         \varepsilon_0\min\{1,\alpha,\kappa, \kappa_B\}
    }{\bar{\theta}(1+\alpha)}
    \right)^2,
\end{aligned}
\end{equation}
the system \eqref{eq:phi-to-analyse}--\eqref{eq:theta-to-analyse} admits a unique classical solution $(\phi,\theta)$ in $(0,T)\times \mathbb{R}^d$, for a sufficiently small time $T>0$. Furthermore, the solution $(\phi,\theta)$ belongs to the function space
\begin{alignat*}{8}
    \phi &\in \tilde{\mathcal{C}}([0,T],B^{\frac{d}{2}+2}),\quad 
    &&\Delta^2\phi \in
    L^1(0,T;B^{\frac{d}{2}}),\qquad 
    &&&&
    \partial_t \phi \in \tilde  L^2(0,T; B^{d/2}),\quad 
    \partial_t \nabla  \phi \in\tilde L^2(0,T; B^{d/2})
    \\
    \theta &\in 
    \tilde{\mathcal{C}}([0,T],B^{\frac{d}{2}}), \quad 
    &&\Delta \theta\in 
    L^1(0,T;B^{\frac{d}{2}}),
    \qquad 
    &&&& \partial_t \theta \in L^1(0,T; B^{d/2}).
\end{alignat*}
\end{theorem}
\noindent 
\begin{remark}
Some remarks are in order here. As already described in Remark \ref{rmk:smallness-condition}, the first relation \eqref{smallness-condition}
is not a standard smallness condition on the initial data. Indeed \eqref{smallness-condition} states that for any initial phase $\phi_0$ (which can thus be taken arbitrarily large) we can determine a sufficiently small $\epsilon>0$ and a sufficiently large temperature in $\bar{\theta}$, such that a local-in-time classical solution exists. The fact that our solutions are not global in time shall therefore be associated to the arbitrariness of the phase field $\phi_0$. One shall also remark that a small condition directly  on $\| \Delta \phi \|_{B^{d/2}}$ for the phase field $\phi_0$  would be in general unnatural, since we aim to model transitions between the pure states $\pm 1$.

\noindent Furthermore, we have imposed in the second inequality of \eqref{smallness-condition}  a more strict smallness condition between the distance of the initial temperature $\theta_0$ and the high constant temperature $\bar{\theta}$. 
In particular, an additional term of the form $1/(1+\alpha)^2$ has been added, which is somehow counterintuitive (at a first glance). Indeed $\alpha$ produces an additional regularising mechanism on the $\phi$-equation (cf.~$\alpha \partial_t\Delta \phi$ in \eqref{eq:phi-to-analyse}), thus we would expect that high values of $\alpha$ support our analysis, not that they further restrict the smallness relation. A deeper analysis however clarifies this aspect, since $\alpha$ does not only affect the $\phi$-equation, but also the entropy production $\theta\Delta^*$ in the temperature equation (cf.~the last term $\alpha | \nabla \partial_t \phi|^2$ of $f_2(\phi, \theta)$ in \eqref{eq:theta-to-analyse}). In other words, a more strict smallness relation copes with the major nonlinearities appearing in the equation of $\theta$.

\end{remark}

\noindent 
In what follows, we describe the main idea for the proof of Theorem \ref{thm:well-posedness} and postpone the detailed estimates that are rather involved to the subsequent sections. We shall first remark some aspects of the hypotheses on the initial data $(\phi_0,\,\theta_0-\bar{\theta})$. Both the initial phase field $\phi_0$ and the initial temperature $\delta\theta_0=\theta_0-\bar{\theta}$ are bounded in $B^{d/2}$ which is an algebra embedded in $L^\infty(\mathbb{R}^d)\cap \mathcal{C}(\mathbb{R}^d)$. Moreover we assume that $\theta_0$ is close to a constant value $\bar{\theta}>0$, which represents a sufficiently high temperature.


\smallskip
\noindent 
The function spaces of Theorem \ref{thm:well-posedness} yields both $\phi$ and $\theta$ to be in $\mathcal{C}^\infty((0,T]\times \mathbb{R}^d)$. Furthermore, if the initial data $\phi_0,\,\delta\theta_0$ belongs to the Schwarz space $\mathcal{S}(\mathbb{R}^d)$, then the solution belongs to $\mathcal{C}^\infty([0,T]\times \mathbb{R}^d)$, therefore we are indeed treating smooth solutions that are local in time. 

\smallskip
\noindent
 We now present the main strategy behind the proof of Theorem \ref{thm:well-posedness}. We make use of a Friedrichs scheme, coupled with some a-priori estimates for a linearised version of our system. To this end we shall first introduce the following function space:
\begin{equation}\label{def:KchiT}
\begin{aligned}
    \mathcal{K}_{\chi, T} &:= 
    \bigg\{
     (\delta \phi,\,  \delta \theta)
     :(0,T)\times \mathbb{R}^d \to \mathbb{R}^2 
     \quad \text{such that}
     \\
     &\delta \phi \in \tilde L^\infty(0,T;B^{\frac{d}{2}+2}),\,
     \Delta^2 \delta \phi\in  L^1(0,T;B^{\frac{d}{2}}),\,
     \partial_t \delta \phi \in \tilde L^2(0,T;B^{\frac{d}{2}}),
      \; \text{and}\;
     \partial_t \nabla \delta \phi \in \tilde L^2(0,T;B^{\frac{d}{2}}),
     \\
     &\delta \theta \,\in \tilde L^\infty(0,T;B^\frac{d}{2}),\;
     \Delta \delta \theta \in  L^1(0,T;B^{\frac{d}{2}})
     \; \text{and}\; 
     \partial_t \delta \theta \,\in   L^1(0,T;B^{\frac{d}{2}}), \quad 
      \| (\delta \phi, \delta \theta) \|_{\mathcal{K}}\leq \chi
    \bigg\}, 
\end{aligned}
\end{equation}
for a general parameter $\chi>0$ and a general time $T>0$. The corresponding norm $\| \cdot  \|_{\mathcal{K}}$ is here denoted by (we drop the notation on $\chi$ and $T$ for the sake of a compact presentation):
\begin{equation}\label{eq:norm-on-K}
\begin{aligned}
    \| (\delta \phi, \delta \theta) \|_{\mathcal{K}} := 
     \| \delta \phi \|_{\tilde L^\infty(0,T;B^{\frac{d}{2}+2})} &+ 
     \| \Delta^2 \delta \phi \|_{L^1(0,T;B^{\frac{d}{2}})} +
     \| \partial_t\delta  \phi \|_{\tilde L^2(0,T;B^{\frac{d}{2}})} +
     \| \partial_t\nabla \delta  \phi \|_{\tilde L^2(0,T;B^{\frac{d}{2}})} + \\
     & 
     +\| \delta \theta \|_{\tilde L^\infty(0,T;B^{\frac{d}{2}})} 
     +\|\Delta \delta \theta \|_{ L^1(0,T;B^{\frac{d}{2}})} +
     \| \partial_t \delta \theta \|_{ L^1(0,T;B^{\frac{d}{2}})}.
\end{aligned}
\end{equation}
Here, $\delta$ has a purely notation purpose, since  our final solution will eventually be given by $\phi(t,x) =\phi_L(t,x) + \delta \phi(t,x)$ and $\theta(t,x) = \bar{\theta} +\delta \theta (t,x)$, where $\bar{\theta}$ is the given constant temperature and  $\phi_L$ is solution of the linear system 
\begin{equation*}
    \partial_t \phi_L - \alpha \Delta\partial_t \phi_L + \epsilon \bar{\theta}\Delta^2\phi_L = 0,
\end{equation*}
defined in $(0,T)\times \mathbb{R}^d$ and with initial data $\phi_{L|t=0}= \phi_0$. At this stage $T>0$ is a general time that will be determined later in terms of $\chi$, $\phi_L$ and $\bar{\theta}$ (cf. \eqref{condition:phi_L-smaller-than-chi^2}).

\noindent 
We aim thus to define a suitable continuous operator $\mathcal{L}:\mathcal{K}_{\chi, T}\to \mathcal{K}_{\chi, T}$ which represents our system of PDE's \eqref{eq:phi-to-analyse}--\eqref{eq:theta-to-analyse}. To this end, for any $(\delta \phi, \delta \theta)\in \mathcal{K}_{\chi, T}$, we set 
\begin{equation}\label{tilde-system}
\mathcal{L}(\delta \phi, \delta \theta):=(\delta \tilde \phi, \delta \tilde \theta),\text{ where }(\delta \tilde \phi, \delta \tilde \theta)\text{ is the solution of }
\begin{cases}
    \partial_t \delta\tilde \phi - \alpha \partial_t \Delta\delta\tilde \phi + 
    \epsilon  \bar{\theta} \Delta^2 \delta\tilde \phi  = f_1(\delta\phi, \delta\theta),\\
    \kappa_B\partial_t\delta \tilde \theta - \kappa \Delta\delta \tilde \theta = f_2(\delta\phi, \delta\theta)
    ,\\
    (\delta \tilde \phi,\, \delta \tilde \theta)_{|t=0} = (0,\,\delta\theta_0)
    .
\end{cases}    
\end{equation}
with $f_1(\delta\phi, \delta \theta) = \tilde f_1(\phi_L+\delta\phi,\bar{\theta}+\delta \theta)$ and $f_2(\delta\phi, \delta \theta) = \tilde f_2(\phi_L+\delta\phi,\bar{\theta}+\delta \theta)$ being defined in accordance to the terms $\tilde f_1$ and $\tilde f_2$ of \eqref{eq:phi-to-analyse}--\eqref{eq:theta-to-analyse}, namely
\begin{equation}\label{def:f1-function}
\begin{aligned}
    f_1(\delta \phi, \delta \theta) = \!\!
    - \epsilon \Delta (\delta \theta \Delta (\phi_L \!+\!\delta \phi)) \!+\! 
    \Delta 
    \Big( 
        \frac{1}{\epsilon(\bar{\theta} + \delta \theta) }
        \Big( 
            ((\phi_L + \delta \phi)^2\!\!-1)(\phi_L + \delta \phi) + 
            \frac{(\delta \theta)^3}{3} (\phi_L + \delta \phi)
        \Big)\!
    \Big)
\end{aligned}
\end{equation}
and
\begin{equation}\label{def:f2}
\begin{aligned}
    &f_2(\delta \phi, \delta \theta):=
        \alpha(\partial_t \phi_L + \partial_t \delta \phi)^2+
        \epsilon\big(\bar{\theta}+\delta\theta\big)\partial_t(\nabla 
        \phi_L + \nabla \delta \phi)\cdot(\nabla 
        \phi_L + \nabla \delta \phi) + \\
        &-(\bar{\theta}+\delta\theta)\partial_t 
        \bigg[
            \frac{1}{\epsilon (\bar{\theta}\!+\!\delta\theta)^2} \Big( \frac{1}{4}(1-|\phi_L \!+\! \delta \phi|^2)^2 + 
            \frac{(\delta \theta)^3}{3}
            (\phi_L\! +\! \delta \phi)^2 \Big)\!-\!
            \frac{1}{\epsilon(\bar{\theta}+\delta\theta)}
            {(\delta \theta)^2}(\phi_L + \delta \phi)^2
        \bigg]  + \\
 &+\Big|\alpha\nabla\partial_t \phi \!\!-\!\!
 \nabla \Big(  \epsilon (\bar{\theta} + \delta \theta )(\Delta \phi_L\! + \!\Delta \delta\phi)  \!-\! 
 \frac{1}{\epsilon(\bar{\theta}\!+\!\delta\theta)}( (\phi_L + \delta \phi)^2-1)(\phi_L\! +\! \delta \phi) + 
    {  \frac{2(\delta \theta)^3}{3\epsilon (\bar{\theta}\!+\!\delta \theta)}} (\phi_L+ \delta \phi) \Big)\Big|^2\!\!.
\end{aligned}
\end{equation}
Through some rather involved estimates, in the forthcoming sections we show that the solution $\mathcal{L}(\delta \phi, \delta \theta)$ of \eqref{tilde-system} still belongs to $\mathcal{K}_{\chi, T}$, as long as $\chi>0$ and $T>0$ are sufficiently small, as well as the condition \eqref{smallness-condition} holds true. 

\smallskip
\noindent
Furthermore, by assuming $\chi$ sufficiently small, we will further show that the operator $\mathcal{L}$ is a contraction on $\mathcal{K}_{\chi, T}$. More precisely, for any $(\delta \phi_1, \delta \theta_1),\,(\delta \phi_2, \delta \theta_2)$,
\begin{equation}\label{relation-recursive-sequence}
    \| \mathcal{L}(\delta \phi_1, \delta \theta_1)-\mathcal{L}(\delta \phi_2, \delta \theta_2)\|_{\mathcal{K}}
    \leq L_{\chi}
    \| (\delta \phi_1, \delta \theta_1)-(\delta \phi_2, \delta \theta_2)\|_{\mathcal{K}},
\end{equation}
for a fixed positive constant $L_{\chi}<1$. This relation implies that the continuous operator $\mathcal{L}$  admits a unique fixed point in $\mathcal{K}_{\chi, T}$. This fixed point corresponds to the unique local solution given by Theorem \ref{thm:well-posedness}.

\smallskip\noindent
The forthcoming analysis shows that $(\delta \tilde \phi,\delta \tilde \theta)$ satisfies the inequalities of \eqref{def:KchiT} and the proof is performed mainly using the estimates provided by the following lemmas about a linearised version of the system.
\begin{lemma}\label{lemma:a-priori-estimates-for-phi}
    For any viscosity $\nu>0$, any damping $\alpha\geq 0$, any function $g \in L^1(0,T; B^{d/2})$ and any initial value $\varphi_0 \in B^{d/2}$, the following Cauchy problem in $(0,T)\times \mathbb{R}^d$ admits a unique solution $\varphi$ in $\tilde L^\infty(0,T; B^{d/2})\cap L^1(0,T;B^{d/2+4})$, with $\alpha \Delta \varphi \in \tilde L^\infty(0,T; B^{d/2})$:
    \begin{equation*}
        \partial_t \varphi  - \alpha \Delta \partial_t  \varphi + \nu \Delta^2  \varphi  = g,
        \qquad 
        \varphi_{|t=0} = \varphi_0.
    \end{equation*}
    Furthermore, there exists a constant $C>0$,  which depends only on the dimension $d\geq 1$, such that
    \begin{equation*}
    \begin{aligned}
        \|  \varphi             \|_{\tilde L^\infty(0,T; B^{\frac{d}{2}})} 
        &\leq 
        C 
        \Big(
            \| \varphi_ 0   \|_{B^{\frac{d}{2}}} + 
            \| g            \|_{L^1(0,T;B^{\frac{d}{2}})}
        \Big),\\
        \alpha 
        \|  \Delta \varphi  \|_{\tilde L^\infty(0,T; B^{\frac{d}{2}})} 
        &\leq 
        C 
        \Big(
            \alpha 
            \| \Delta \varphi_ 0   \|_{B^{\frac{d}{2}}} + 
            \| g            \|_{L^1(0,T;B^{\frac{d}{2}})}
        \Big),\\  
        \nu 
        \|  \Delta^2 \varphi            \|_{L^1(0,T;B^{\frac{d}{2}})}+
        \| \partial_t  \varphi          \|_{L^1(0,T;B^{\frac{d}{2}})}
        &\leq 
        C\Big(
            \| \varphi_ 0   \|_{B^{\frac{d}{2}}}+
            \alpha 
            \|  \Delta \varphi_0  \|_{B^{\frac{d}{2}}} + 
            \| g         \|_{L^1(0,T;B^{\frac{d}{2}})}
        \Big).
    \end{aligned}
    \end{equation*}
    Finally, if $\alpha>0$ and the function $g\in \tilde L^2(0,T; B^{d/2-1})$ is of the form $g = \Delta G$, we eventually have that
    \begin{equation*}
        \| \partial_t  \varphi \|_{\tilde L^2(0,T;B^{\frac{d}{2}})} + 
        \sqrt \alpha \| \partial_t \nabla \varphi \|_{\tilde L^2(0,T;B^{\frac{d}{2}})}
        \leq 
        C
        \Big(
             \sqrt{\nu} \| \Delta \phi_0 \|_{B^{\frac{d}{2}}}+
            \frac{1}{\sqrt{\alpha}}
            \|{\color{red}
            g
            }  \|_{ \tilde L^2(0,T;B^{\frac{d}{2}-1})}
        \Big).
    \end{equation*}
\end{lemma}
\begin{lemma}\label{lemma:a-priori-estimates-for-theta}
For any function $h\in L^1(0,T; B^{\frac{d}{2}})$ with $s \in \mathbb{R}$ and any initial value $\theta_0\in B^{\frac{d}{2}}$, the following Cauchy problem in $(0,T)\times \mathbb{R}^d$ admits a unique solution $\delta \theta \in \tilde L^\infty(0,T; B^{s})\cap L^1(0,T; B^{s+2})$
\begin{equation*}
    \kappa_B\partial_t   \theta - \kappa \Delta   \theta = h,\qquad 
      \theta_{|t=0} = \theta_0.
\end{equation*}
Furthermore, there exists a constant $C>0$ such that
\begin{equation*}
    \kappa_B\|  \theta \|_{\tilde L^\infty(0,T;B^{s})} + \kappa \| \Delta  \theta \|_{  L^1(0,T;B^{s})} + 
    \kappa_B\| \partial_t \theta \|_{ L^1(0,T;B^{s})}
    \leq C\Big\{ \kappa_B\| \theta_0 \|_{B^{s}}+ \| h \|_{L^1(0,T;B^{s})}\Big\}.
\end{equation*}
\end{lemma}
\noindent 
We postpone the proof of both lemmas to Section \ref{section:Besov-spaces}. We shall nevertheless remark that the viscosity $\nu>0$ of Lemma \ref{lemma:a-priori-estimates-for-phi} is defined in the $\phi$-equation in \eqref{tilde-system} by $\nu = \epsilon \bar{\theta}$, which depends both on the constant $\epsilon>0$ as well on our constant temperature $\bar{\theta}>0$. We shall thus spend some effort in explicitly determining the impact of this viscosity (and of course also of the other constants), in relation to our estimates.

\subsection{First part of the proof of Theorem \ref{thm:well-posedness}}

Starting from this section, we proceed to prove Theorem \ref{thm:well-posedness}. 
We show that the operator $\mathcal{L}$ defined in \eqref{tilde-system} maps $\mathcal{K}_{\chi, T}$ into itself, as long as $\chi$ and $T$ are assumed sufficiently small. 

\smallskip\noindent
For the sake of simplicity, we introduce the following compact notation $\tilde L^r_T B^s$ for the function space $\tilde L^r(0,T;B^s)$. The proof of Theorem \ref{thm:well-posedness} relies on the following estimates, given by Lemma \ref{lemma:a-priori-estimates-for-phi} and Lemma~\ref{lemma:a-priori-estimates-for-theta}:
\begin{equation}\label{est:first-basic-estimate-of-deltaphi-deltatheta}
    \begin{aligned}
        \|\delta \tilde  \phi  \|_{\tilde L^\infty_T B^\frac{d}{2}} + 
        \alpha 
        \| \Delta \delta \tilde  \phi \|_{\tilde L^\infty_T B^{\frac{d}{2}}}+ \epsilon \bar{\theta} \| \Delta^2 \delta \tilde  \phi \|_{L^1_T B^\frac{d}{2}}
        &\leq 
        \tilde C 
        \| f_1(\delta \phi,\,\delta \theta) \|_{L^1_T B^\frac{d}{2}},\\
        \| \partial_t \delta \tilde \phi \|_{\tilde L^2_T B^{\frac{d}{2}}} +
        \sqrt{\alpha}
        \| \partial_t\nabla  \delta \tilde \phi \|_{\tilde L^2_T B^{\frac{d}{2}}}
        &\leq 
        \tilde C 
        \| f_1(\delta \phi,\,\delta \theta) \|_{\tilde L^2_T B^{\frac{d}{2}-1}},
        \\
        \kappa_B
        \| \partial_t \delta \tilde \theta \|_{L^1_T B^{\frac{d}{2}}}
        +
        \kappa_B\|\delta \tilde  \theta  \|_{\tilde L^\infty_T B^\frac{d}{2}} 
        + 
        \kappa
        \| \Delta  \delta \tilde  \theta \|_{L^1_T B^\frac{d}{2}} 
        &\leq 
               \tilde C 
        \Big(
               \|\theta_0- \bar{\theta}\|_{B_{2, 1}^\frac{d}{2}}+ 
        \| f_2(\delta \phi,\,\delta \theta) \|_{L^1_T B^\frac{d}{2}}
        \Big),
    \end{aligned}
\end{equation}
where $\tilde C>1$ is a positive constant, which depends only on the dimension $d\geq 1$. Here, we have used the fact that the initial data $\delta \phi_{|t=0}=0$ is null, 
while $\delta\theta_{|t=0} = \theta_0 - \bar{\theta}$. In order to show the boundedness of $\mathcal{L}$ from  $\mathcal{K}_{\chi, T}$ to itself, we need to ensure that $( \delta\tilde \phi,\, \delta \tilde\theta)$ still belongs to $\mathcal{K}_{\chi,T}$, for any forcing data $(\delta \phi,\delta \theta)\in \mathcal{K}_{\chi,T}$. This can be done by showing that the right-hand side of \eqref{est:first-basic-estimate-of-deltaphi-deltatheta} is sufficiently small in terms of $\chi$. The main ingredient to prove this relation will be the next lemma, which allows to control both terms $f_1(\delta \phi,\,\delta \theta)$ and $f_2(\delta \phi,\,\delta \theta)$. 
\begin{lemma}\label{lemma:f1-estimate}
     Let the initial data $(\phi_0,\, \theta_0)$ satisfy the conditions of Theorem \ref{thm:well-posedness} and let $T_\chi \in (0,1]$ such that the following relation on the linearised state variable $\phi_L$ is satisfied
     \begin{equation}\label{condition:phi_L-smaller-than-chi^2}
     \begin{aligned}
       \| \nabla           \phi_L    \|_{\tilde L^2(0,T_\chi;B^{\frac{d}{2}})}
       &+\| \Delta           \phi_L    \|_{\tilde L^2(0,T_\chi;B^{\frac{d}{2}})}
       +
       \| \Delta^2         \phi_L      \|_{L^1(0,T_\chi;B^{\frac{d}{2}})} +
        \\
       &
        +\| \Delta \nabla    \phi_L      \|_{\tilde L^2(0,T_\chi;B^{\frac{d}{2}})}
        +\| \partial_t       \phi_L \|_{\tilde L^2(0,T_\chi;B^{\frac{d}{2}})}+    
        \| \partial_t \nabla \phi_L         \|_{\tilde L^2(0,T_{\chi}; B^\frac{d}{2})}
        \leq \chi^2.
     \end{aligned}
     \end{equation}
     Then for any $T\in (0,T_\chi)$ and any pair $(\delta\phi,\delta\theta)\in \mathcal{K}_{\chi,T}$
     \begin{enumerate}[(a)]
         \item the term $f_1(\delta \theta,\delta \phi)$ satisfies
     \begin{equation*}
     \begin{aligned}
         \| f_1(\delta \theta,\delta \phi) \|_{L^1_TB^\frac{d}{2}}&+
         \| f_1(\delta \theta,\delta \phi) \|_{\tilde L^2_TB^{\frac{d}{2}-1}}
           \\ 
         &\leq 
         R_1
         \left[
            \epsilon \| \Delta \phi_0 \|_{B^{\frac{d}{2}}} + 
            \frac{1}{\bar{\theta}}
            \max\Big\{ 
                1,\frac{1}{\epsilon }  \Big\}
            \left( \| \phi_0 \|_{B^{\frac{d}{2}}} + 1 \right)^4 \chi
             +
            (1+\epsilon) \chi
         \right]
        \chi,
     \end{aligned}
     \end{equation*}
     \item the term $f_2(\delta \theta,\delta \phi)$ fulfills
     \begin{equation*}
     \begin{aligned}     
        \| f_2(\delta \theta,\delta \phi) \|_{L^1_TB^\frac{d}{2}}
        &
        \leq
         R_2 
         \frac{\Big(\| \phi_0 \|_{B^{\frac{d}{2}}} + 1\Big)^4}{\epsilon\bar{\theta}}
         \Big(
            1 + 
            \chi 
         \max
         \Big\{ 
            1,\alpha\epsilon \bar{\theta},\,\alpha^2\epsilon \bar{\theta},\,\epsilon^2\bar{\theta}^2,\,\bar{\theta}^2
         \Big\}
         \Big)
         \chi,
     \end{aligned}
     \end{equation*}
     \end{enumerate}
     where the positive constants $R_1>1$ and $R_2>1$ depend only on the dimension $d$.
\end{lemma}

\smallskip
\noindent 
Motivated by this result, we are in the position to show that $(\delta \tilde \phi,\,\delta \tilde \theta)$ still belongs to $\mathcal{K}_{\chi,T}$ whenever the relation \eqref{smallness-condition} holds true. As might be expected, this requires to impose a sufficiently small $\varepsilon_0$ and a sufficiently small parameter $\chi$, as we will show in the next steps. 
\noindent 
In what follows, we aim to show that the following estimates on the $\mathcal{K}_{\chi, T}$-norms hold true:
\begin{equation}\label{deltaphi-deltatheta-to-show}
    \begin{aligned}
         \|\delta \tilde  \phi \|_{\tilde L^\infty_T B^\frac{d}{2}} 
        + 
        \| \Delta \delta \tilde  \phi \|_{\tilde L^\infty_T B^\frac{d}{2}}
        + 
        \| \Delta^2 \delta \tilde  \phi \|_{L^1_T B^\frac{d}{2}} 
        +
        \| \partial_t \delta \tilde \phi \|_{\tilde L^2_T B^{\frac{d}{2}}}
        +
         \| \partial_t \nabla \delta \tilde \phi \|_{\tilde L^2_T B^{\frac{d}{2}}}
        &\leq \frac{\chi}{2},\\
        \|  \delta \tilde  \theta   \|_{\tilde L^\infty_T B^\frac{d}{2}} +
        \|  \Delta \delta \tilde  \theta    \|_{L^1_T B^\frac{d}{2}} +
        \| \partial_t \delta \tilde \theta  \|_{L^1_T B^{\frac{d}{2}}}
        &\leq \frac{\chi}{2}.
    \end{aligned}
\end{equation}
We begin with the $\delta \tilde \phi$-norms and we first obtain an estimate in terms of the forcing term $f_1(\delta \theta, \delta \phi)$. 
Since the smallness condition \eqref{smallness-condition} implies $\epsilon \bar{\theta}>1$, by invoking inequality \eqref{est:first-basic-estimate-of-deltaphi-deltatheta} we observe that
\begin{equation*}
    \begin{aligned}
        \|\delta \tilde  \phi & \|_{\tilde L^\infty_T B^\frac{d}{2}} 
        + 
        \| \Delta \delta \tilde  \phi \|_{\tilde L^\infty_T B^\frac{d}{2}}+ 
        \| \Delta^2 \delta \tilde  \phi \|_{L^1_T B^\frac{d}{2}} 
        +
        \| \partial_t \delta \tilde \phi \|_{\tilde L^2_T B^{\frac{d}{2}}}
        +
         \| \partial_t \nabla \delta \tilde \phi \|_{\tilde L^2_T B^{\frac{d}{2}}}
        \\
        &\leq 
        \frac{1}{\min\{1,\alpha\}}
        \Big( 
          \|\delta \tilde  \phi  \|_{\tilde L^\infty_T B^\frac{d}{2}} + 
        \alpha 
        \| \Delta \delta \tilde  \phi \|_{\tilde L^\infty_T B^\frac{d}{2}}+ \epsilon \bar{\theta} \| \Delta^2 \delta \tilde  \phi \|_{L^1_T B^\frac{d}{2}}  +
        \| \partial_t \delta \tilde \phi \|_{\tilde L^2_T B^{\frac{d}{2}}} +
        \sqrt{\alpha}
        \| \partial_t \nabla \delta \tilde \phi \|_{\tilde L^2_T B^{\frac{d}{2}}}
        \Big)\\
        &\leq 
        \frac{\tilde C}{\min\{1,\alpha\}}
        \Big( \| f_1(\delta \theta,\delta \phi) \|_{L^1_TB^\frac{d}{2}}+
        \| f_1(\delta \theta,\delta \phi) \|_{\tilde L^2_TB^{\frac{d}{2}-1}}
        \Big)\\
        &\leq 
        \frac{\tilde C}{\min\{1,\alpha, \kappa, \kappa_B\}}
        \Big( \| f_1(\delta \theta,\delta \phi) \|_{L^1_TB^\frac{d}{2}}+
        \| f_1(\delta \theta,\delta \phi) \|_{\tilde L^2_TB^{\frac{d}{2}-1}}
        \Big).
    \end{aligned}
\end{equation*}    
Next,
we combine part (a) of Lemma \ref{lemma:f1-estimate} together with the smallness condition in \eqref{smallness-condition}, to obtain
\begin{equation}\label{estimate-of-f1-in-terms-of-chi-and-eps}
\begin{aligned}
    \|\delta \tilde  \phi  \|_{\tilde L^\infty_T B^\frac{d}{2}} &+ 
        \| \Delta \delta \tilde  \phi \|_{\tilde L^\infty_T B^\frac{d}{2}}+ 
        \| \Delta^2 \delta \tilde  \phi \|_{L^1_T B^\frac{d}{2}} 
        +
        \| \partial_t \delta \tilde \phi \|_{\tilde L^2_T B^{\frac{d}{2}}}
        +
         \| \partial_t \nabla \delta \tilde \phi \|_{\tilde L^2_T B^{\frac{d}{2}}}\\
     &\leq 
    \frac{\tilde C R_1}{\min\{1,\alpha, \kappa, \kappa_B\}}
    \Big(\varepsilon_0\min\{1, \alpha, \kappa,\kappa_B\}(1+\chi) + (1+\epsilon)\chi
    \Big)\chi.
\end{aligned}
\end{equation}
We aim to show that the right-hand side of the above inequality is indeed bounded by $\chi/2$, for sufficiently small $\varepsilon_0$ and $\chi$. To this end, we first impose
\begin{equation}\label{range-of-epsilon0}
\varepsilon_0 \leq \frac{1}{8\tilde C \max\{R_1, R_2\}}
\end{equation}
 (we will soon exploit this relation, however we shall remark immediately that it depends uniquely on the dimension $d\geq 1$, as it was for $\tilde C$ in \eqref{est:first-basic-estimate-of-deltaphi-deltatheta} and $R_1$,  $R_2$ in Lemma \ref{lemma:f1-estimate}). 
Furthermore, we set the following range for the small parameter $\chi>0$:
\begin{equation}\label{range-of-chi}
     \| \theta_0 - \bar{\theta} \|_{B^\frac{d}{2}} 
    <  
     \frac{\min\{1,\alpha,\kappa, \kappa_B\}}{4\tilde C}\chi 
    <
    \left(
    \min 
    \Big\{
        \epsilon^2,\, 
        \frac{1 }{1+\epsilon}
    \Big\}
    \frac{
         \varepsilon_0\min\{1,\alpha,\kappa, \kappa_B\}
    }{\bar{\theta}(1+\alpha)}
    \right)^2,
\end{equation}
which restricts also the lifespan $T$ through \eqref{condition:phi_L-smaller-than-chi^2}. We remark that   this range of $\chi$ is well defined because of the second smallness condition in \eqref{smallness-condition}). In particular, we have that $\chi<1$, since
\begin{equation*}
    \chi < 
    \underbrace{
    \min\Big\{
        \epsilon^2,\, 
        \frac{1 }{1+\epsilon}
    \Big\}^2}_{<1}
    \frac{\varepsilon_0^2}{\bar{\theta}^2(1+\alpha)^2}
    4\tilde C
    \min\{1, \alpha, \kappa, \kappa_B\}
    <
    4\tilde C\varepsilon_0^2
    <1.
\end{equation*}
Coming back to our estimate, we plug the above relations into  \eqref{estimate-of-f1-in-terms-of-chi-and-eps}, to gather that
\begin{equation}\label{i-dont-know-how-to-call-it}
    \begin{aligned}
        \|\delta \tilde  \phi \|_{\tilde L^\infty_T B^\frac{d}{2}} 
        &+ 
        \| \Delta \delta \tilde  \phi \|_{\tilde L^\infty_T B^\frac{d}{2}}
        + 
        \| \Delta^2 \delta \tilde  \phi \|_{L^1_T B^\frac{d}{2}} 
        +
        \| \partial_t \delta \tilde \phi \|_{\tilde L^2_T B^{\frac{d}{2}}}
        +
         \| \partial_t \nabla \delta \tilde \phi \|_{\tilde L^2_T B^{\frac{d}{2}}}
        \\
        &\leq 
        \frac{\tilde C R_1 }{\min\{1,\alpha, \kappa, \kappa_B\}}
        \Big(
            \varepsilon_0
            \min\{1, \alpha, \kappa,\kappa_B\}(1+\chi) + (1+\epsilon) \chi
        \Big)\chi
        \\
        &\leq
         \frac{\tilde C}{\min\{1,\alpha,\kappa, \kappa_B\}}
         R_1
         \bigg( 
             \frac{\min\{1, \alpha, \kappa,\kappa_B\}}{8\tilde C\max\{R_1,R_2\}}(1+1) +  (1+\epsilon)\chi
         \bigg)\chi \\
         &\leq 
        \frac{\chi}{4}
         + 
         \frac{\tilde CR_1 (1+\epsilon)}{\min\{1,\alpha,\kappa, \kappa_B\}}
         \chi^2 .
    \end{aligned}
\end{equation}  
Next, we show that $\chi\leq \min\{1,\alpha,\kappa, \kappa_B\}/(4\tilde CR_1 (1+\epsilon))$, because of \eqref{range-of-chi}: 
\begin{equation*}
\begin{aligned}
    \chi
    &\leq 
     \min\Big\{
        \epsilon^2,\, 
        \frac{1 }{1+\epsilon}
    \Big\}^2
    \frac{4\tilde C\varepsilon_0^2}{\bar{\theta}^2(1+\alpha)^2}
    \min\{1, \alpha, \kappa, \kappa_B\}\\
    &\leq 
    \frac{1 }{(1+\epsilon)}
    4\tilde C\varepsilon_0^2
    \min\{1, \alpha, \kappa, \kappa_B\},\\
    &\leq 
    \frac{1 }{(1+\epsilon)}
    \frac{4\tilde C}{64\tilde C^2 \max\{R_1, R_2\}^2}
    \min\{1, \alpha, \kappa, \kappa_B\}
    \leq 
    \frac{1}{4\tilde CR_1 (1+\epsilon)}\min\{1, \alpha, \kappa, \kappa_B\},
\end{aligned}
\end{equation*}
Hence, plugging this inequality into \eqref{i-dont-know-how-to-call-it}, we finally deduce that
\begin{equation*}
        \|\delta \tilde  \phi \|_{\tilde L^\infty_T B^\frac{d}{2}} 
        + 
        \| \Delta \delta \tilde  \phi \|_{\tilde L^\infty_T B^\frac{d}{2}}
        + 
        \| \Delta^2 \delta \tilde  \phi \|_{L^1_T B^\frac{d}{2}} 
        +
        \| \partial_t \delta \tilde \phi \|_{\tilde L^2_T B^{\frac{d}{2}}}
        +
         \| \partial_t \nabla \delta \tilde \phi \|_{\tilde L^2_T B^{\frac{d}{2}}}
        \leq 
        \frac{\chi}{4} + \frac{\chi}{4} = \frac{\chi}{2},
\end{equation*}
which concludes the estimates of $\delta \tilde \phi$ in \eqref{deltaphi-deltatheta-to-show}. 

\noindent
We can now bring our attention to $\delta \tilde \theta$ and we apply the last inequality of \eqref{est:first-basic-estimate-of-deltaphi-deltatheta}: 
\begin{equation*}
    \begin{aligned}
        \|  \delta \tilde  \theta   &\|_{\tilde L^\infty_T B^\frac{d}{2}} +
        \|  \Delta \delta \tilde  \theta    \|_{L^1_T B^\frac{d}{2}} +
        \| \partial_t \delta \tilde \theta  \|_{L^1_T B^{\frac{d}{2}}} \\
        &\leq 
        \frac{1}{\min\{1,\kappa , \kappa_B\}}
        \Big(
            \kappa_B
            \|  \delta \tilde  \theta           \|_{\tilde L^\infty_T B^\frac{d}{2}} +
            \kappa
            \|  \Delta \delta \tilde  \theta    \|_{L^1 B^\frac{d}{2}} +
            \kappa_B
            \| \partial_t \delta \tilde \theta  \|_{ L^1_T B^{\frac{d}{2}}}
        \Big)
        \\
        &\leq 
       \frac{\tilde C}{\min\{1,\alpha,\kappa, \kappa_B\}} 
            \Big(
                        \| \theta_0 - \bar{\theta} \|_{B^\frac{d}{2}}+
            \| f_2(\delta \theta,\delta \phi)   \|_{L^1_TB^\frac{d}{2}} 
            \Big)
    \end{aligned}
\end{equation*}
Next, thanks to Lemma \ref{lemma:f1-estimate}, part (b),
\begin{equation}\label{almost-finished-for-theta}
    \begin{aligned}
         \|  \delta \tilde  \theta   &\|_{\tilde L^\infty_T B^\frac{d}{2}} +
        \|  \Delta \delta \tilde  \theta    \|_{L^1_T B^\frac{d}{2}} +
        \| \partial_t \delta \tilde \theta  \|_{L^1_T B^{\frac{d}{2}}}\\
        &\leq 
         \frac{\tilde C}{\min\{1,\alpha,\kappa, \kappa_B\}}
        \Big(
                      \| \theta_0 - \bar{\theta} \|_{B^\frac{d}{2}}+
         R_2\varepsilon_0 \min\{1,\alpha,\,\kappa, \kappa_B\}
         \Big( 
             1 +\chi 
         \max
         \Big\{ 
           1,\alpha\epsilon \bar{\theta},\,\alpha^2\epsilon \bar{\theta},\,\epsilon^2\bar{\theta}^2,\,\bar{\theta}^2
         \Big\}
         \Big)\chi 
        \Big)
    \end{aligned}
\end{equation}
and our aim is therefore to show that the right-hand side is bounded by $\chi/2$. From the range of the parameter $\chi$ in \eqref{range-of-chi}, one has 
\begin{equation}\label{estimates-on-chi-for-final-estimate}
        \| \theta - \theta_0 \|_{B^\frac{d}{2}}
        \leq 
        \frac{\min\{1, \alpha, \kappa, \kappa_B\}}{4\tilde C}
        \chi
        \quad
        \text{and}
        \quad 
        \chi 
         \max
         \Big\{ 
           1,\alpha\epsilon \bar{\theta},\,\alpha^2\epsilon \bar{\theta},\,\epsilon^2\bar{\theta}^2,\,\bar{\theta}^2
         \Big\}<1.
\end{equation}
The last inequality can be shown analysing each component (remember that all constants $\bar{\theta}$, $\tilde C$, $R_1$ and $R_2$ are larger than $1$, as well as $4\tilde C\varepsilon_0^2<1$ from \eqref{range-of-epsilon0}):
\begin{alignat*}{4}
    \chi \alpha \epsilon \bar{\theta} 
    &<  
    \min\Big\{
        \epsilon^\frac{3}{2},\, 
        \frac{\sqrt{\epsilon} }{1+\epsilon}
    \Big\}^2
    \frac{\alpha}{\bar{\theta}(1+\alpha)^2}4\tilde C\varepsilon_0^2
    &&<1,\\
    \chi \alpha^2 \epsilon \bar{\theta} 
    &<  
    \min\Big\{
        \epsilon^\frac{3}{2},\, 
        \frac{\sqrt{\epsilon} }{1+\epsilon}
    \Big\}^2
    \frac{\alpha^2}{\bar{\theta}(1+\alpha)^2}4\tilde C\varepsilon_0^2
    &&<1,\\
    \chi  \epsilon^2 \bar{\theta}^2 
    &<  
    \min\Big\{
        \epsilon^3,\, 
        \frac{\epsilon }{1+\epsilon}
    \Big\}^2
    \frac{\alpha}{\bar{\theta}(1+\alpha)^2}4\tilde C\varepsilon_0^2
    &&<1,\\
    \chi  \bar{\theta}^2 
    &<  
    \min\Big\{
        \epsilon^2,\, 
        \frac{1 }{1+\epsilon}
    \Big\}^2
    \frac{1}{ (1+\alpha)^2}4\tilde C\varepsilon_0^2
    &&<1.
\end{alignat*}
Plugging \eqref{range-of-epsilon0} and \eqref{estimates-on-chi-for-final-estimate}  into \eqref{almost-finished-for-theta}, we deduce that
\begin{equation*}
         \|  \delta \tilde  \theta   \|_{\tilde L^\infty_T B^\frac{d}{2}} +
        \|  \Delta \delta \tilde  \theta    \|_{L^1_T B^\frac{d}{2}} +
        \| \partial_t \delta \tilde \theta  \|_{L^1_T B^{\frac{d}{2}}}
        \leq 
        {\color{red} \frac{\chi}{4}} +
        \frac{\tilde C R_2}{\min\{1,\alpha,\,\kappa, \kappa_B\}}
        \frac{\min\{1,\alpha,\,\kappa, \kappa_B\}}{8\tilde CR_2}
        \Big( 1 + 1\Big)\chi  \leq \frac{\chi}{2},
\end{equation*}
which concludes the estimate for $\delta \tilde \theta$ in \eqref{deltaphi-deltatheta-to-show}. 
Summarising, we have shown that $(\delta \tilde \phi,\,\delta \tilde \theta)$ belongs to $\mathcal{K}_{\chi,T}$ (namely the operator $\mathcal{L}$ maps $\mathcal{K}_{\chi, T}$ into itself), as long as \eqref{range-of-epsilon0} and \eqref{range-of-chi}  are satisfied.

\noindent
Our estimates were based in  Lemma \ref{lemma:f1-estimate}, which 
we shall now prove.

\begin{proof}[Proof of Lemma \ref{lemma:f1-estimate}, part (a)]
    This part of the proof handles any term related to $f_1(\delta \theta,\delta \phi)$ in \eqref{def:f1-function}. In the subsequent analysis, we denote by $C$ a generic positive constants, which is harmless to our overall estimates (namely it depends only on the dimension $d$) and may vary from line to line. Specific dependence will be pointed out if necessary. 
    We further recall the the parameter $\chi>0$ of the space $\mathcal{K}_{T,\chi}$ in \eqref{def:KchiT} is taken sufficiently small, in particular $\chi<1$ (cf.~\eqref{range-of-chi}).
    We provide any detail to estimate the norm in $L^1_T B^{d/2}$ of $f_1(\delta \theta,\delta \phi)$, while we only sketch the estimates of the norm in $\tilde L^2_T B^{d/2-1}$ (since the procedure is analogous). 
    
    \noindent
    We first mention that the product operates continuously between the following function spaces
    \begin{equation}\label{product-between-chemin-lerner}
        \begin{alignedat}{4}
            &\tilde L^p_TB^\frac{d}{2} \times \tilde L^q_TB^\frac{d}{2} 
            &&\to  \tilde L^r_T B^\frac{d}{2},\\
            &\tilde L^p_TB^\frac{d}{2} \times \tilde L^q_TB^{\frac{d}{2}-1} 
            &&\to  \tilde L^r_T B^{\frac{d}{2}-1},
        \end{alignedat} 
        \quad \text{for any}\quad 1\leq \frac{1}{r}:=\frac{1}{p}+ \frac{1}{q},
    \end{equation}
    where the constant of continuity depends only on the dimension $d$.
    We develop the first term of $f_1(\delta \theta,\delta \phi)$ in in \eqref{def:f1-function}, by means of  
    \begin{equation*}
        \epsilon \Delta (\delta \theta \Delta (\phi_L + \delta \phi)) 
        = 
        \epsilon \Delta \delta \theta (\Delta \phi_L + \Delta \delta \phi) + 2\epsilon \nabla \delta \theta \cdot ( \Delta \nabla\phi_L +  \Delta \nabla \delta \phi) + 
        \epsilon 
        \delta \theta (\Delta^2\phi_L + \Delta^2 \delta \phi).
    \end{equation*}
    Thanks to the first continuity relation in \eqref{product-between-chemin-lerner}, with $(p,\,q,\,r)=(1,\infty,1)$ and $(p,\,q,\,r)=(2,2,1)$, we get
    \begin{equation*}
    \begin{aligned}
        \| \epsilon \Delta &(\delta \theta \Delta (\phi_L + \delta \phi) \|_{L^1_T B^\frac{d}{2}} 
        \leq 
        C
        \epsilon 
        \Big\{ 
            \| \Delta \delta \theta \|_{L^1_T B^\frac{d}{2}} 
            \big(
                \| \Delta        \phi_L \|_{\tilde L^\infty_T B^\frac{d}{2}} + 
                \| \Delta \delta \phi   \|_{\tilde L^\infty_T B^\frac{d}{2}}
            \big) 
            + \\ &+ 
            \| \nabla  \delta \theta \|_{\tilde L^2_T B^\frac{d}{2}} 
            \big(
                \| \Delta  \nabla \phi_L        \|_{\tilde L^2_T B^\frac{d}{2}} + 
                \| \Delta  \nabla \delta \phi   \|_{\tilde L^2_T B^\frac{d}{2}}
            \big) 
            + 
             \|   \delta \theta \|_{\tilde L^\infty_T B^\frac{d}{2}}
            \big(
                \| \Delta^2   \phi_L        \|_{L^1_T B^\frac{d}{2}} + 
                \| \Delta^2 \delta \phi     \|_{L^1_T B^\frac{d}{2}}
            \big) 
        \Big\}.
    \end{aligned}
    \end{equation*}
    We are now in the condition to invoke the smallness relation \eqref{condition:phi_L-smaller-than-chi^2} on the linear solution $\phi_L$, which implies the inequality 
    $\| \Delta  \nabla \phi_L        \|_{\tilde L^2_T B^{d/2}}+
     \| \Delta^2   \phi_L        \|_{L^1_T B^{d/2}}\leq \chi^2$.
    Furthermore, from Lemma \ref{lemma:a-priori-estimates-for-phi}, one has that
    $   \| \Delta        \phi_L \|_{\tilde L^\infty_T B^{d/2}}\leq 
        C\| \Delta       \phi_0 \|_{ B^{d/2}}$. Finally, 
    we recall that $(\delta \phi,\,\delta \theta)$ belongs to $\mathcal{K}_{\chi,\,T}$ in \eqref{def:KchiT}, which means $ \| \Delta \delta \theta \|_{L^1_T B^{d/2}}  +
    \| \nabla  \delta \theta \|_{\tilde L^2_T B^{d/2}} + \|   \delta \theta \|_{\tilde L^\infty_T B^{d/2}}\leq \chi$. These last relations tell us that
    \begin{align*}
        \| \epsilon \Delta (\delta \theta \Delta (\phi_L + \delta \phi) \|_{L^1_T B^\frac{d}{2}} 
        \leq 
        C\epsilon 
        \Big\{ \chi ( \| \Delta \phi_0 \|_{B^{\frac{d}{2}}} + \chi ) + 
        2\chi ( \chi^2 + \chi ) \Big\}
        \leq 
        C\epsilon \| \Delta \phi_0 \|_{B^{\frac{d}{2}}} \chi + C\epsilon\chi^2.
    \end{align*}
    We shall highlight that the above constant $C$ depends only on the dimension $d$. An an analogous estimate holds true,  when replaying the norm in $L^1_T B^\frac{d}{2}$ with the one in $\tilde L^2_T B^{\frac{d}{2}-1}$:
    \begin{equation*}
    \begin{aligned}
        \| \epsilon \Delta &(\delta \theta \Delta (\phi_L + \delta \phi) \|_{\tilde L^2_T B^{\frac{d}{2}-1}} 
        \leq 
        C
        \epsilon 
        \Big\{ 
            \| \Delta \delta \theta \|_{\tilde L^2_T B^{\frac{d}{2}-1}} 
            \big(
                \| \Delta        \phi_L \|_{\tilde L^\infty_T B^\frac{d}{2}} + 
                \| \Delta \delta \phi   \|_{\tilde L^\infty_T B^\frac{d}{2}}
            \big) 
            + \\ &+ 
            \| \nabla  \delta \theta \|_{\tilde L^\infty_T B^{\frac{d}{2}-1}} 
            \big(
                \| \Delta  \nabla \phi_L        \|_{\tilde L^2_T B^\frac{d}{2}} + 
                \| \Delta  \nabla \delta \phi   \|_{\tilde L^2_T B^\frac{d}{2}}
            \big) 
            + 
             \|   \delta \theta \|_{\tilde L^\infty_T B^\frac{d}{2}}
            \big(
                \| \Delta^2   \phi_L        \|_{\tilde L^2_T B^{\frac{d}{2}-1}} + 
                \| \Delta^2 \delta \phi     \|_{\tilde L^2_T B^{\frac{d}{2}-1}}
            \big) 
        \Big\}\\
    &\leq C\epsilon \| \Delta \phi_0 \|_{B^{\frac{d}{2}}} \chi + C\epsilon\chi^2.
    \end{aligned}
    \end{equation*}
    This last estimate is in particular achieved through the following interpolations and Bernstein inequalities 
    \begin{equation*}
    \begin{aligned}
         \| \Delta \delta \theta \|_{\tilde L^2_T B^{\frac{d}{2}-1}} 
         \leq C  \| \nabla \delta \theta \|_{\tilde L^2_T B^{\frac{d}{2}}} 
         \leq C
         \| \delta \theta \|_{\tilde L^\infty_T B^{\frac{d}{2}}}^\frac{1}{2}
         \| \Delta \delta \theta \|_{L^1_T B^{\frac{d}{2}}}^\frac{1}{2}
         &\leq C\chi,\\
         \| \Delta^2   \phi_L        \|_{\tilde L^2_T B^{\frac{d}{2}-1}}  
         \leq C
         \| \nabla \Delta   \phi_L        \|_{\tilde L^2_T B^{\frac{d}{2}}}
         &\leq C\chi^2,\\
         \| \Delta^2 \delta \phi     \|_{\tilde L^2_T B^{\frac{d}{2}-1}}
         \leq C
         \| \nabla \Delta   \delta \phi         \|_{\tilde L^2_T B^{\frac{d}{2}}}
         \leq C
         \|  \Delta   \delta \phi         \|_{\tilde L^\infty_T B^{\frac{d}{2}}}^\frac{1}{2}
         \|  \Delta^2   \delta \phi         \|_{L^1_T B^{\frac{d}{2}}}^\frac{1}{2}
         &\leq C\chi.
    \end{aligned}
    \end{equation*}
    Next, we aim to estimate the remaining term of $f_1(\delta \phi, \delta \theta)$ in \eqref{def:f1-function}. The argument relies on the following decomposition
    \begin{equation}\label{eq:def-I1toI6}
        \Delta 
         \left( 
        \frac{1}{\epsilon(\bar{\theta} + \delta \theta) }
        \left( 
            ((\phi_L + \delta \phi)^2-1)(\phi_L + \delta \phi) + 
                \delta \theta^2
            (\phi_L + \delta \phi)
        \right)
        \right)
        = \sum_{n=1}^8 \mathcal{I}_n,
    \end{equation}
    where each term $\mathcal{I}_n$, $n=1,\dots,8$, corresponds to
    \begin{equation}\label{def-of-all-I}
    \begin{aligned}
        \mathcal{I}_1 &:= \Delta \delta \theta\frac{\delta \theta (\delta \theta +2 \bar{\theta} ) }{\epsilon\bar{\theta}^2(\bar{\theta} + \delta\theta )^2}
        \left( 
            ((\phi_L + \delta \phi)^2-1)(\phi_L + \delta \phi) + 
                (\delta \theta)^2
            (\phi_L + \delta \phi)
        \right),\\
        \mathcal{I}_2 &:= - \frac{\Delta \delta \theta }{\epsilon \bar{\theta}^2}
        \left( 
            ((\phi_L + \delta \phi)^2-1)(\phi_L + \delta \phi) +
                (\delta \theta)^2
            (\phi_L + \delta \phi)
        \right),\\
        \mathcal{I}_3 &:= 2\frac{\delta \theta (\delta \theta +2 \bar{\theta} ) }{\epsilon\bar{\theta}^2(\bar{\theta} + \delta\theta )^2}\nabla \delta \theta\cdot \nabla 
        \left( 
            ((\phi_L + \delta \phi)^2-1)(\phi_L + \delta \phi) + 
                 (\delta \theta)^2
            (\phi_L + \delta \phi)
        \right),\\
        \mathcal{I}_4 &:=
         -
         2\frac{1}{\epsilon\bar{\theta}^2}\nabla \delta \theta \cdot 
         \nabla \left( 
            ((\phi_L + \delta \phi)^2-1)(\phi_L + \delta \phi) + 
                 (\delta \theta)^2
            (\phi_L + \delta \phi)
        \right),\\
        \mathcal{I}_5 &:= 
        - \frac{\delta \theta}{\epsilon \bar{\theta}(\bar{\theta} + \delta \theta)}\Delta \left( 
            ((\phi_L + \delta \phi)^2-1)(\phi_L + \delta \phi) + 
                 (\delta \theta)^2
            (\phi_L + \delta \phi)
        \right),\\
        \mathcal{I}_6 &:= 
         \frac{1}{\epsilon \bar{\theta}}\Delta \left( 
            ((\phi_L + \delta \phi)^2-1)(\phi_L + \delta \phi) + 
                (\delta \theta)^2
            (\phi_L + \delta \phi)
        \right),\\
        \mathcal{I}_7 &:= -2|\nabla \delta \theta|^2\frac{\delta \theta (\delta \theta^2 +3\bar{\theta}^2 + 3\bar{\theta}\delta \theta ) }{\epsilon\bar{\theta}^3(\bar{\theta} + \delta\theta )^3}
        \left( 
            ((\phi_L + \delta \phi)^2-1)(\phi_L + \delta \phi) + 
                 (\delta \theta)^2
            (\phi_L + \delta \phi)
        \right)\\
        \mathcal{I}_8 &:= 2|\nabla \delta \theta|^2\frac{1 }{\epsilon\bar{\theta}^3}
        \left( 
            ((\phi_L + \delta \phi)^2-1)(\phi_L + \delta \phi) + 
                 (\delta \theta)^2
            (\phi_L + \delta \phi)
        \right).
    \end{aligned}
    \end{equation}
    In particular, any term related to $1/(\varepsilon(\bar{\theta} + \delta \theta))$ has been split into a constant (cf. $\mathcal{I}_2$, $\mathcal{I}_4$ and $\mathcal{I}_6$) 
    and a function depending on $\delta \theta$ which vanishes whenever $\delta \theta \equiv 0$ (cf. $\mathcal{I}_1$, $\mathcal{I}_3$ and $\mathcal{I}_6$). This choice will turn out useful when applying some estimates related to the composition of functions between Besov spaces (cf. Lemma \ref{lemma:action_of_smooth_function_on_B}).
    
    \smallskip
    \noindent
    We first address the estimate of $\mathcal{I}_1$. As in \eqref{product-between-chemin-lerner}, we invoke the continuity of the product from  
    $\tilde L^\infty_T B^{d/2}\times L^1_T B^{d/2}$ to $L^1_T  B^{d/2}$, from  
    $\tilde L^\infty_T B^{d/2}\times \tilde L^2_T B^{d/2-1}$ to $\tilde L^2_T  B^{d/2-1}$, and finally the algebra structure of $\tilde L^\infty_T B^{d/2}$. These relations ensures that
    \begin{equation*}
        \begin{aligned}
            \| \mathcal{I}_1 \|_{L^1_T B^\frac{d}{2}}&+
            \| \mathcal{I}_1 \|_{\tilde L^2_T B^{\frac{d}{2}-1}}
            \leq 
            C
            \big(
            \| \Delta \delta \theta \|_{L^1_T B^\frac{d}{2}}+
            \| \Delta \delta \theta \|_{\tilde L^2_T B^{\frac{d}{2}-1}}
            \big)
            \left\| 
                \frac{\delta \theta (\delta \theta +2 \bar{\theta}) }{\epsilon \bar{\theta}^2(\bar{\theta} + \delta\theta )^2}
            \right\|_{\tilde L^\infty_T B^\frac{d}{2}} 
            \Big\{ 
                \Big( 
                     \| \phi_L          \|_{\tilde L^\infty_T B^\frac{d}{2}}^2 +
                     \\
                &+ \| \delta \phi    \|_{\tilde L^\infty_T B^\frac{d}{2}}^2 
                      +
                      1
                \Big)
                \Big(
                     \| \phi_L          \|_{\tilde L^\infty_T B^\frac{d}{2}} +
                     \| \delta \phi    \|_{\tilde L^\infty_T B^\frac{d}{2}}
                \Big)
                +
                      \| \delta \theta \|_{\tilde L^\infty_T B^\frac{d}{2}}^2
                \Big(
                     \| \phi_L          \|_{\tilde L^\infty_T B^\frac{d}{2}} +
                      \| \delta \phi    \|_{\tilde L^\infty_T B^\frac{d}{2}}
                \Big)
            \Big\} \\
            &\leq 
            \frac{C}{\epsilon\bar{\theta}^2}
            \chi 
            \left\| 
                \frac{\delta \theta (\delta \theta +2 \bar{\theta}) }{(\bar{\theta} + \delta\theta )^2}
            \right\|_{\tilde L^\infty_T B^\frac{d}{2}} 
            \Big\{ 
                (\| \phi_0 \|_{B^\frac{d}{2}}^2 + \chi^2 + 1 )
                (\| \phi_0 \|_{B^\frac{d}{2}}+ \chi)+
                \chi^2
                (\| \phi_0 \|_{B^\frac{d}{2}}+\chi )
            \Big\}\\
            &\leq 
            \frac{C}{\epsilon\bar{\theta}^2}
            \chi 
            \left\| 
                \frac{\delta \theta (\delta \theta +2 \bar{\theta}) }{(\bar{\theta} + \delta\theta )^2}
            \right\|_{\tilde L^\infty_T B^\frac{d}{2}}
            \big(
                \| \phi_0  \|_{B^{\frac{d}{2}}}^2+1
            \big)
            \big(
                \| \phi_0         \|_{B^{\frac{d}{2}}}+\chi
            \big).
        \end{aligned}
    \end{equation*}
    To deal with the remaining norm, we introduce the smooth function $h:(0,\infty)\to \mathbb{R}$ with $h(x) = x(x + 2\bar{\theta} )/(\bar{\theta} +x)^2$ and we estimate the composition $h\circ \delta \theta$ in $\tilde L^\infty_T B^{d/2}$. As a key ingredient, we make use of Lemma \ref{lemma:action_of_smooth_function_on_B}, which states that for a general smooth function $h$, with $h(0) = 0$, 
    \begin{equation*}
        \| h\circ \delta \theta \|_{\tilde L^\infty_T B^\frac{d}{2}} 
        \leq 
        C \max_{l \in \{0,\dots, [d/2]+1\}}\Big(  \| \delta \theta \|_{L^\infty_TL^\infty}^l \sup_{|x| \leq C\| \delta \theta \|_{L^\infty_TL^\infty}} |h^{(l+1)}(x)|\Big)\| \delta \theta \|_{\tilde L^\infty_T B^\frac{d}{2}},
    \end{equation*}
    where the constant $C$ depends only on the dimension $d$. From the embedding 
    $B^{d/2} \hookrightarrow L^\infty(\mathbb{R}^d)$, one gather that $\| \delta \theta \|_{L^\infty_TL^\infty}\leq C\chi$, while the derivatives of $h$ are given by $h'(x) = 2\bar{\theta}^2/(\bar{\theta}+x)^3 $ and  
    $h^{(l+1)}(x) = (-1)^{l}(l+2)!\bar{\theta}^2/(\bar{\theta}+x)^{l+3}$, for any 
    $l\in  \mathbb{N}$. We deduce that
    \begin{equation}\label{estimate:comp_h_deltatheta}
    \begin{aligned}
        \| h\circ \delta \theta \|_{\tilde L^\infty_T B^\frac{d}{2}} 
        &\leq 
        C \max_{l \in \{0,\dots, [d/2]+1\}}
        \bigg(  
                \chi^l \sup_{|x| \leq C\chi }\left|2(l+2)!\frac{\bar{\theta}^2}{(x+\bar{\theta})^{l+3}}\right|
        \bigg)\| \delta \theta \|_{\tilde L^\infty_T B^\frac{d}{2}}
        \leq C \chi,
    \end{aligned}
    \end{equation}
    where we stress out that the final constant $C>0$ depends only on the dimension. Collecting the above estimates and remarking that $\bar{\theta}\geq 1$ from \eqref{smallness-condition}, we conclude that 
    \begin{equation*}
         \| \mathcal{I}_1 \|_{L^1_T B^\frac{d}{2}}+
        \| \mathcal{I}_1 \|_{\tilde L^2_T B^{\frac{d}{2}-1}}
         \leq 
         \frac{C}{\epsilon \bar{\theta}^2}
            \big(
                \| \phi_0  \|_{B^{\frac{d}{2}}}^2+1
            \big)
            \big(
                \| \phi_0         \|_{B^{\frac{d}{2}}}+\chi
            \big)\chi^2
            \leq
            \frac{C}{\bar{\theta}}
            \max\Big\{1, \frac{1}{\varepsilon}\Big\}
            \big(
                \| \phi_0  \|_{B^{\frac{d}{2}}}+1
            \big)^4
            \chi^2.
    \end{equation*}
    We now tackle those terms associated to $\mathcal{I}_2$ in \eqref{eq:def-I1toI6}. First we observe that
    \begin{align*}
        \| \mathcal{I}_2 \|_{L^1_T B^\frac{d}{2}} &+ 
        \| \mathcal{I}_2 \|_{\tilde L^2_T B^\frac{d}{2}}
        \leq 
        \frac{C}{\epsilon \bar{\theta}}
        \big(
            \| \Delta \delta \theta \|_{L^1_T B^\frac{d}{2}}+
            \| \Delta \delta \theta \|_{\tilde L^2_T B^{\frac{d}{2}-1}}
        \big)
        \Big( 
            (
                \| \phi_L       \|_{\tilde L^\infty_T B^\frac{d}{2}}^2 + 
                \| \delta \phi  \|_{\tilde L^\infty_T B^\frac{d}{2}}^2 + 
        \\
        &+1
            )( 
                \| \phi_L       \|_{\tilde L^\infty_T B^\frac{d}{2}} + 
                \| \delta \phi  \|_{\tilde L^\infty_T B^\frac{d}{2}}) 
            + 
        \| \delta \theta \|_{\tilde L^\infty_T B^\frac{d}{2}}^2 
        (
            \|\phi_L        \|_{\tilde L^\infty_T B^\frac{d}{2}} +  
            \| \delta \phi  \|_{\tilde L^\infty_T B^\frac{d}{2}})
        \Big)\\
        &\leq 
        \frac{C}{\epsilon \bar{\theta}}
        \chi  
        \Big\{
            ( \| \phi_0 \|_{B^{\frac{d}{2}}}^2 + \chi^2 + 1 )(\| \phi_0 \|_{B^{\frac{d}{2}}} + \chi) + \chi^2(\| \phi_0 \|_{B^{\frac{d}{2}}} + \chi)
        \Big\},
    \end{align*}
    therefore
    \begin{equation*}
        \| \mathcal{I}_2 \|_{L^1_T B^\frac{d}{2}}+
        \| \mathcal{I}_2 \|_{\tilde L^2_T B^{\frac{d}{2}-1}}
         \leq 
         \frac{C}{\epsilon \bar{\theta}}
         \big(
                \| \phi_0  \|_{B^{\frac{d}{2}}}^2+1
            \big)
            \big(
                \| \phi_0         \|_{B^{\frac{d}{2}}}+\chi
            \big)\chi^2
            \leq
            \frac{C}{\bar{\theta}}
            \max\Big\{1, \frac{1}{\varepsilon}\Big\}
            \big(
                \| \phi_0  \|_{B^{\frac{d}{2}}}+1
            \big)^4
            \chi^2.
    \end{equation*}
    Let us now turn to the estimate of $\mathcal{I}_3$ in \eqref{eq:def-I1toI6}:
    \begin{equation*}
        \begin{aligned}
            \| \mathcal{I}_3 \|_{L^1_T B^\frac{d}{2}}
            &+
            \| \mathcal{I}_3 \|_{\tilde L^2_T B^{\frac{d}{2}-1}}
            \leq 
            C
            \left\| 
                \frac{\delta \theta (\delta \theta +2 \bar{\theta} ) }{\epsilon \bar{\theta}^2(\bar{\theta} + \delta\theta )^2}
            \right\|_{\tilde L^\infty_T B^\frac{d}{2}} 
            \Big(
                \| \nabla \delta \theta \|_{\tilde L^2_T B^\frac{d}{2}} +
                \| \nabla \delta \theta \|_{\tilde L^\infty_T B^{\frac{d}{2}-1}}
            \Big)
            \Big( 
                      \|  \phi_L         \|_{\tilde L^\infty_T B^\frac{d}{2}}^2 +
                      \\
                      &
                        +
                      \| \delta \phi     \|_{\tilde L^\infty_T B^\frac{d}{2}}^2 
                        +
                      \| \delta \theta \|_{\tilde L^\infty_T B^\frac{d}{2}}^2
                      +
                      1 
            \Big)
            \Big(
                      \| \nabla  \phi_L          \|_{\tilde L^2_T B^\frac{d}{2}} +
                      \| \nabla  \delta \phi     \|_{\tilde L^2_T B^\frac{d}{2}} +
                      \| \nabla  \delta \theta   \|_{\tilde L^2_T B^\frac{d}{2}}
            \Big)\\
            &\leq 
            \frac{2}{\epsilon \bar{\theta}^2}
            \chi 
            \left\| 
                \frac{\delta \theta (\delta \theta +2 \bar{\theta}) }{(\bar{\theta} + \delta\theta )^2}
            \right\|_{\tilde L^\infty_T B^\frac{d}{2}} 
            \Big\{ 
                (\| \phi_0 \|_{B^{\frac{d}{2}}}^2 + \chi^2 + 1 )
                (\chi^2 + \chi )
            \Big\},
        \end{aligned}
    \end{equation*}
    which together with \eqref{estimate:comp_h_deltatheta} implies that
    \begin{equation*}
         \| \mathcal{I}_3 \|_{L^1_T B^\frac{d}{2}}+
         \| \mathcal{I}_3 \|_{\tilde L^2_T B^{\frac{d}{2}-1}} 
        \leq 
         \frac{C}{\epsilon \bar{\theta}^2}
         \left( \| \phi_0 \|_{B^{\frac{d}{2}}} + 1 \right)^4 \chi^2
         \leq
            \frac{C}{\bar{\theta}}
            \max\Big\{1, \frac{1}{\varepsilon}\Big\}
            \big(
                \| \phi_0  \|_{B^{\frac{d}{2}}}+1
            \big)^4
            \chi^2.
    \end{equation*}
    Concerning $\mathcal{I}_4$, we observe that this term satisfies similar properties as  $\mathcal{I}_2$ and $\mathcal{I}_3$, namely
    \begin{align*}
        \| \mathcal{I}_4        \|_{L^1_T B^\frac{d}{2}}+
        \| \mathcal{I}_4        \|_{\tilde L^2_T B^{\frac{d}{2}-1}}
        &\leq 
        \frac{C}{\epsilon \bar{\theta}}
        (
            \|  \nabla  \delta \theta \|_{\tilde L^2_T B^\frac{d}{2}}+
            \| \nabla \delta \theta \|_{\tilde L^\infty_T B^{\frac{d}{2}-1}}
        )
        \big(
            \| \phi_L           \|_{\tilde L^\infty_T B^\frac{d}{2}}^2 + 
            \| \delta \phi      \|_{\tilde L^\infty_T B^\frac{d}{2}}^2 +
             \\
        &
        \hspace{1cm}
            +
            \| \delta \theta    \|_{\tilde L^\infty_T B^\frac{d}{2}}^2+ 
            1
        \big)
        \big( 
            \| \nabla \phi_L        \|_{\tilde L^2_T B^\frac{d}{2}} + 
            \| \nabla \delta \phi   \|_{\tilde L^2_T B^\frac{d}{2}}+
            \| \nabla \delta \theta \|_{\tilde L^2_T B^\frac{d}{2}}
        \big) 
        \\ 
        &
        \leq 
        \frac{C}{\epsilon \bar{\theta}}
        \chi  
        ( \| \phi_0 \|_{B^{\frac{d}{2}}}^2 + \chi^2 + 1 )(\chi^2+\chi),
    \end{align*}
    which eventually leads to
    \begin{equation*}
        \| \mathcal{I}_4 \|_{L^1_T B^\frac{d}{2}}+
        \| \mathcal{I}_4        \|_{\tilde L^2_T B^{\frac{d}{2}-1}}
         \leq 
         \frac{C}{\epsilon \bar{\theta}}
         \Big( \| \phi_0 \|_{B^\frac{d}{2}} + 1 \Big)^4 \chi^2
         \leq
            \frac{C}{\bar{\theta}}
            \max\Big\{1, \frac{1}{\varepsilon}\Big\}
            \big(
                \| \phi_0  \|_{B^{\frac{d}{2}}}+1
            \big)^4
            \chi^2.
    \end{equation*}
    We next analyse $\mathcal{I}_5$ and $\mathcal{I}_6$. Redefining the function $h:(0,\infty)\to \mathbb{R}$ by $h(x) = x/(\bar{\theta} +x)$, we apply Lemma \ref{lemma:action_of_smooth_function_on_B} to obtain
    \begin{equation}\label{estimate:comp_h_deltatheta2}
    \begin{aligned}
        \| h\circ \delta \theta \|_{\tilde L^\infty_T B^\frac{d}{2}} 
        &\leq 
        C \max_{l \in \{0,\dots, [d/2]+1\}} 
        \bigg( 
            \| \delta \theta \|_{L^\infty_T L^\infty}^l 
            \sup_{|x| \leq C\| \delta \theta \|_{L^\infty_TL^\infty}} 
                \bigg| \frac{\bar{\theta} (l+1)!}{(\bar{\theta} + x)^{l+2}}
                \bigg|
        \bigg)
        \| \delta \theta \|_{\tilde L^\infty_T B^\frac{d}{2}}
        \leq C \chi.
    \end{aligned}
    \end{equation}
    With an analogous procedure as the one used in the previous estimate, we arrive at
    \begin{align*}
        \| \mathcal{I}_5 \|_{L^1_T B^\frac{d}{2}}+
        \| \mathcal{I}_5 \|_{\tilde L^2_T B^{\frac{d}{2}-1}}
        &\leq 
        \frac{C}{\epsilon \bar{\theta}}
        \| h\circ  \delta \theta \|_{\tilde L^\infty_T B^\frac{d}{2}}
        \Big\{
        \big(
            \| \phi_L           \|_{\tilde L^\infty_T B^\frac{d}{2}}^2 + 
            \| \delta \phi      \|_{\tilde L^\infty_T B^\frac{d}{2}}^2 +
            \| \delta \theta    \|_{\tilde L^\infty_T B^\frac{d}{2}}^2+ 
            1
        \big)
        \big( 
            \| \Delta \phi_L \|_{L^1_T B^\frac{d}{2}} + 
        \\ 
        &+    
            \| \Delta \delta \phi   \|_{L^1_T B^\frac{d}{2}}+
            \| \Delta \delta \theta \|_{L^1_T B^\frac{d}{2}}
        \big) 
        +
        \big(
            \| \nabla \phi_L        \|_{\tilde L^2_T B^\frac{d}{2}}^2 + 
            \| \nabla \delta \phi   \|_{\tilde L^2_T B^\frac{d}{2}}^2 +
            \| \nabla \delta \theta \|_{\tilde L^2_T B^\frac{d}{2}}^2 
        \big)\cdot 
        \\
        &\hspace{3cm}
        \cdot 
        \big( 
            \|   \phi_L        \|_{\tilde L^\infty_T B^\frac{d}{2}} + 
            \|  \delta \phi   \|_{\tilde L^\infty_T B^\frac{d}{2}}+
            \|  \delta \theta \|_{\tilde L^\infty_T B^\frac{d}{2}} +1
        \big)
         \Big\}
        \\ 
        &
        \leq 
        \frac{C}{\epsilon \bar{\theta}}
        \chi
        \Big\{
        ( \| \phi_0 \|_{B^\frac{d}{2}}^2 + \chi^2 + 1 ) (\chi^2 + \chi) + \chi^2
        ( \| \phi_0 \|_{B^\frac{d}{2}}  + \chi  + 1 )
        \Big\}
        \\ 
        &
         \leq
            \frac{C}{\bar{\theta}}
            \max\Big\{1, \frac{1}{\varepsilon}\Big\}
            \big(
                \| \phi_0  \|_{B^{\frac{d}{2}}}+1
            \big)^4
            \chi^2.
    \end{align*}
    We turn to the estimate of the term $\mathcal{I}_6$:
    \begin{equation*}
    \begin{aligned}
        \| \mathcal{I}_6 \|_{L^1_T B^\frac{d}{2}}
        &+
        \| \mathcal{I}_6 \|_{\tilde L^2_T B^{\frac{d}{2}-1}}
         \leq 
         \frac{C}{\epsilon\bar{\theta}}
        \Big\{
             \big(
                \| \phi_L           \|_{\tilde L^\infty_T B^\frac{d}{2}}^2 + 
                \| \delta \phi      \|_{\tilde L^\infty_T B^\frac{d}{2}}^2 +
                \| \delta \theta    \|_{\tilde L^\infty_T B^\frac{d}{2}}^2+ 
                1
            \big)
            \big( 
                \| \Delta \phi_L \|_{L^1_T B^\frac{d}{2}} + 
            \\ 
            &+    
                \| \Delta \delta \phi   \|_{L^1_T B^\frac{d}{2}}+
                \| \Delta \delta \theta \|_{L^1_T B^\frac{d}{2}}
            \big) 
        +
        \big(
            \| \nabla \phi_L        \|_{\tilde L^2_T B^\frac{d}{2}}^2 + 
            \| \nabla \delta \phi   \|_{\tilde L^2_T B^\frac{d}{2}}^2 +
            \| \nabla \delta \theta \|_{\tilde L^2_T B^\frac{d}{2}}^2+ 
            1
        \big)\cdot 
        \\
        &\hspace{3cm}
        \cdot 
        \big( 
            \| \nabla \phi_L        \|_{\tilde L^2_T B^\frac{d}{2}} + 
            \| \nabla \delta \phi   \|_{\tilde L^2_T B^\frac{d}{2}}+
            \| \nabla \delta \theta \|_{\tilde L^2_T B^\frac{d}{2}}
        \big) 
        \Big\}
        \\&
        \leq 
        \frac{C}{\epsilon \bar{\theta}}
        ( \| \phi_0 \|_{B^\frac{d}{2}}^2 + \chi^2 + 1 )\chi^2
         \leq
            \frac{C}{\bar{\theta}}
            \max\Big\{1, \frac{1}{\varepsilon}\Big\}
            \big(
                \| \phi_0  \|_{B^{\frac{d}{2}}}+1
            \big)^4
            \chi^2.
    \end{aligned}
    \end{equation*}
    We finally address the last terms $\mathcal{I}_7$ and $\mathcal{I}_8$ in \eqref{def-of-all-I}. First
    \begin{equation*}
        \begin{aligned}
             \| \mathcal{I}_7 \|_{L^1_T B^\frac{d}{2}}
            &+
            \| \mathcal{I}_7 \|_{\tilde L^2_T B^{\frac{d}{2}-1}}
            \leq 
            C
             \| \nabla \delta \theta \|_{\tilde L^2_T B^\frac{d}{2}}
            \Big(
                \| \nabla \delta \theta \|_{\tilde L^2_T B^\frac{d}{2}}+
                \| \nabla \delta \theta \|_{\tilde L^\infty_T B^{\frac{d}{2}-1}}
            \Big)
            \left\| 
               \frac{\delta \theta (\delta \theta^2 +3\bar{\theta}^2 + 3\bar{\theta}\delta \theta ) }{\epsilon\bar{\theta}^3(\bar{\theta} + \delta\theta )^3}
            \right\|_{\tilde L^\infty_T B^\frac{d}{2}} \cdot \\
            &\cdot 
                \Big( 
                      \|  \phi_L         \|_{\tilde L^\infty_T B^\frac{d}{2}}^2 
                        +
                      \| \delta \phi     \|_{\tilde L^\infty_T B^\frac{d}{2}}^2 +
                       \| \delta \theta  \|_{\tilde L^\infty_T B^\frac{d}{2}}^2     
                      +
                      1 
                \Big)
                \Big(
                      \| \phi_L          \|_{\tilde L^\infty_T B^\frac{d}{2}} +
                      \|  \delta \phi     \|_{\tilde L^\infty_T B^\frac{d}{2}}
                \Big) \\
                &\leq
            \frac{C}{\bar{\theta}}
            \max\Big\{1, \frac{1}{\varepsilon}\Big\}
            \big(
                \| \phi_0  \|_{B^{\frac{d}{2}}}+1
            \big)^4
            \chi^2,
        \end{aligned}
    \end{equation*}
    furthermore
     \begin{equation*}
        \begin{aligned}
             \| \mathcal{I}_8 \|_{L^1_T B^\frac{d}{2}}
            &+
            \| \mathcal{I}_8 \|_{\tilde L^2_T B^{\frac{d}{2}-1}}
            \leq 
            \frac{C}{\epsilon \bar{\theta}^3}
             \| \nabla \delta \theta \|_{\tilde L^2_T B^\frac{d}{2}}
            \Big(
                \| \nabla \delta \theta \|_{\tilde L^2_T B^\frac{d}{2}}+
                \| \nabla \delta \theta \|_{\tilde L^\infty_T B^{\frac{d}{2}-1}}
            \Big)
                \Big( 
                      \|  \phi_L         \|_{\tilde L^\infty_T B^\frac{d}{2}}^2 
                        + \\
                        &+
                      \| \delta \phi     \|_{\tilde L^\infty_T B^\frac{d}{2}}^2 +
                       \| \delta \theta  \|_{\tilde L^\infty_T B^\frac{d}{2}}^2     
                      +
                      1 
                \Big)
                \Big(
                      \| \phi_L          \|_{\tilde L^\infty_T B^\frac{d}{2}} +
                      \|  \delta \phi     \|_{\tilde L^\infty_T B^\frac{d}{2}}
                \Big) \leq 
            \frac{C}{\bar{\theta}}
            \max\Big\{1, \frac{1}{\varepsilon}\Big\}
            \big(
                \| \phi_0  \|_{B^{\frac{d}{2}}}+1
            \big)^4
            \chi^2.
        \end{aligned}
    \end{equation*}
    Collecting the above estimates, we eventually achieve the part (a) of Lemma \ref{lemma:f1-estimate}. 
\end{proof}

\begin{proof}[Proof of Lemma \ref{lemma:f1-estimate}, part (b)]
    The proof follows essentially the lines of part (a). We separately control each term of $f_2(\delta \phi,\,\delta \theta)$ in \eqref{def:f2}. As first term we bound
    \begin{align*}
        \| \alpha (\partial_t \phi_L +  \partial_t \delta \phi)^2 \|_{L^1_T B^{\frac{d}{2}}}
        \leq 2\alpha \| \partial_t \phi_L \|_{\tilde L^2_T B^{\frac{d}{2}}}^2 +  
        2\alpha \| \partial_t \delta \phi \|_{\tilde L^2_T B^{\frac{d}{2}}}^2
        \leq 2\alpha \chi^4 +2 \alpha \chi^2 \leq 4 \alpha \chi^2.
    \end{align*}
    Now let us turn to the estimate of
    \begin{equation*}
    \begin{aligned}
        \epsilon \| &(\bar{\theta} + \delta \theta)(\nabla \partial_t  \phi_L + \nabla \partial_t \delta \phi )\cdot (\nabla \phi_L + \nabla \delta \phi) \|_{L^1_T B^\frac{d}{2}} 
        \\
        &\leq 
        C\epsilon
        \big( 
            \bar{\theta}  + \| \delta \theta \|_{\tilde L^\infty_T B^\frac{d}{2}}
        \big)
        \big(
            \| \partial_t \nabla \phi_L         \|_{\tilde L^2_T B^\frac{d}{2}} + 
            \| \partial_t \nabla \delta \phi    \|_{\tilde L^2_T B^\frac{d}{2}}
        \big)
        \big(
            \| \nabla \phi_L          \|_{\tilde L^2_T B^\frac{d}{2}} + 
            \| \nabla \delta \phi     \|_{\tilde L^2_T B^\frac{d}{2}}
        \big)\\
        &\leq 
        C\epsilon
        ( \bar{\theta} + \chi)
        (\chi^2 + \chi)^2\leq C\epsilon \bar{\theta} \chi^2.
    \end{aligned}
    \end{equation*}
    Recalling the definition of $f_2(\delta \phi,\,\delta \theta)$ in \eqref{def:f2}, we now address the estimate of 
    \begin{equation*}
    \begin{aligned}
        (\bar{\theta}+\delta\theta)\partial_t \bigg[\frac{1}{\epsilon (\bar{\theta}+\delta\theta)^2} \left( \frac{1}{4}(1-|\phi_L + \delta \phi|^2)^2 +{ \frac{(\delta \theta)^3}{3}}(\phi_L + \delta \phi)^2 \right)
-\frac{(\delta \theta)^2}{\epsilon(\bar{\theta}+\delta\theta)}(\phi_L + \delta \phi)^2\bigg] =
    \sum_{n=1}^5\mathcal{J}_n,
    \end{aligned}    
    \end{equation*}
    where each term $\mathcal{J}_n$ is defined through
    \begin{align*}
        \mathcal{J}_1 &:= (\bar{\theta} + \delta \theta)
        \frac{2\delta \theta ((\delta \theta)^2 + 3 \bar{\theta} \delta \theta  + 3\bar{\theta}^2)}{\epsilon \bar{\theta}^3(\bar{\theta} + \delta \theta)^3}
        \partial_t \delta \theta 
        \left( \frac{1}{4}(1-|\phi_L + \delta \phi|^2)^2 +
        { \frac{(\delta \theta)^3}{3}}
        (\phi_L + \delta \phi)^2 \right),\\
        \mathcal{J}_2 &:= 
         -\frac{2(\bar{\theta} + \delta \theta)}
         {\epsilon\bar{\theta}^3}\partial_t \delta \theta 
         \left( \frac{1}{4}(1-|\phi_L + \delta \phi|^2)^2 + 
         { \frac{(\delta \theta)^3}{3}}
         (\phi_L + \delta \phi)^2 \right),\\
         \mathcal{J}_3 &:= 
         \frac{1}{\epsilon (\bar{\theta} + \delta \theta)}
         \Big\{
            (|\phi_L + \delta \phi|^2-1)
            (\phi_L + \delta \phi)(\partial_t \phi_L + \partial_t \delta \phi)
            + \\ & 
            \hspace{1cm}
            +
            {(\delta \theta)^2 } \partial_t  \delta \theta(\phi_L + \delta \phi)^2
            + 
           2{\frac{(\delta \theta)^3}{3}}
            (\phi_L + \delta \phi)(\partial_t \phi_L + \partial_t \delta \phi)
        \Big\},\\
        \mathcal{J}_4 &:=
        -
            \frac{\delta \theta (2\bar{\theta} + \delta \theta) }{2\epsilon (\bar{\theta} + \delta \theta)}
        \partial_t \delta \theta 
        (\phi_L + \delta \phi)^2,\;
        \\
        \mathcal{J}_5&:=
        -2\frac{(\delta \theta)^2}{\epsilon} (\phi_L + \delta \phi)(\partial_t \phi_L + \partial_t \delta \phi).
    \end{align*}
    First, we introduce the function $h(x) = 2x(x^2+3\bar{\theta}x + 3\bar{\theta}^2)/(\bar{\theta} +x)^3$. From inequality \eqref{ineq:composition-besov} and the embedding $B^{d/2} \hookrightarrow L^\infty(\mathbb{R}^d)$, one gathers that
    \begin{equation}
    \begin{aligned}
        \| h\circ \delta \theta \|_{\tilde L^\infty_T B^\frac{d}{2}} 
        &\leq 
        C  \max_{l \in \{0,\dots, [d/2]+1\}}
        \Big(
        \| \delta \theta \|_{L^\infty_TL^\infty}^l \sup_{|x| \leq C\| \delta \theta \|_{L^\infty_TL^\infty}} |h^{(l+1)}(x)|\Big)\| \delta \theta \|_{\tilde L^\infty_T B^\frac{d}{2}}.
    \end{aligned}
    \end{equation}
    Since for any $l\in \mathbb{N}_0$ one has
    $$h^{(l+1)}(x) = (l+1)!(-1)^l \frac{ 18 x^2-12(l-1)x\bar{\theta} +(l-7 )l\bar{\theta}^2}{(x+\bar{\theta})^{4+l}},$$
    we eventually gather that $
        \| h\circ \delta \theta \|_{\tilde L^\infty_T B^\frac{d}{2}} 
        \leq   C\chi . $ Therefore
    \begin{align*}
        \| \mathcal{J}_1 &\|_{L^1_T B^{\frac{d}{2}}}
        \leq \frac{C}{\varepsilon} 
        \big(\bar{\theta} + \| \delta \theta \|_{\tilde L^\infty_T B^\frac{d}{2}}\big)
        \| h\circ \delta \theta \|_{\tilde L^\infty_T B^{\frac{d}{2}}}
        \| \partial_t \delta \theta \|_{L^1_T B^{\frac{d}{2}}} 
        \Big\{
            1 + 
            \| \phi_L \|_{\tilde L^\infty_T B^{\frac{d}{2}}}^4 
            +
            \| \delta \phi      \|_{\tilde L^\infty_T B^{\frac{d}{2}}}^4 +\\ 
            &+  \| \delta \theta \|_{\tilde L^\infty_T B^{\frac{d}{2}}}^3 
            \big(
            \| \phi_L \|_{L^\infty_T B^{\frac{d}{2}}}^2 
            +
            \| \delta \phi      \|_{\tilde L^\infty_T B^{\frac{d}{2}}}^2 
            \big)
        \Big\}
        \leq 
        \frac{C}{\varepsilon} 
        (\bar{\theta}+\chi)\chi^2
        \big\{1+ \| \phi_0 \|_{B^{\frac{d}{2}}}^4 +\chi^4 + 
        \chi^3( \| \phi_0 \|_{B^{\frac{d}{2}}}^2+\chi )\big\}\\
        &\leq 
        C\frac{\bar{\theta}}{\varepsilon}(1+ \| \phi_0 \|_{B^{\frac{d}{2}}}^4)\chi^2
        \leq 
        \frac{C}{\bar{\theta}\varepsilon}(1+ \| \phi_0 \|_{B^{\frac{d}{2}}})^4\bar{\theta}^2\chi^2.
    \end{align*}
    Now let us turn to the estimate of $\mathcal{J}_2$, namely
    \begin{align*}
        \| \mathcal{J}_2 &\|_{L^1_T B^{\frac{d}{2}}}
        \leq 
        \frac{C}{\varepsilon\bar{\theta}^3} 
        \big(\bar{\theta} + \| \delta \theta \|_{\tilde L^\infty_T B^\frac{d}{2}}\big)
        \| \partial_t \delta \theta \|_{L^1_T B^{\frac{d}{2}}} 
         \Big\{
            1 + 
            \| \phi_L \|_{\tilde L^\infty_T B^{\frac{d}{2}}}^4 
            +
            \| \delta \phi      \|_{\tilde L^\infty_T B^{\frac{d}{2}}}^4 +\\ 
            &+  \| \delta \theta \|_{\tilde L^\infty_T B^{\frac{d}{2}}}^3 
            \big(
            \| \phi_L \|_{L^\infty_T B^{\frac{d}{2}}}^2 
            +
            \| \delta \phi      \|_{\tilde L^\infty_T B^{\frac{d}{2}}}^2 
            \big)
        \Big\}
        \leq 
        \frac{C}{\varepsilon\bar{\theta}^2}
        \Big(1+ \| \phi_0 \|_{B^{\frac{d}{2}}}^4\Big) \chi^2
        \leq 
        \frac{C}{\varepsilon\bar{\theta}}
        \Big(1+ \| \phi_0 \|_{B^{\frac{d}{2}}}\Big)^4 \chi^2,
    \end{align*}
    where in the last inequality we have used the fact that $\bar{\theta}\geq 1$. 
    The estimate of $\mathcal{J}_3$ is performed as follows
    \begin{align*}
         \| \mathcal{J}_3 \|_{L^1_T B^{\frac{d}{2}}}
         \leq 
         \frac{C}{\epsilon}
         \Big(
            \frac{1}{\bar{\theta}} + 
            \left\|
                \frac{\delta \theta }{\bar{\theta}+ \delta \theta} 
            \right\|_{\tilde L^\infty_T B^{\frac{d}{2}}}
        \Big)
        \Big\{
            \Big(
                \|          \phi_L  \|_{\tilde L^\infty_t B^{\frac{d}{2}}}^2 + 
                \| \delta   \phi    \|_{\tilde L^\infty_t B^{\frac{d}{2}}}^2 + 1
            \Big)
            \Big(
                \|                      \phi_L  \|_{\tilde L^\infty_T B^{\frac{d}{2}}}
               +
                \|              \delta  \phi    \|_{\tilde L^\infty_T B^{\frac{d}{2}}} 
            \Big)\cdot 
            \\
            \cdot 
            \Big(
                  \|  \partial_t          \phi_L      \|_{L^1_T B^{\frac{d}{2}}} 
              +
                \|  \partial_t  \delta  \phi        \|_{L^1_T B^{\frac{d}{2}}} 
            \Big)
            +
            \| \delta \theta \|_{\tilde L^\infty_T B^\frac{d}{2}}^2
            \| \partial_t \delta \theta \|_{L^1_T B^\frac{d}{2}}
            \Big(
                \|          \phi_L  \|_{\tilde L^\infty_T B^{\frac{d}{2}}}^2 + 
                \| \delta   \phi    \|_{\tilde L^\infty_T B^{\frac{d}{2}}}^2 
            \Big)
            + \\ +
            \| \delta \theta    \|_{\tilde L^\infty_T B^{\frac{d}{2}}}^3 
            \Big(
                \|          \phi_L  \|_{\tilde L^\infty_T B^{\frac{d}{2}}} + 
                \| \delta   \phi    \|_{\tilde L^\infty_T B^{\frac{d}{2}}} 
            \Big)
            \Big(
                \|  \partial_t          \phi_L     \|_{L^1_T B^{\frac{d}{2}}} +
                \|  \partial_t  \delta  \phi        \|_{L^1_T B^{\frac{d}{2}}}
            \Big)\Big\},
    \end{align*}
    from which we deduce
    \begin{align*}
        \| \mathcal{J}_3 \|_{L^1_T B^{\frac{d}{2}}}
        &\leq 
        \frac{C}{\epsilon}
        \Big(
            \frac{1}{\bar{\theta}} + 
            \chi
        \Big)
        \Big\{
            \big(
                \| \phi_0 \|_{B^{\frac{d}{2}}}^2 + \chi^2 + 1
            \big)
            \big(
                \| \phi_0 \|_{B^{\frac{d}{2}}} + \chi
            \big)
            \big( \chi^2 + \chi \big) 
            + \\ 
            &+ 
            \chi 
            \big(
                \| \phi_0 \|_{B^{\frac{d}{2}}}^2 + \chi^2
            \big)
            +
            \chi^3 
            \big(
                \| \phi_0 \|_{B^{\frac{d}{2}}} + \chi
            \big)
            \big(
               \chi + \chi^2
            \big)\\
            &\leq 
            C
            \Big(
                \frac{1}{\epsilon \bar{\theta}} + 
                \frac{\chi}{\epsilon}
            \Big)
            \Big(
                \| \phi_0 \|_{B^{\frac{d}{2}}} +1 
            \Big)^4
            \chi.
    \end{align*}
    Let us now turn our attention to $\mathcal{J}_4$
    \begin{align*}
         \| \mathcal{J}_4 \|_{L^1_T B^{\frac{d}{2}}} 
         &\leq C 
         \Big\| 
            \frac{\delta \theta (\bar{\theta} +2 \delta \theta) }{2\epsilon (\bar{\theta} + \delta \theta)}
         \Big\|_{\tilde L^\infty_T B^\frac{d}{2}}
         \| \partial_t \delta \theta\|_{L^1_T B^\frac{d}{2}} 
         \Big(
            \|      \phi_L  \|_{\tilde L^\infty_T B^\frac{d}{2}}^2 + 
            \| \delta \phi  \|_{\tilde L^\infty_T B^\frac{d}{2}}^2 
         \Big)\\
         &\leq 
         \frac{C}{\epsilon}
         \big(
            \| \phi_0 \|_{B^{\frac{d}{2}}}^2 + 
            1
         \big)
         \chi^2,
    \end{align*}
    while for $\mathcal{J}_5$ we get
    \begin{align*}
         \| \mathcal{J}_5 \|_{L^1_T B^{\frac{d}{2}}} 
         &\leq 
         \frac{C}{\epsilon \bar{\theta}}
         \|     \delta \theta       \|_{\tilde L^\infty_T B^\frac{d}{2}}^2 
         \Big(
            \|      \phi_L                  \|_{\tilde L^\infty_T B^\frac{d}{2}} + 
            \| \delta \phi                  \|_{\tilde L^\infty_T B^\frac{d}{2}} 
         \Big)
         \Big(
            \|  \partial_t   \phi_L         \|_{L^1_T B^\frac{d}{2}} + 
            \|  \partial_t  \delta \phi     \|_{L^1_T B^\frac{d}{2}} 
         \Big)\\
         &\leq 
         \frac{C}{\bar{\theta}\epsilon}
         \big(
            \| \phi_0 \|_{B^{\frac{d}{2}+2}} + 
            1
         \big)^4
         \chi^3.
    \end{align*}
    We shall now address the last term of $f_2(\delta \phi,\delta \theta)$ in \eqref{def:f2}, namely
    \begin{equation*}
        \left|
            \alpha\nabla\partial_t \phi -
            \nabla 
            \left(  
                \epsilon (\bar{\theta} + \delta \theta )(\Delta \phi_L + \Delta \delta\phi)  - 
                \frac{1}{\epsilon(\bar{\theta}+\delta\theta )}( (\phi_L + \delta \phi)^2-1)(\phi_L + \delta \phi) 
                + 
                \frac 23\frac{(\delta \theta)^3}{\epsilon(\bar{\theta}+\delta\theta )} 
                (\phi_L+ \delta \phi) 
            \right)
        \right|^2.
    \end{equation*}
    We begin with the first term $\alpha \nabla \partial_t  \phi = \alpha \nabla \partial_t \phi_L +\alpha \nabla \partial_t \delta \phi $. Taking the square from the absolute value, 
    \begin{equation*}
        \| |\alpha\nabla (\partial_t \phi_L + \partial_t \delta \phi)|^2 \|_{L^1_T B^\frac{d}{2}}
        \leq 
        C\alpha^2 
        \Big(
            \| \nabla  \partial_t \phi_L          \|_{\tilde L^2_T B^\frac{d}{2}}^2+
            \| \nabla  \partial_t \delta \phi     \|_{\tilde L^2_T B^\frac{d}{2}}^2
        \Big)
    \end{equation*}
    which implies that
    \begin{equation*}
        \| |\alpha\nabla (\partial_t \phi_L + \partial_t \delta \phi)|^2 \|_{L^1_T B^\frac{d}{2}}
        \leq 
        C
        \alpha^2 
        \Big(
            \chi^4+
            \chi^2
        \Big)
        \leq 
        C\alpha^2 \chi^2.
    \end{equation*}
    Next, we handle
    \begin{align*}
        \|  |\epsilon \nabla \delta \theta (\Delta \phi_L + \Delta \delta \phi)|^2 \|_{L^1_T B^\frac{d}{2}}
        &\leq 
        C\epsilon^2 
        \| \nabla \delta \theta \|_{\tilde L^2_T B^\frac{d}{2}}^2
        \Big(
            \| \Delta \phi_L        \|_{\tilde L^2_T B^\frac{d}{2}}^2+
            \| \Delta \delta \phi   \|_{\tilde L^2_T B^\frac{d}{2}}^2
        \Big) \leq 
        C
        \epsilon^2 
        \chi^4.
    \end{align*}
    Next, we address the following term
    \begin{align*}
        \|  |\epsilon \delta \theta (\nabla \Delta \phi_L + \nabla \Delta \delta \phi)|^2 \|_{L^1_T B^\frac{d}{2}}
        &\leq 
        C\epsilon^2 
        \|   \delta \theta \|_{\tilde L^\infty_T B^\frac{d}{2}}^2
        \Big(
            \| \nabla \Delta \phi_L        \|_{L^1_T B^\frac{d}{2}}^2+
            \| \nabla \Delta \delta \phi   \|_{L^1_T B^\frac{d}{2}}^2
        \Big) \leq 
        C
        \epsilon^2 
        \chi^4.
    \end{align*}
    Furthermore
    \begin{align*}
        \Big\|
        \Big|
            \frac{\nabla \delta \theta }{\epsilon(\bar{\theta}+ \delta \theta)^2}
            &( (\phi_L + \delta \phi)^2-1)(\phi_L + \delta \phi)
        \Big|^2
        \Big\|_{L^1_T B^\frac{d}{2}}
        \leq 
        C\frac{1}{\epsilon^2 \bar{\theta}^4}
        \Big(
            1 +  
            \left\|
                \frac{\delta \theta (\bar{\theta}+2\delta \theta ) }{\bar{\theta}^2(\bar{\theta} + \delta \theta)^2}
            \right\|^2_{\tilde L^\infty_T B^\frac{d}{2}}
        \Big)
        \cdot \\ 
        &\cdot 
        \| \nabla \delta \theta \|_{\tilde L^2_T B^\frac{d}{2}}^2
        \Big(
            \| \phi_L  \|_{\tilde L^\infty_T B^\frac{d}{2}}^2 +
            \| \delta \phi  \|_{\tilde L^\infty_T B^\frac{d}{2}}^2 +
            1
        \Big)
        \Big(
            \| \phi_L  \|_{\tilde L^\infty_T B^\frac{d}{2}}^2 +
            \| \delta \phi   \|_{\tilde L^\infty_T B^\frac{d}{2}}^2
        \Big)\\
        &\leq 
        \frac{C}{\epsilon \bar{\theta}}
        \Big( 
            1 +
            \chi^2
        \Big)
        \chi^2
        \Big(
           \| \phi_0 \|_{B^\frac{d}{2}}^2+\chi^2 + 1 
        \Big)
        \Big(
           \| \phi_0 \|_{B^\frac{d}{2}}^2+\chi^2 
        \Big)\\
        &\leq 
        \frac{C}{\bar{\theta} }\max\Big\{ 1, \frac{1}{\epsilon}\Big\}
        \Big(
           \| \phi_0 \|_{B^\frac{d}{2}}+1 
        \Big)^4
        \chi^2,
    \end{align*}
    and
    \begin{align*}
        \Big\|
        \Big|
            &\Big(
                \frac{1}{\epsilon \bar{\theta}} - 
                \frac{\delta \theta }{\epsilon \bar{\theta}(\bar{\theta}+ \delta \theta)}
            \Big)
            \Big( 
                3(\phi_L + \delta \phi)^2 -1 
            \Big)
            \Big(
                \nabla \phi_L + \nabla \delta \phi 
            \Big)
        \Big|^2
        \Big\|_{L^1_T B^\frac{d}{2}}
        \\
        &
        \leq 
        C
        \Big(
          \frac{1}{\epsilon^2 \bar{\theta}^2} + 
          \Big\| 
             \frac{\delta \theta }{\epsilon \bar{\theta}(\bar{\theta}+ \delta \theta)}
          \Big\|_{\tilde L^\infty_T B^\frac{d}{2}}^2
        \Big)
        \Big(
            \|  \phi_L          \|_{\tilde L^\infty_T B^\frac{d}{2}}^2 + 
            \| \delta   \phi    \|_{\tilde L^\infty_T B^\frac{d}{2}}^2 +
            1
        \Big)
        \Big(
            \| \nabla   \phi_L       \|_{\tilde L^2_T B^\frac{d}{2}}^2 + 
            \| \nabla \delta \phi    \|_{\tilde L^2_T B^\frac{d}{2}}^2
        \Big) \\
        &\leq 
        C
        \frac{1}{\epsilon^2 \bar{\theta}^2} 
        \big(
            1+\chi \epsilon^2\bar{\theta}^2
        \big)
        \Big( 
        \| \phi_0 \|_{B^{\frac{d}{2}}}^2 + 
        \chi^2 +
        1
        \Big)\chi^2
        \leq 
        \frac{C}{\bar{\theta} }\max\Big\{ 1, \frac{1}{\epsilon}\Big\}
        \Big(
           \| \phi_0 \|_{B^\frac{d}{2}}+1 
        \Big)^4
        \big(
            1+\chi \epsilon^2\bar{\theta}^2
        \big)
        \chi^2.
    \end{align*}
    Finally, we decompose the remaining term into
    \begin{align*}
            \nabla 
            \Big(
                \frac 23 \frac{(\delta \theta)^3}{\epsilon(\bar{\theta}+\delta \theta)}
                (\phi_L+ \delta \phi)
            \Big) 
            &= 
            \frac{2(\delta \theta)^2(2\delta \theta + 3\bar{\theta})}{3\epsilon (\delta \theta+ \bar{\theta})^2}
            \nabla \delta \theta (\phi_L + \delta \phi) + 
            \frac 23\frac{ (\delta \theta)^3}{\epsilon(\bar{\theta}+\delta \theta)}
            (\nabla \phi_L + \nabla \delta \phi)\\
            &= 
            \frac{2}{3\epsilon} 
            \delta \theta
            \frac{\delta \theta }{\bar{\theta}+\delta \theta}
            \Big(
             3-\frac{\delta \theta}{\bar{\theta} + \delta \theta}
            \Big)
            \nabla \delta \theta (\phi_L + \delta \phi) + 
            \frac{2}{3\epsilon} (\delta \theta)^2\frac{ \delta \theta}{ \bar{\theta}+\delta \theta}
            (\nabla \phi_L + \nabla \delta \phi),
    \end{align*}
    thus, recalling \eqref{estimate:comp_h_deltatheta2}, we gather
    \begin{align*}
        \Big\|
            \Big|
                &\nabla 
                \Big(
                     \frac 23 \frac{(\delta \theta)^3}{\epsilon(\bar{\theta}+\delta \theta)}
                    (\phi_L+ \delta \phi)
                \Big)
            \Big|^2
        \Big\|_{L^1_T B^\frac{d}{2}}
        \leq 
        \frac{C}{\epsilon}
        \Big\{
        \| \delta \theta \|_{\tilde L^\infty_T B^\frac{d}{2}}^2
        \Big\| \frac{\delta \theta}{\bar{\theta} + \delta \theta} \Big\|_{\tilde L^\infty_T B^\frac{d}{2}}^2
        \Big(
            3+
            \Big\| \frac{\delta \theta}{\bar{\theta} + \delta \theta} \Big\|_{\tilde L^\infty_T B^\frac{d}{2}}
        \Big)^2\cdot \\
        &\cdot 
        \| \nabla \delta \theta \|_{\tilde L^2_T B^\frac{d}{2}}^2
        \Big(
            \| \phi_L       \|_{\tilde L^\infty_T B^\frac{d}{2}}^2+
            \| \delta \phi  \|_{\tilde L^\infty_T B^\frac{d}{2}}^2
        \Big) +
       \| \delta \theta \|_{\tilde L^\infty_T B^\frac{d}{2}}^4
        \Big\| \frac{\delta \theta}{\bar{\theta} + \delta \theta} 
        \Big\|_{\tilde L^\infty_T B^\frac{d}{2}}^2
        \Big(
            \| \nabla \phi_L       \|_{\tilde L^2_T B^\frac{d}{2}}^2+
            \| \nabla \delta \phi  \|_{\tilde L^2_T B^\frac{d}{2}}^2
        \Big)
        \Big\}\\
        &\leq 
        \frac{C}{\epsilon} 
        \Big\{
        \chi^6
        ( \|\phi_0 \|_{B^{\frac{d}{2}}}^2 + \chi^2 ) + 
           \chi^6(\chi^4+\chi^2)
        \Big\} \leq 
        \frac{C}{\epsilon\bar{\theta}} 
        \Big(\|\phi_0 \|_{B^{\frac{d}{2}}}+ 1 \Big)^4
         \bar{\theta} 
        \chi^2
        \leq 
        \frac{C}{\bar{\theta} }\max\Big\{ 1, \frac{1}{\epsilon}\Big\}
        \Big(
           \| \phi_0 \|_{B^\frac{d}{2}}+1 
        \Big)^4
        \bar{\theta} \chi^2.
    \end{align*}
    This concludes the proof of the lemma.
\end{proof}

\subsection{Second part of the proof of Theorem \ref{thm:well-posedness}}\label{section:second-part-of-the-proof}

This section concludes the proof of Theorem \ref{thm:well-posedness}. We show that the operator $\mathcal{L}:\mathcal{K}_{\chi, T}\to \mathcal{K}_{\chi, T}$ defined in \eqref{tilde-system} is a contraction, as long as the parameter $\varepsilon_0$ in the smallness condition of \eqref{smallness-condition}  and the perturbation $\chi>0$ in \eqref{def:KchiT} are considered sufficiently small (cf.~relations \eqref{range-of-epsilon0} and \eqref{range-of-chi}). 

\noindent
To better understand this smallness relation, we recall that  $(\delta \tilde \phi,\delta  \tilde \theta)=\mathcal{L}(\delta \phi, \delta \theta)$ is the unique solution of problem \eqref{tilde-system}. Thus we can apply the a-priori estimates of Lemma \ref{lemma:a-priori-estimates-for-phi} and Lemma \ref{lemma:a-priori-estimates-for-theta}, to gather
\begin{equation}\label{eq:first-estimate-contraction}
\begin{aligned}
    \| &\mathcal{L}(\delta \phi_1, \delta \theta_1) 
    - \mathcal{L}(\delta \phi_2, \delta \theta_2) \|_{\mathcal{K}} \leq 
    \tilde C 
    \bigg\{
    \max \Big\{1, \frac{1}{\alpha}   \Big\}
    \Big(
        \|  f_1(\delta \phi_1, \delta \theta_1) -  f_1(\delta \phi_2, \delta \theta_2) \|_{L^1_T B^\frac{d}{2}} + \\
     &\quad
+   \|  f_1(\delta \phi_1, \delta \theta_1) -  f_1(\delta \phi_2, \delta \theta_2) \|_{\tilde L^2_T B^{\frac{d}{2}-1}} 
\Big)
    +
      \max \Big\{\frac{1}{\kappa_B} , \frac{1}{\kappa}\Big\}
     \|  f_2(\delta \phi_1, \delta \theta_1) -  f_2(\delta \phi_2, \delta \theta_2) \|_{L^1_T B^\frac{d}{2}}
     \bigg\},
\end{aligned}
\end{equation}
for any couple $(\delta \phi_1,\delta \theta_1)$ and $(\delta \phi_2,\delta \theta_2)$ in $\mathcal{K}_{\chi, T}$. 
Here the constant $\tilde C>0$ depends only on the dimension and it coincides with the one of \eqref{est:first-basic-estimate-of-deltaphi-deltatheta}. The aim of this section is to prove the following proposition, which unlocks some estimates of the terms related to $f_1$ and $f_2$ on the right-hand side of \eqref{eq:first-estimate-contraction}.
\begin{prop}\label{prop:contraction}
Assume that the hypotheses of Theorem \ref{thm:well-posedness} and Lemma \ref{lemma:f1-estimate} as well as the relations \eqref{range-of-epsilon0} and \eqref{range-of-chi} are satisfied. The following inequalities hold true for any couple $(\delta \phi_1, \delta \theta_1)$ and $(\delta \phi_2, \delta \theta_2)$ in $\mathcal{K}_{\chi, T}$:
    \begin{align*}
        (i&) \;  \|  f_1(\delta \phi_1, \delta \theta_1) -  f_1(\delta \phi_2, \delta \theta_2) \|_{L^1_T B^\frac{d}{2}} 
        +\|  f_1(\delta \phi_1, \delta \theta_1) -  f_1(\delta \phi_2, \delta \theta_2) \|_{\tilde L^2_T B^{\frac{d}{2}-1}} 
        \\
        & \hspace{2cm}
        \leq R_1
        \Big(
            \epsilon (\|\Delta \phi_0 \|_{B^\frac{d}{2}}+\chi) + \frac{1}{\bar \theta}\max\Big\{1,\frac{1}{\epsilon}\Big\} (1+\|\phi_0\|_{B^\frac{d}{2}})^4
        \Big)   \| (\delta \phi_1,\delta \theta_1)- (\delta \phi_2,\delta \theta_2) \|_{\mathcal{K}},\\
        (ii&) \;  \|  f_2(\delta \phi_1, \delta \theta_1) -  f_2(\delta \phi_2, \delta \theta_2) \|_{L^1_T B^\frac{d}{2}} 
        \\
        & \hspace{0.5cm}
       \leq R_2
        \Big(
              \frac{1}{\bar \theta}\max\Big\{1,\frac{1}{\epsilon}\Big\} (1+\|\phi_0\|_{B^\frac{d}{2}})^4+(\epsilon\bar{\theta} + \alpha +\alpha^2)\chi
        \Big) \| (\delta \phi_1,\delta \theta_1)- (\delta \phi_2,\delta \theta_2) \|_{\mathcal{K}},
\end{align*}
for some constants $R_1$ and $R_2$ which depend only on the dimension $d\geq 1$.     
\end{prop}
\noindent
Without loss of generality, we can assume that the constant $R_1$ and $R_2$ coincide with the ones of Lemma \ref{lemma:f1-estimate}. Indeed, they depend only on the dimension $d\geq 1$, therefore one can take the maximum of all of them.

\noindent 
By coupling Proposition \ref{prop:contraction} and the assumptions on $\varepsilon_0$ and $\chi$ (cf. \eqref{range-of-epsilon0} and \eqref{range-of-chi}) we obtain that $\mathcal{L}$ is a contraction, i.e.:
\begin{equation}
    \begin{aligned}
        \| \mathcal{L}(\delta \phi_1, \delta \theta_1) - \mathcal{L}(\delta \phi_2, \delta \theta_2) \|_{\mathcal{K}}\leq \frac{1}{2}  \| (\delta \phi_1,\delta \theta_1)- (\delta \phi_2,\delta \theta_2) \|_{\mathcal{K}}.
    \end{aligned}
\end{equation}
Thus, to conclude our procedure, it remains to prove Proposition \ref{prop:contraction}. This property concludes the proof of Theorem \ref{thm:well-posedness}. Indeed there exists a fixed point $(\delta \phi, \delta \theta)\in \mathcal{K}_{\chi, T}$ for $\mathcal{L}$, more precisely $(\delta \phi, \delta \theta)$ is solution of our original system \eqref{eq:phi-to-analyse}--\eqref{eq:theta-to-analyse}.

\begin{proof}[Proof of  Proposition \ref{prop:contraction}]
In this proof, we limit ourselves to the norm $L^1_TB^{d/2}$ of case $(i)$. 
The estimate of the norm $L^2_TB^{d/2}$ can be derived analogously. The case $(ii)$ can be derived with similar techniques as the ones used in Lemma \ref{lemma:f1-estimate}, combined with the following line of arguments.

\noindent
We  turn our attention to the estimate of each term on $(i)$. 
We reintroduce the compact notation $\phi = \phi_L + \delta \phi$ and $\theta = \bar{\theta}+\delta \theta$. Hence we remark that $\delta \phi_1 -\delta \phi_2 = \phi_1-\phi_2$, as well as $\delta \theta_1 -\delta \theta_2 = \theta_1-\theta_2$. We now recall that $f_1$ in \eqref{def:f1-function} stands for
\begin{equation*}
\begin{aligned}
    f_1(\delta \phi, \delta \theta) &=  
    - \epsilon \Delta (\delta \theta \Delta \phi)  +  
    \Delta 
    \Big( 
        \frac{1}{\epsilon \theta  }
        \Big( 
            (\phi^2 -1 )\phi + 
            \frac{(\delta \theta)^3}{3} \phi
        \Big)
    \Big).
\end{aligned}
\end{equation*}
We can hence split the difference in $(i)$ by means of
\begin{equation}\label{eq:prop-contraction-def-I-II-III}
\begin{aligned}
    f_1(&\delta \phi_1, \delta \theta_1) -f_1(\delta \phi_2, \delta \theta_2)
    = 
    \\
    &=
    \underbrace{-\epsilon \Delta (\delta \theta_1 \Delta \phi_1 ) +  \epsilon \Delta (\delta \theta_2 \Delta \phi_2)}_{=:\mathcal{I}}+
    \underbrace{
    \Delta 
    \Big( 
        \frac{1}{\epsilon \theta_1 }
            (\phi_1^2\!\!-1)\phi_1
    \Big)-
    \Delta 
    \Big( 
        \frac{1}{\epsilon \theta_2 }
            (\phi_2^2\!\!-1)\phi_2
    \Big)}_{=:\mathcal{I}\mathcal{I}} 
    +
    \underbrace{
    \Delta \Big(
    \frac{(\delta \theta_1)^3 \phi_1}{3\epsilon \theta_1 }
     -
     \frac{(\delta \theta_2)^3 \phi_2}{3\epsilon \theta_2 }
    \Big)}_{=:\mathcal{I}\mathcal{I}\mathcal{I}}.
\end{aligned}
\end{equation}
We thus aim to show that each component $\mathcal{I}$, $\mathcal{I}\mathcal{I}$ and $\mathcal{I}\mathcal{I}\mathcal{I}$ is bounded by
\begin{equation}\label{est-prop-final-I-II-III-new}
\begin{aligned}
    \| \mathcal{I}                          \|_{L^1_T B^\frac{d}{2}}
    &\leq C\epsilon
    \Big( \| \Delta \phi_0 \|_{B^{\frac{d}{2}}} + \chi \Big)
    \| (\delta \phi_1,\delta \theta_1)- (\delta \phi_2,\delta \theta_2) \|_{\mathcal{K}},\\
    \| \mathcal{I} \mathcal{I}              \|_{L^1_T B^\frac{d}{2}}
    &\leq 
    \frac{C}{\bar{\theta}}
    \max
    \Big\{
        1,\frac{1}{\epsilon}
    \Big\}
    \big(1+ \|\phi_0 \|_{B^{\frac{d}{2}}}\big)^4
   \|(\delta \phi_1,\delta \theta_1) -(\delta \phi_2,\delta \theta_2)\|_{\mathcal{K}},\\
    \| \mathcal{I} \mathcal{I} \mathcal{I}  \|_{L^1_T B^\frac{d}{2}} 
    &\leq 
    \frac{C}{\bar{\theta}}
    \max
    \Big\{
        1,\frac{1}{\epsilon}
    \Big\}
    \big(1+ \|\phi_0 \|_{B^{\frac{d}{2}}}\big)^4
   \|(\delta \phi_1,\delta \theta_1) -(\delta \phi_2,\delta \theta_2)\|_{\mathcal{K}},
\end{aligned}
\end{equation}
for a suitable constant $C>0$ which depends only on the dimension $d\geq 1$. 
We shall first bound each term $\mathcal{I}$, $\mathcal{I}\mathcal{I}$ and $\mathcal{I}\mathcal{I}\mathcal{I}$, with respect to the topology of $L^1_T B^\frac{d}{2}$. In what follows we only address the first and second inequalities in \eqref{est-prop-final-I-II-III-new}, since the estimate on $\mathcal{I}\mathcal{I}\mathcal{I}$ can be achieved with an analogous procedure as for $\mathcal{I}\mathcal{I}$.

\noindent 
The first term we handle is
\begin{align*}
    \| \mathcal{I} \|_{L^1_T B^\frac{d}{2}}
    &\leq 
    \epsilon
    \|   \Delta ((\delta \theta_1- \delta \theta_2) \Delta (\phi_L \!+\!\delta \phi_1)) 
    \|_{L^1_T B^\frac{d}{2}} + 
    \epsilon 
    \|\Delta (\delta \theta_2 \Delta (\delta \phi_1 - \delta \phi_2))\|_{L^1_T B^\frac{d}{2}} =: \mathcal{I}_1+\mathcal{I}_2.
\end{align*}
We control $\mathcal{I}_1$ by means of
\begin{align*}
    \mathcal{I}_1&\leq 
    C\epsilon\bigg\{
    \| \Delta (\delta \theta_1- \delta \theta_2) \|_{L^1_T B^\frac{d}{2}}
    \Big(
        \| \Delta \phi_L        \|_{\tilde L^\infty_T B^\frac{d}{2}}+
         \| \Delta \delta \phi_1   \|_{\tilde L^\infty_T B^\frac{d}{2}}
    \Big)+
    \| \nabla  (\delta \theta_1- \delta \theta_2) \|_{\tilde L^2_T B^\frac{d}{2}}
    \Big(
        \| \nabla \Delta \phi_L        \|_{\tilde L^2_T B^\frac{d}{2}}
    +\\&\hspace{1cm}+
        \| \nabla \Delta \delta \phi_1   \|_{\tilde L^2_T B^\frac{d}{2}}
    \Big)+
    \| \delta \theta_1- \delta \theta_2 \|_{\tilde L^\infty_T B^\frac{d}{2}}
    \Big(
        \|  \Delta^2 \phi_L        \|_{L^1_T B^\frac{d}{2}}
        +
        \|  \Delta^2 \delta \phi_1 \|_{L^1_T B^\frac{d}{2}}
    \Big)\bigg\}.
\end{align*}
We now remark that $\| \Delta \phi_L  \|_{\tilde L^\infty_T B^{d/2}} \leq \| \Delta \phi_0  \|_{B^{d/2}}$ for any time $T>0$. Moreover, we assume that $T\in [0,T_\chi)$, where $T_\chi$ is defined through the relation
$ \| \nabla \Delta \phi_L  \|_{\tilde L^2(0,{T_\chi}; B^{d/2})}+
\|  \Delta^2 \phi_L  \|_{ L^1(0,{T_\chi}; B^{d/2})}\leq \chi^2$. 
Since $(\delta \phi_1,\, \delta \theta_1)\in \mathcal{K}_{\chi, T}$
\begin{equation*}
    \mathcal{I}_1
    \leq C\epsilon \| (\delta \phi_1,\delta \theta_1)- (\delta \phi_2,\delta \theta_2) \|_{\mathcal{K}} 
    \Big( \| \Delta \phi_0 \|_{B^{\frac{d}{2}}} + \chi + \chi^2\Big).
\end{equation*}
With an analogous procedure, we deal with $\mathcal{I}_2$ as follows
\begin{align*}
    \mathcal{I}_2&\leq 
    \epsilon\bigg\{
    \| \Delta \delta \theta_2                   \|_{L^1_T B^\frac{d}{2}}
    \| \Delta (\delta \phi_1 - \delta \phi_2)   \|_{\tilde L^\infty_TB^\frac{d}{2}} 
    +
    \| \nabla \delta \theta_2                           
    \|_{\tilde L^2_T B^\frac{d}{2}} 
    \| \nabla \Delta (\delta \phi_1 - \delta \phi_2) 
    \|_{\tilde L^2_T B^\frac{d}{2}}
    +\\
    &+
    \| \delta \theta_2  \|_{\tilde L^\infty_T B^\frac{d}{2}} 
    \|  \Delta^2 (\delta \phi_1 - \delta \phi_2)       
    \|_{L^1_T B^\frac{d}{2}}\bigg\}\leq 
    C\varepsilon \chi \| (\delta \phi_1,\delta \theta_1)- (\delta \phi_2,\delta \theta_2) \|_{\mathcal{K}}. 
\end{align*}
Coupling the estimates of $\mathcal{I}_1$ and $\mathcal{I}_2$ leads to the first inequality of \eqref{est-prop-final-I-II-III-new}.
Next, we turn to the estimate of $\mathcal{I}\mathcal{I}$ in \eqref{est-prop-final-I-II-III-new}:
\begin{equation}\label{eq:prop-contraction-est2}
\begin{aligned}
    &\big\|  \mathcal{I}\mathcal{I}
    \big\|_{L^1_T B^\frac{d}{2}}  
    =
    \Big\| 
        \Delta 
        \left(
        \Big(
            \frac{1}{\epsilon \theta_1 }\!-\!
            \frac{1}{\epsilon \theta_2 }
        \Big)
                (\phi_1^2\!\!-1)\phi_1
        \right)
        +
        \Delta 
        \left( 
            \frac{1}{\epsilon \theta_2 }(\phi_1 -\phi_2)
             ( \phi_1^2+ \phi_2^2 +\phi_1\phi_2 -1 )
        \right)
    \Big\|_{L^1_T B^\frac{d}{2}} \\
    &\leq 
    \underbrace{\frac{1}{\epsilon }
    \Big\| 
        \Delta 
        \left(
        \Big(
            \frac{1}{\theta_1 }\!-\!
            \frac{1}{  \theta_2 }
        \Big)
                (\phi_1^2\!\!-1)\phi_1
        \right)
    \Big\|_{L^1_T B^\frac{d}{2}}
    }_{\mathcal{I}\mathcal{I}_1}
    \!+\!
    \underbrace{
    \frac{1}{\epsilon\bar{\theta}}
    \Big\|
        \Delta 
        \left( 
        \Big(
            1 - 
            \frac{\delta \theta_2}{\bar{\theta}+\delta \theta_2}
        \Big)
            (\phi_1 \!-\!\phi_2)
            ( \phi_1^2\!+\! \phi_2^2 \!+\!\phi_1\phi_2 \!-\!1 )
        \right)
    \Big\|_{L^1_T B^\frac{d}{2}}
    }_{\mathcal{I}\mathcal{I}_2}.
\end{aligned}
\end{equation}
The first term is then developed through
\begin{align*}
     \mathcal{I}\mathcal{I}_1   \!\leq \!
    \frac{C}{\epsilon} 
    \bigg\{\!
        \Big\|
            \Delta\Big( \frac{\theta_2-\theta_1}{\theta_1\theta_2}\Big)
        \Big\|_{L^1_T B^\frac{d}{2}} 
        \big(
            \| \phi_1
            &\|_{\tilde L^\infty_T B^\frac{d}{2}}^2 
            \!\!+1
        \big)
        \|\phi_1
        \|_{\tilde L^\infty_T B^\frac{d}{2}} 
         \!\!+\!
        \Big\|
            \nabla \Big( \frac{\theta_2-\theta_1}{\theta_1\theta_2}\Big)
        \Big\|_{\tilde L^2_T B^\frac{d}{2}} 
        \big(
            \|
                \phi_1
            \|_{\tilde L^\infty_T B^\frac{d}{2}}^2  \!\!+ \!1
        \big)
        \|
               \nabla  \phi_1
        \|_{\tilde L^2_T B^\frac{d}{2}}
        \\
        &+ \Big\|
             \frac{\theta_2-\theta_1}{\theta_1\theta_2}
        \Big\|_{\tilde L^\infty_T B^\frac{d}{2}} 
        \big(
            \|
                \phi_1
            \|_{\tilde L^\infty_T B^\frac{d}{2}}  + 1
        \big)
        \big(
        \|
               \Delta  \phi_1
        \|_{ L^1_T B^\frac{d}{2}}+
        \|
               \nabla \phi_1
        \|_{\tilde L^2_T B^\frac{d}{2}}^2
        \big)
    \bigg\}. 
\end{align*}
At this stage, we use the fact that $\phi_1 -\phi_2 = \delta \phi_1 -\delta \phi_2$ and $\theta_1 -\theta_2 = \delta \theta_1 -\delta \theta_2$. Thus, applying Lemma \ref{lemma:action_of_smooth_function_on_B} and recalling that both $(\delta \phi_1,\, \delta \theta_1) $ and $(\delta \phi_2,\, \delta \theta_2) $ belong to $ \mathcal{K}_{\chi, T}$
\begin{align*}
    \mathcal{I}\mathcal{I}_1
    \leq 
    \frac{C}{\epsilon} 
    \bigg\{
        &\frac{1}{\bar{\theta}^2}
        \Big(
            \| \Delta (\delta \theta_1-\delta \theta_2 ) 
            \|_{L^1_T B^\frac{d}{2}} +
            \| \nabla (\delta \theta_1-\delta \theta_2 ) 
            \|_{\tilde L^2_T B^\frac{d}{2}} +
            \|  \delta \theta_1-\delta \theta_2 
            \|_{\tilde L^\infty_T B^\frac{d}{2}} 
        \Big)\cdot \\
        &\cdot 
        \Big(
            \big( \| \phi_0 \|_{B^{\frac{d}{2}}} + \chi\big)^2+1
        \Big)
        \big( 
            \| \phi_0 \|_{B^{\frac{d}{2}}} + \chi
        \big)
         +
         \\
        &+
        \frac{1}{\bar{\theta}^2}
        \Big(
            \| \nabla (\delta \theta_1-\delta \theta_2 ) 
            \|_{\tilde L^2_T B^\frac{d}{2}} +
            \|  \delta \theta_1-\delta \theta_2 
            \|_{\tilde L^\infty_T B^\frac{d}{2}} 
        \Big)
        \Big(
            \big( \| \phi_0 \|_{B^{\frac{d}{2}}} + \chi\big)^2+1
        \Big)
        \big( 
           \chi^2 + \chi
        \big)
        + \\ 
        &+
        \frac{1}{\bar{\theta}^2}
        \big\|  \delta \theta_1- \delta \theta_2
        \big\|_{\tilde L^\infty_T B^\frac{d}{2}} 
        \big(
            \|  \phi_0
            \|_{B^\frac{d}{2}} +\chi + 1
        \big)
        \big(
            \chi^2+\chi + \chi^4+ \chi^2
        \big)
    \bigg\}, 
\end{align*}
which eventually leads to the estimate
\begin{equation}
   \mathcal{I}\mathcal{I}_1
    \leq 
    \frac{C}{\bar{\theta}}
    \max
    \Big\{
        1,\frac{1}{\epsilon}
    \Big\}
    \big(1+ \|\phi_0 \|_{B^{\frac{d}{2}}}\big)^4
    \|(\delta \phi_1,\delta \theta_1) -(\delta \phi_2,\delta \theta_2)\|_{\mathcal{K}}.
\end{equation}
We now continue our estimate of \eqref{eq:prop-contraction-est2} by taking into account $\mathcal{I}\mathcal{I}_2$ in \eqref{eq:prop-contraction-est2}:
\begin{align*}
    \mathcal{I}\mathcal{I}_2
    \leq 
    \frac{C}{\varepsilon\bar{\theta}}
    \bigg\{
     \Big\| \Delta \frac{\delta \theta_2}{\bar{\theta}+ \delta \theta_2} 
     \Big\|_{L^1_T B^\frac{d}{2}}
     \| \phi_1 -\phi_2 \|_{\tilde L^\infty_T B^\frac{d}{2}}
     \Big(
     \|\phi_1 \|_{\tilde L^\infty_T B^\frac{d}{2}}^2+
     \|\phi_2 \|_{\tilde L^\infty_T B^\frac{d}{2}}^2+1
     \Big) 
     + 
     \Big\| \nabla \frac{\delta \theta_2}{\bar{\theta}+ \delta \theta_2} 
     \Big\|_{\tilde L^2_T B^\frac{d}{2}}
      \cdot \\ \cdot 
     \| \nabla (\phi_1 -\phi_2) \|_{\tilde L^2_T B^\frac{d}{2}}
     \Big(
     \|\phi_1 \|_{\tilde L^\infty_T B^\frac{d}{2}}^2+
     \|\phi_2 \|_{\tilde L^\infty_T B^\frac{d}{2}}^2+1
     \Big)+
     \Big\| \nabla \frac{\delta \theta_2}{\bar{\theta}+ \delta \theta_2} 
     \Big\|_{\tilde L^2_T B^\frac{d}{2}}
     \|  \phi_1 -\phi_2 \|_{\tilde L^\infty_T B^\frac{d}{2}}
     \Big(
     \|\phi_1           \|_{\tilde L^\infty_T B^\frac{d}{2}}
     +\\ +
     \|\phi_2           \|_{\tilde L^\infty_T B^\frac{d}{2}}
     \Big)
     \Big(
     \|\nabla \phi_1    \|_{\tilde L^2_T B^\frac{d}{2}}+
     \|\nabla \phi_2    \|_{\tilde L^2_T B^\frac{d}{2}}
     \Big)
      +
     \Big(
     1+    
     \Big\| \frac{\delta \theta_2}{\bar{\theta}+ \delta \theta_2} 
     \Big\|_{\tilde L^\infty_T B^\frac{d}{2}}
     \Big)
     \|  \Delta (\phi_1 -\phi_2) \|_{ L^1_T B^\frac{d}{2}}
     \Big(
     \|\phi_1 \|_{\tilde L^\infty_T B^\frac{d}{2}}^2
     +\\ +
     \|\phi_2 \|_{\tilde L^\infty_T B^\frac{d}{2}}^2+1
     \Big)
      + 
     \Big(
     1+    
     \Big\| \frac{\delta \theta_2}{\bar{\theta}+ \delta \theta_2} 
     \Big\|_{\tilde L^\infty_T B^\frac{d}{2}}
     \Big)
     \|   (\phi_1 -\phi_2) \|_{\tilde  L^\infty_T B^\frac{d}{2}}
     \Big(
     \|\phi_1 \|_{\tilde L^\infty_T B^\frac{d}{2}}
     +
     \|\phi_2 \|_{\tilde L^\infty_T B^\frac{d}{2}}
     \Big)
     \cdot \\ \cdot 
     \Big(
     \|\Delta \phi_1 \|_{L^1_T B^\frac{d}{2}}+
     \|\Delta \phi_2 \|_{L^1_T B^\frac{d}{2}}
     \Big)
      + 
     \Big(
     1+    
     \Big\| \frac{\delta \theta_2}{\bar{\theta}+ \delta \theta_2} 
     \Big\|_{\tilde L^\infty_T B^\frac{d}{2}}
     \Big)
     \|   (\phi_1 -\phi_2) \|_{\tilde  L^\infty_T B^\frac{d}{2}}
     \Big(
     \|\nabla \phi_1 \|_{\tilde L^2_T B^\frac{d}{2}}^2
     +\\ +
     \|\nabla \phi_2 \|_{\tilde L^2_T B^\frac{d}{2}}^2
     \Big)
      + 
    \Big(
     1+    
     \Big\| \frac{\delta \theta_2}{\bar{\theta}+ \delta \theta_2} 
     \Big\|_{\tilde L^\infty_T B^\frac{d}{2}}
     \Big)
     \|   \nabla (\phi_1 -\phi_2) \|_{\tilde  L^2_T B^\frac{d}{2}}
     \Big(
     \|\phi_1           \|_{\tilde L^\infty_T B^\frac{d}{2}}
     +\\ +
     \|\phi_2           \|_{\tilde L^\infty_T B^\frac{d}{2}}
     \Big)
     \Big(
     \|\nabla \phi_1    \|_{\tilde L^2_T B^\frac{d}{2}}+
     \|\nabla \phi_2    \|_{\tilde L^2_T B^\frac{d}{2}}
     \Big)
     \Bigg\}.
\end{align*}
We can further apply Lemma \ref{lemma:action_of_smooth_function_on_B} to gather
\begin{align*}
    \mathcal{I}\mathcal{I}_2
    \leq 
    \frac{C}{\bar{\theta}}
    \max
    \Big\{
        1,\frac{1}{\epsilon}
    \Big\}
    \big(1+ \|\phi_0 \|_{B^{\frac{d}{2}}}\big)^4
   \|(\delta \phi_1,\delta \theta_1) -(\delta \phi_2,\delta \theta_2)\|_{\mathcal{K}}.
\end{align*}
Coupling the estimates of $\mathcal{I}\mathcal{I}_1$ and $\mathcal{I}\mathcal{I}_2$ we obtain the second inequality of \eqref{est-prop-final-I-II-III-new}. \\
This concludes the proof of the proposition.

\end{proof}

\section{Toolbox of Harmonic Analysis}\label{section:Besov-spaces}
\noindent 
In this section we briefly recall some basics about the function spaces,  which we exploited in the previous analysis. A more detailed and exhaustive overview can be found in \cite{BCD}.

\smallskip
\noindent
First of all, let us introduce the so called “Littlewood-Paley decomposition”, based on a non- homogeneous dyadic partition of unity with respect to the Fourier variable. To set up the dyadic partition, we first fix a smooth radial function $\chi$ supported on the ball $B(0,4/3)\subset \mathbb{R}^d $, equal to $1$ in $B(0,3/4)$ and such that $r\to \chi(r\mathbf{e})$ is nonincreasing over $\mathbb{R}_+$ for all unitary vectors $\mathbf{e}\in\mathbb{R}^d$. Set
$\varphi\left(\xi\right)=\chi\left(\xi/2\right)-\chi\left(\xi\right)$ and
$\varphi_q(\xi):=\varphi(2^{-q}\xi)$ for all $q\geq 0$, $\xi \in \mathbb{R}^d$.
%
Denoting by $\tilde{\mathbb{N}}= \mathbb{N}\cup\{-1\}$, the (nonhomogeneous) dyadic blocks $(\Dd_q)_{q\in\tilde{\mathbb{N}}}$ are defined by
$$
	\Dd_q f = \mathfrak{F}^{-1}( \varphi_q  \mathfrak{F}f ),\qquad q = -1,0, \dots
$$ 
where $\mathfrak{F}$ stands for the Fourier transform in $\mathbb{R}^d$. 

\begin{definition}\label{def:Besov-spaces}
   Let $s\in \mathbb{R}$ and $1\leq p,\,r\leq \infty$. The Besov space $B_{p,r}^s=B_{p,r}^s(\mathbb{R}^d)$ consists of all distribution $f$ such that
   \begin{equation*}
       \| f \|_{B_{p,r}^s}:=
       \Big\| \big( 2^{qs}\| \Dd_q f \|_{L^p} \big)_{q\in \tilde{\mathbb{N}}} \Big\|_{\ell^r(\tilde{\mathbb{N}})}
       <\infty.
   \end{equation*}
\end{definition}
\noindent
In our analysis, we fix $B^{s}:= B^{s}_{2,1}$ for the sake of a short notation. 
We recall the definition of the so-called Chemin-Lerner spaces, whose norms are determined by
\begin{equation*}
    \| \phi \|_{\tilde L^\rho(0,T; B_{p,r}^s)} = 
    \Big\| \Big(\| \Dd_q \phi \|_{\tilde L^\rho(0,T; L^p(\mathbb{R}^d))}\Big)_{q=-1,0,\dots} \Big\|_{\ell^r}.
\end{equation*}
In contrast with the classical space $ L^\rho(0,T; B_{p,r}^s)$, we here first take the $L^\rho$-norm in time at each component $\| \Dd_q \phi \|_{L^p(\mathbb{R}^d)}$, and then perform an $\ell^r$-norm in the index $q\in \tilde{\mathbb{N}}$. 

\noindent We shall remark that If $\rho = r$ then the two function spaces coincide, namely  $ \tilde L^r(0,T; B_{p,r}^s)= L^r(0,T; B_{p,r}^s)$. This fact has been extensively used in the previous sections, when $r= \rho = 1$ (and $p=2$).

\smallskip
\noindent
We  now  state the following lemma, which addresses some a-priori estimates for the linear operator related to the phase field $\phi$ within Chemin-Lerner spaces.
\begin{lemma}\label{lemma:a-priori-estimates-for-phi-appx}
    For any viscosity $\nu>0$, any damping $\alpha\geq 0$, any function $g \in L^1(0,T; B^{d/2})$ and any initial value $\varphi_0 \in B^{d/2}$, the following Cauchy problem in $(0,T)\times \mathbb{R}^d$ admits a unique solution $\varphi$ in $\tilde L^\infty(0,T; B^{d/2})\cap L^1(0,T;B^{d/2+4})$, with $\alpha \Delta \varphi \in \tilde L^\infty(0,T; B^{d/2})$:
    \begin{equation*}
        \partial_t \varphi  - \alpha \Delta \partial_t  \varphi + \nu \Delta^2  \varphi  = g,
        \qquad 
        \varphi_{|t=0} = \varphi_0.
    \end{equation*}
    Furthermore, there exists a constant $C>0$,  which depends only on the dimension $d\geq 1$, such that
    \begin{equation*}
    \begin{aligned}
        \|  \varphi             \|_{\tilde L^\infty(0,T; B^{\frac{d}{2}})} 
        &\leq 
        C 
        \Big(
            \| \varphi_ 0   \|_{B^{\frac{d}{2}}} + 
            \| g            \|_{L^1(0,T;B^{\frac{d}{2}})}
        \Big),\\
        \alpha 
        \|  \Delta \varphi  \|_{\tilde L^\infty(0,T; B^{\frac{d}{2}})} 
        &\leq 
        C 
        \Big(
            \alpha 
            \| \Delta \varphi_ 0   \|_{B^{\frac{d}{2}}} + 
            \| g            \|_{L^1(0,T;B^{\frac{d}{2}})}
        \Big),\\  
        \nu 
        \|  \Delta^2 \varphi            \|_{L^1(0,T;B^{\frac{d}{2}})}+
        \| \partial_t  \varphi          \|_{L^1(0,T;B^{\frac{d}{2}})}
        &\leq 
        C\Big(
            \| \varphi_ 0   \|_{B^{\frac{d}{2}}}+
            \alpha 
            \|  \Delta \varphi_0  \|_{B^{\frac{d}{2}}} + 
            \| g         \|_{L^1(0,T;B^{\frac{d}{2}})}
        \Big).
    \end{aligned}
    \end{equation*}
    Finally, if $\alpha>0$ and the function $g\in \tilde L^2(0,T; B^{d/2-1})$ is of the form $g = \Delta G$, we eventually have that
    \begin{equation*}
        \| \partial_t  \varphi \|_{\tilde L^2(0,T;B^{\frac{d}{2}})} + 
        \sqrt \alpha \| \partial_t \nabla \varphi \|_{\tilde L^2(0,T;B^{\frac{d}{2}})}
        \leq 
        C
        \Big(
             \sqrt{\nu} \| \Delta \phi_0 \|_{B^{\frac{d}{2}}}+
            \frac{1}{\sqrt{\alpha}}
                       \| g \|_{ \tilde L^2(0,T;B^{\frac{d}{2}-1})}
                    \Big).
    \end{equation*}
\end{lemma}
\begin{proof}
    We split the solution $ \varphi = \varphi_L + \delta \varphi$ of the considered equation into its homogeneous and inhomogeneous components:
    \begin{equation*}
    \begin{alignedat}{8}
        \partial_t \varphi_L  - \alpha \Delta \partial_t  \varphi_L + \nu \Delta^2  \varphi_L  
        &= 0\quad (0,T)\times \mathbb{R}^d,
        \qquad 
        \varphi_{L|t=0} &&= \varphi_0\quad &&&&\mathbb{R}^d,\\
        \partial_t \delta  \varphi  - \alpha \Delta \partial_t \delta  \varphi + \nu \Delta^2  \delta  \varphi &= g\quad (0,T)\times \mathbb{R}^d,
        \qquad 
        \delta \varphi_{|t=0} &&= 0\quad &&&&\mathbb{R}^d.
    \end{alignedat}    
    \end{equation*}
    The linear component $\varphi_L$ is determined by the Fourier multiplier 
    $\hat{\varphi}_L(t,\xi) = e^{-t|\xi|^4/(1+\alpha|\xi|^2)} \hat{\varphi}_0(\xi)$, therefore $\varphi_L$ satisfies each inequality of the Lemma, when imposing $g\equiv 0$. 
    
    \noindent We next focus on the non-homogeneous component $\delta \varphi$. We first apply the dyadic block $\Dd_q$ to the linear equation and we remark that 
    $\Dd_q \delta \varphi$ is solution of
    \begin{equation}\label{eq-phi-Dq-in-lemma}
        \partial_t \Dd_q \delta \varphi - \alpha \Delta \Dd_q \delta \varphi + \Delta^2 \Dd_q \delta \varphi = \Dd_q g
        \quad (0,T)\times \mathbb{R}^d,\qquad 
        \Dd_q \delta \varphi_{|t=0} = 0\quad \mathbb{R}^2.
    \end{equation}
    We hence multiply the equation by $\Dd_q \phi$ and integrate the result on $\mathbb{R}^d$, to gather 
    the following $L^2$-estimate:
    \begin{equation*}
        \frac{d}{dt}\frac 12 \| \Dd_q \delta \varphi(t) \|_{L^2}^2 + \alpha\frac{d}{dt}\frac 12 \| \nabla  \Dd_q \delta \varphi(t) \|_{L^2}^2 + 
        \| \Delta\Dd_q \delta \varphi(t) \|_{L^2}^2 \leq C\| \Dd_q g(t) \|_{L^2} \| \Dd_q \delta \varphi(t) \|_{L^2},
    \end{equation*}
    where the time derivative is intended in Sobolev sense, namely $t\in [0,T]\to \| \Dd_q \delta \varphi(t) \|_{L^2}^2 + \alpha \| \nabla \Dd_q \delta \varphi (t) \|_{L^2}^2$ belongs to $W^{1,1}([0,T])\hookrightarrow C([0,T])$. Next, we integrate in time and we take the supremum over $t\in [0,T]$. We thus recover that
    \begin{equation}\label{lemma:est-linear-eq-on-phi-est10}
    \begin{aligned}
        \frac{1}{2} \| \Dd_q\delta \varphi \|_{L^\infty(0,T;L^2)}^2 + 
        \frac{\alpha }{2} \| \Dd_q \nabla \delta \varphi \|_{L^\infty(0,T;L^2)}^2
        &+\int_0^T \| \Delta \Dd_q \delta \varphi(s) \|_{L^2}^2d s  \\
        &\leq 
        C \int_0^T \| \Dd_q g(s) \|_{L^2}ds\, \| \Dd_q \delta \varphi \|_{L^\infty(0,T;L^2)}
    \end{aligned}
    \end{equation}
    This last relation tells us that the $L^\infty(0,T;L^2)$-norm of   $\Dd_q\delta \varphi$ is bounded by the $L^1(0,T;L^2)$-norm of $\Dd_q g$, up to a constant which depends only on the dimension. Taking the $\ell^1$-norm in $q\in \tilde{\mathbb{N}}$ we hence recover the first inequality of the Lemma, when $\varphi_0 \equiv 0$. The second inequality is given by the following Bernstein relation, which holds for any $q\in \mathbb{N}$,
    \begin{equation*}
    \begin{alignedat}{8}
        \| \Dd_q \delta \varphi(t) \|_{L^2} &\leq \tilde C2^{-q} \| \Dd_q \nabla \delta \varphi(t) \|_{L^2},\qquad
        \| \Dd_q \Delta \delta \varphi(t)\|_{L^2}\leq \tilde C2^{q} \| \Dd_q \nabla \delta \varphi(t) \|_{L^2},
   \end{alignedat}
   \end{equation*}
   where $\tilde C>0$ depends only on the dimension $d\geq 1$. Indeed, by the inequality \eqref{lemma:est-linear-eq-on-phi-est10}, we have that
   \begin{equation*}
        \alpha  \| \nabla \Dd_q \delta \varphi \|_{L^\infty(0,T;L^2)}^2\leq 
        C 
        \int_0^T \| \Dd_q g(s) \|_{L^2}ds\,2^{-q} \| \Dd_q \nabla \varphi \|_{L^\infty(0,T;L^2)}.
    \end{equation*}
    This implies that, for any $q\in \mathbb{N}$,
    \begin{equation*}
        \alpha  \| \Delta \Dd_q \delta \varphi \|_{L^\infty(0,T;L^2)} 
        \leq \tilde C\alpha 2^q \| \nabla \Dd_q \delta \varphi \|_{L^\infty(0,T;L^2)} 
        \leq C\int_0^T \| \Dd_q g(s) \|_{L^2}ds.
    \end{equation*}
    We remark that this last inequality is valid also for $q=-1$, since for this constant we handle only the low frequencies of $\delta \varphi$. By taking the sum as $q\in  \tilde{\mathbb{N}}$, we gather that
    \begin{equation*} 
        \alpha \| \Delta \delta  \varphi \|_{L^\infty(0,T;B^s)} 
        \leq C \|  g \|_{L^1(0,T;B^s)}. 
    \end{equation*}
     The remaining estimates on $\partial_t \phi$ and $\Delta^2$ follow from similar arguments, multiplying the equation \eqref{eq-phi-Dq-in-lemma} by $\partial_t \Dd_q\phi$.

\end{proof}
\noindent 
The next result guarantees similar estimates for the heat flow. The proof can be achieved with a similar argument as the one used to prove Lemma \ref{lemma:a-priori-estimates-for-phi}.
\begin{lemma}\label{lemma:a-priori-estimates-for-theta-appx}

For any function $h\in L^1(0,T; B^{\frac{d}{2}})$ with $s \in \mathbb{R}$ and any initial value $\theta_0\in B^{\frac{d}{2}}$, the following Cauchy problem in $(0,T)\times \mathbb{R}^d$ admits a unique solution $ \theta \in \tilde L^\infty(0,T; B^{s})\cap L^1(0,T; B^{s+2})$
\begin{equation*}
    \kappa_B\partial_t   \theta - \kappa \Delta   \theta = h,\qquad 
      \theta_{|t=0} = \theta_0.
\end{equation*}
Furthermore, there exists a constant $C>0$ such that
\begin{equation*}
    \kappa_B\|  \theta \|_{\tilde L^\infty(0,T;B^{s})} + \kappa \| \Delta  \theta \|_{  L^1(0,T;B^{s})} + 
    \kappa_B\| \partial_t \theta \|_{ L^1(0,T;B^{s})}
    \leq C\Big\{ \kappa_B\| \theta_0 \|_{B^{s}}+ \| h \|_{L^1(0,T;B^{s})}\Big\}.
\end{equation*}

\end{lemma}
\noindent
Finally, We state the following result about the effect of left composition by smooth functions on Besov spaces $B_{p,r}^s$ (cf.~Theorem 2.87 in \cite{BCD}).
\begin{lemma}\label{lemma:action_of_smooth_function_on_B}
    Let $h$ be a smooth function vanishing at $0$, $s$ be a positive real number, and $(p,r)\in [1,\infty]^2$. If $u$ belongs to $B_{p,r}^s\cap L^\infty$, then so does $h\circ u$ and we have
    \begin{equation*}
        \| h\circ u \|_{B_{p,r}^s} \leq C(s, h, \| u \|_{L^\infty})\| u \|_{B_{p,r}^s}.
    \end{equation*}
\end{lemma}
\noindent 
The constant $C(s, h, \| u \|_{L^\infty})$ continuously depends on $\| u \|_{L^\infty}$ and it is increasing in this norm. An explicit formulation of the constant $C(s, h, \| u \|_{L^\infty})$ can be computed from the proof of Theorem 2.87 in  \cite{BCD}:
\begin{equation}\label{ineq:composition-besov}
    \|h\circ u \|_{B_{p,r}^s}   \leq 
    C_s 
    \Big( 
        \max_{l \in \{0,\dots, [s]+1\}} \| u \|_{L^\infty}^l \sup_{|x| \leq C\| u \|_{L^\infty}} |h^{(l+1)}(x)|
    \Big)
    \| u \|_{B_{p,r}^s},
\end{equation}
    where $l$ is an index for the derivatives of the function $h$,
$C_s>0$ depends only on the regularity $s>0$ and the dimension $d$. 

\section{Conclusion and relation with previous models}\label{sec:conclusion}

\noindent 
The scheme that we have developed shows how specific physical quantities, together with their kinematics and the law of thermodynamics, determine the overall system of equations, which is consistent with the main law of thermodynamics. Specifically, these quantities correspond to
\begin{itemize}
    \item the free energy $\psi(\phi, \nabla \phi, \theta)$ of the media,
    \item the kinematics of $\phi$ (which is fixed, being a conserved quantity)
    and the transport of $\theta$ (which varies and changes the final form of the system),
    \item the rate of dissipation $\mathcal{D}$  in terms of the
     effective microscopic velocity (which in this paper takes the form of a Darcy's law),
    \item the heat flux ${\bf q}$, which in this manuscript we have consider by the Fourier's law.
\end{itemize}
The first two quantities determine the conservative forces for the phase field, while the rate of dissipation and Onsager's principle unlock the dissipative forces with the equation for $\phi$. 
The equation for the temperature $\theta$ is then determined by the second law of thermodynamics for the local entropy $s(\phi, \theta)$, whose entropy production $\theta \Delta^*(\phi, \theta)$ can be computed from the first law of thermodynamics for the conservation of the internal energy.

\smallskip 
\noindent 
We next briefly exploit our modelling technique when computing certain systems, that were previously proposed in literature.   We first mention that our proposed technique can not derive the equations of Caginalp in \cite{Caginalp90}. The author indeed coupled the Cahn-Hilliard equation with an heat balance of the following  form  
\begin{equation*}
\begin{aligned}
    \tau \partial_t \phi = -\xi^2 \Delta \frac{\delta W}{\delta \phi} \quad  \text{with}\quad 
    W(\phi, \nabla \phi, \theta) = 
    \xi^2 \frac{|\nabla \phi|^2}{2} + 
    \frac{(\phi^2-1)^2}{8a}- 2\theta \phi
    \quad\text{and}\quad 
    \partial_t \theta + \frac{l}{2}\partial_t \phi = k \Delta \theta,
\end{aligned}
\end{equation*}
for suitable parameters $\tau,\, \xi,\, a,\, l,\, k$. In particular, this set of equations does not preserve the total energy $e = W+ \theta s$ of the system, given by
\begin{equation*}
    e = W + \theta s = \xi^2 \frac{|\nabla \phi|^2}{2} + 
    \frac{(\phi^2-1)^2}{8a},
\end{equation*}
while our technique ensures that this quantity is conserved by the first law of thermodynamics. 
Our overview shall therefore begin with a special case of the work of Alt and Paw{\l}ow in \cite{alt1992mathematical}. The simplest form of their model reads as follows
\begin{equation}\label{eq:alt-pawlow}
    \partial_t \phi - m\Delta \Big(\frac{\tilde{\mu}}{\theta}\Big)= 0,\qquad \partial_t e +\dv \Big(\kappa (\theta)\nabla \Big(\frac{1}{\theta}\Big)\Big)  = 0
\end{equation}
Here $e$ stands for the internal energy  (which in \cite{alt1992mathematical} is denoted by $E$), i.e. $e = \psi + \theta s$, where we recall that $\psi = \psi(\phi,\nabla \phi, \theta)$ is the free energy and $s=-\partial_\theta \psi$ is the local entropy. 
With an abuse of notation, we have here denoted by $\tilde{\mu}$ their chemical potential, which is defined by the variation of a rescaled energy density, i.e. $\tilde{\mu} = \partial_\phi \psi - \theta \dv ( \partial_{\nabla \phi} \psi/\theta)$.
Furthermore, their heat flux is defined as 
${\bf q}(\theta) = -\kappa(\theta)\nabla (1/\theta)$. 

\noindent 
We aim to show that our proposed approach derives a slight variation of their model.  
We mention that the choice of the form \eqref{eq:alt-pawlow} is for the sake of a compact presentation, we however claim that the general setting in \cite{alt1992mathematical} can also be derived by this formalism (with some minor adjustments). 
To proceed, we need first to exploit several assumptions. 

\noindent
At the level of the least action principle in Section \ref{sec:least-action-principle}, the temperature should be considered related only to the background, i.e. it does not evolve with the flow $\theta(t,x(t,X)) = \theta_0(X)$. 

\noindent
Furthermore the action in Section \ref{sec:least-action-principle} and the dissipation of Section \ref{sec:onsager} shall be considered as
\begin{equation}\label{eq:final-sec-action-dissipation}
\begin{aligned}
    \mathcal{A}(x) &= 
    \int_{\Omega} \frac{\psi(\phi, \nabla \phi, \theta)}{\theta}dx := 
    \int_{\Omega_0}\frac{\psi(\mathcal{G}(t,X))}{\theta_0(X)}\det F(t,X) dX,\qquad
    \mathcal{D}(u) &= \frac{1}{m}
    \int_{\Omega}  \phi^2 |u|^2dx,
\end{aligned}
\end{equation}
where $\mathcal{G}(t,X)$ has been defined in \eqref{action-in-lagrangian}.
Here the Action $\mathcal{A}$ simply relates to the recasted free energy $\psi(\phi, \nabla \phi, \theta)$, while the dissipation $\mathcal{D}$ is still of Darcy type. 
Then applying Onsager's principle of  Section \ref{sec:least-action-principle}, one gathers the relation
\begin{equation*}
    \phi \nabla \left( \frac{\tilde{\mu}}{\theta} \right) = {\bf f}_{\rm cons} = -{\bf f}_{\rm diss} = -\frac{1}{m}\phi^2 u, 
\end{equation*}
which implies the first equation of \eqref{eq:alt-pawlow}, thanks to the continuity equation $\partial_t \phi + \dv ( u \,\phi) = 0$. Furthermore, an analogous computation as in \eqref{eq:derivation-of-energy-eq}, leads to the following relation for the entropy production:
\begin{equation*}
     \theta\Delta^* = \phi^2 |u|^2 + \theta {\bf q}\cdot \nabla  \Big( \frac{1}{\theta} \Big)
     = \Big| \nabla \Big( \frac{\tilde{\mu}}{\theta} \Big)\Big|^2
     + \theta \kappa(\theta)\Big| \nabla 
    \Big( 
            \frac{1}{\theta} 
    \Big)\Big|^2.
\end{equation*}
Therefore, our final equation for the energy eventually reads
\begin{equation*}
    \partial_t e = -\dv {\bf q} - \dv \Big(u\phi \mu  - \partial_{\nabla \phi} \psi \partial_t \phi \Big) \\
    = -\dv \Big(\kappa (\theta)\nabla \Big(\frac{1}{\theta}\Big)\Big)   - \dv \Big(u\phi \mu  - \partial_{\nabla \phi} \psi \partial_t \phi \Big),
\end{equation*}
which differs from the one proposed by Alt and Pawlow in \eqref{eq:alt-pawlow}, because of the term $ \dv \big(u\phi \mu  - \partial_{\nabla \phi} \psi \partial_t \phi \big)$. We claim that this term shall be interpreted as work produced by the system. One shall remark however that it vanishes when integrating in space. 

\smallskip
\noindent 
A second model we intend to address is a special case of the system proposed by Miranville and Schimperna in \cite{MR2151731} (cf.~also \cite{Marveggio21} for a mathematical analysis):
\begin{equation}\label{Miranville-Schimperna-eq}
\begin{alignedat}{4}
    \partial_t \phi &= - \dv h ,\quad &&\text{with}\quad h = 
    - \nabla \left( \frac{\mu}{\theta}\right),\\
    \partial_t e  &= -\dv {\bf q} -\dv (\mu h - \partial_{\nabla \phi}\psi \partial_t \phi),\quad &&\text{with}\quad  -
    {\bf q}+ \mu h = \kappa(\theta) 
    \nabla \left( \frac{1}{\theta}\right).
\end{alignedat}
\end{equation}
We shall remark that here the chemical potential $\mu = \delta \psi/\delta \phi$ is not rescaled by $1/\theta$. For this system, the action $\mathcal{A}$ shall be taken as in \eqref{eq:final-sec-action-dissipation}, while the dissipation shall be reformulated as 
\begin{equation*}
    \mathcal{D}(u) = 
    \int_{\Omega} \Big|\phi u - \partial_{\nabla \phi} \psi\cdot  \nabla \Big( \frac{1}{\theta}\Big) \Big|^2.
\end{equation*}
Applying Onsager's principle of  Section \ref{sec:least-action-principle}, one gathers the relation
\begin{equation*}
    \phi\nabla \Big(   \frac{\mu}{\theta}  + \partial_{\nabla \phi} \psi \cdot \nabla \Big( \frac{1}{\theta} \Big)\Big)= {\bf f}_{\rm cons} = -{\bf f}_{\rm diss} = -\phi^2 u+ \phi \nabla\Big(  \partial_{\nabla \phi} \psi \cdot \nabla \Big( \frac{1}{\theta} \Big)\Big).
\end{equation*}
Coupling this relation with the continuity equation for $\phi$ leads to  $\phi u = h$ and the first equation of \eqref{Miranville-Schimperna-eq}. Moreover, an analogous computation as in \eqref{eq:derivation-of-energy-eq}, leads to the following identity for the entropy production:
\begin{equation*}
     \theta\Delta^* = \phi^2 |u|^2 + \theta {\bf q}\cdot \nabla  \Big( \frac{1}{\theta} \Big)
     = \Big| \nabla \Big( \frac{\mu}{\theta} \Big)\Big|^2
     + \theta \kappa(\theta)\Big| \nabla \Big( \frac{1}{\theta} \Big)\Big|^2.
\end{equation*}
Therefore \eqref{eq:derivation-of-energy-eq} corresponds exactly to the energy equation of \eqref{Miranville-Schimperna-eq}.

\smallskip
\noindent  
In conclusion, in this paper we have derived two thermodynamically consistent models of the Cahn-Hilliard equations with variable temperature. Our derivation is based on several mechanical and thermodynamical principles. Furthermore, we have highlighted how certain assumptions on the transport of the temperature change the final structure of our models. Our analysis about local-in-time classical solutions selects one model above the others, in which the transport of the temperature is determined by the background and not by the effective velocity. This last section has also shown how certain models that have been previously proposed in literature can also be derived by the formalism of this article.

\subsection*{Data Availability Statement} 
Data sharing is not applicable to this article, since no datasets were generated or analysed during the current study.
 
\subsection*{Acknowledgment} 
The authors were partially supported by NSF grant DMS-1950868, and the United States–Israel Binational Science Foundation (BSF) grant {\#}2024246, as well as the  Deutsche Forschungsgemeinschaft (DFG, German Research Foundation), project number 391682204. 

\printbibliography
\end{document}